\documentclass[oneside,american,english]{amsart}
\usepackage{mathptmx}

\usepackage[T1]{fontenc}
\usepackage[latin9]{inputenc}
\usepackage{color}
\usepackage{babel}
\usepackage{amstext}
\usepackage{amsthm}
\usepackage{amssymb}
\usepackage{undertilde}
\usepackage[unicode=true,pdfusetitle,
 bookmarks=true,bookmarksnumbered=false,bookmarksopen=false,
 breaklinks=false,pdfborder={0 0 1},backref=false,colorlinks=false]
 {hyperref}
\usepackage{breakurl}

\makeatletter
\numberwithin{equation}{section}
\numberwithin{figure}{section}
\theoremstyle{plain}
\newtheorem{thm}{\protect\theoremname}
  \theoremstyle{remark}
  \newtheorem{rem}[thm]{\protect\remarkname}
  \theoremstyle{plain}
  \newtheorem{lem}[thm]{\protect\lemmaname}
  \theoremstyle{definition}
  \newtheorem{defn}[thm]{\protect\definitionname}
  \theoremstyle{plain}
  \newtheorem{cor}[thm]{\protect\corollaryname}
  \theoremstyle{plain}
  \newtheorem{prop}[thm]{\protect\propositionname}

\makeatother

  \addto\captionsamerican{\renewcommand{\corollaryname}{Corollary}}
  \addto\captionsamerican{\renewcommand{\definitionname}{Definition}}
  \addto\captionsamerican{\renewcommand{\lemmaname}{Lemma}}
  \addto\captionsamerican{\renewcommand{\propositionname}{Proposition}}
  \addto\captionsamerican{\renewcommand{\remarkname}{Remark}}
  \addto\captionsamerican{\renewcommand{\theoremname}{Theorem}}
  \addto\captionsenglish{\renewcommand{\corollaryname}{Corollary}}
  \addto\captionsenglish{\renewcommand{\definitionname}{Definition}}
  \addto\captionsenglish{\renewcommand{\lemmaname}{Lemma}}
  \addto\captionsenglish{\renewcommand{\propositionname}{Proposition}}
  \addto\captionsenglish{\renewcommand{\remarkname}{Remark}}
  \addto\captionsenglish{\renewcommand{\theoremname}{Theorem}}
  \providecommand{\corollaryname}{Corollary}
  \providecommand{\definitionname}{Definition}
  \providecommand{\lemmaname}{Lemma}
  \providecommand{\propositionname}{Proposition}
  \providecommand{\remarkname}{Remark}
\providecommand{\theoremname}{Theorem}

\begin{document}

\title{{\normalsize{}Steep Points of Gaussian Free Fields in Any Dimension}}

\author{Linan Chen}
\begin{abstract}
This work aims to extend the existing results on the Hausdorff dimension
of the classical thick point sets of a Gaussian free field (GFF) to
a more general class of exceptional sets. We adopt the circle or sphere
averaging regularization to treat a singular GFF in any dimension,
and introduce the notion of ``$f-$steep point'' of the GFF for
certain test function $f$. Roughly speaking, the $f-$steep points
of a generic element of the GFF are locations where, when weighted
by the function $f$, the ``steepness'', or in other words, the
``rate of change'' of the regularized field element becomes unusually
large. Different choices of $f$ lead to the study of various exceptional
behaviors of the GFF. We investigate the Hausdorff dimension of the
set consisting of $f-$steep points, from which we can recover the
existing results on thick point sets for both log-correlated and polynomial-correlated
GFFs, and also obtain new results on exceptional sets that, to our
best knowledge, have not been previously studied. Our method is inspired
by the one used to study the thick point sets of the classical 2D
log-correlated GFF.
\end{abstract}

\selectlanguage{american}%

\address{Department of Mathematics and Statistics, McGill University, 805
Sherbrooke Street West, Montréal, QC, H3A 0B9, Canada. }

\email{Email: linan.chen@mcgill.ca}



\thanks{The author is partially supported by NSERC Discovery Grant G241023.}
\maketitle
\selectlanguage{english}%

\section{Introduction}

Gaussian Free Field (GFF) has played an essential role in many recent
achievements in quantum physics and statistical mechanics. Although
originated in physics, the mathematical study of GFFs has been a fast
developing field of probability theory, generating fruitful results
on problems arising from discrete math, analysis, geometry and other
subjects. Heuristically speaking, GFFs are analogues of the Brownian
motion with multidimensional time parameters. Just as the Brownian
motion can be viewed naturally as a random univariate function, GFFs
can be interpreted as random multivariate functions or generalized
functions. Also, just as the graph of the Brownian motion naturally
models a random curve, graphs of GFFs are considered as promising
candidates for modeling random surfaces or random manifolds, which
ultimately lead to the study of random geometry. On one hand, GFFs
have been applied to construct random geometric objects such as random
measures, for example, the Liouville Quantum Gravity measure which
we will mention briefly below. On the other hand, the study of geometric
properties of a GFF itself gives rise to many interesting problems,
most of which remain open to date. The main reason that such problems
are challenging, at least for a typical GFF concerned in our work,
is that a generic element of the GFF is only a tempered distribution
which may not be point-wisely defined, to which we refer as the \emph{singularity}
of the GFF. To tackle this kind of singularity, it is natural to consider
a GFF in the discrete setting, for example, on a discrete lattice,
in which case the GFF will be defined on every vertex. A rich literature
has been established on the geometry of discrete GFFs. For instance,
the distribution of extrema and near-extrema of a discrete GFF has
been extensively studied\emph{ }(e.g.,\emph{ }\cite{DingZeitouni,DingRoyZeitouni,ChatterjeeDemboDing}).
However, for a GFF in the continuum setting, the notion of ``extrema''
is not applicable due to the lack of point-wise values of the field.
To overcome this issue, one needs to apply a procedure, known as a
\emph{regularization} in physics literature, to approximate point-wise
values of the continuum GFF. Various regularization procedures have
long been considered in the study of related problems. Below we only
allude to two commonly used regularization procedures.

\textcolor{black}{The first one is based on the theory of }\textcolor{black}{\emph{Gaussian
Multiplicative Chaos}}\textcolor{black}{{} (GMC) introduced by Kahane
in his seminal work \cite{Kah}. The GMC theory enables one to define
in any dimension a random Borel measure which formally takes the form
``$e^{h\left(x\right)}dx$'', where $h$ is a generic element of
a log-correlated Gaussian random field, and $dx$ is the Lebesgue
measure. Such a measure, known as the }\textcolor{black}{\emph{Liouville
Quantum Gravity}}\textcolor{black}{{} (LQG) measure, is an important
object in quantum field theory. Kahane's work has led to the multi-fractal
analysis of the LQG measure by showing that such a measure is supported
on a Borel set where the regularized $h$ achieves ``unusually''
large values. Over the past decade, further results on the support
of the LQG measure and the geometry of log-correlated GFFs have been
established under the framework of GMC (e.g., \cite{BM,RV,RV11,RV13,JJRV,DSRV14}).
Besides, using the tool of GMC, the extreme values of the regularized
$h$ are also treated in \cite{Madaule}.}

Besides the GMC approach, one can also regularize a continuum GFF
by averaging the generic field element $h$ over some sufficiently
``nice'' Borel sets. Since convolution or integration is the natural
way to ``tame'' the singularity of a tempered distribution, such
an averaging procedure becomes a natural choice when it comes to the
study of the ``landscape'' of $h$. For example, a more recent breakthrough
in the study of quantum gravity was the work of Duplantier and Sheffield
(\cite{DS1}), which, based on the averages of $h$ over circles,
gave a rigorous construction of the LQG measure in 2D, and a rigorous
proof of the long celebrated \emph{Knizhnik-Polyakov-Zamolodchikov
formula}, in the context of linking the scaling dimension of the LQG
measure with that of the underlying Lebesgue measure. Along the way,
\cite{DS1} also derived the same property for the support of the
LQG measure as mentioned above, i.e., it is supported where the averaged
$h$ becomes unusually large. Meanwhile, also using circular averages
of $h$, Hu, Miller and Peres (\cite{HMP}) studied specifically the
points where the regularized $h$ is unusually large, introduced the
notion of ``thick point''\footnote{The term ``thick point'' is borrowed from the literature of stochastic
analysis. There it refers to the extremes of the occupation measure
of a stochastic process (see, e.g., \cite{DPRZ}). }, and determined the Hausdorff dimension of the set consisting of
thick points. Based on a sphere averaging regularization, some of
the results on the LQG measure were generalized to higher-even-dimensional
log-correlated GFFs by \cite{CJ}, and the study of thick points was
extended to four-dimensional log-correlated GFFs by \cite{CiprianiHazra13},
and then later to polynomial-correlated GFFs in any dimension by \cite{Chen_thick_point}. 

\subsection{A Brief Review of Thick Point}

Besides being the support of the LQG measure, thick point sets characterize
a basic aspect of the ``landscape'' of the GFFs, that is, where
the ``high peaks'' occur, so thick points are of importance to understanding
the geometry of the GFFs. The purpose of this article is to consolidate
the existing results on thick point sets for both log-correlated GFFs
and polynomial-correlated GFFs, and to extend our study to a more
general class of exceptional sets. We will begin with a brief (and
not exhaustive) review on what is known about thick point sets of
log-correlated or polynomial-correlated GFFs. 

\subsubsection{Thick Points of Log-Correlated GFFs }

Following the same notations as above, let $h$ be a generic element
of the GFF associated with the operator $\Delta$ on a bounded domain
$D\subseteq\mathbb{R}^{2}$ with the Dirichlet boundary condition.
Governed by the properties of the Green's function of $\Delta$ in
2D, such a GFF is log-correlated, and it is possible to make sense
of the circular average of $h$:
\[
\bar{h}_{t}\left(z\right):=\frac{1}{2\pi t}\int_{\partial B\left(z,t\right)}h\left(x\right)\sigma\left(dx\right)
\]
where $z\in D$, $\partial B\left(z,t\right)$ is the circle centered
at $z$ with radius $t$ and $\sigma\left(dx\right)$ is the length
measure along the circle. To get an approximation of ``$h\left(z\right)$'',
it is to our interest to study $\bar{h}_{t}\left(z\right)$ as $t\searrow0$.
For every $\gamma\geq0$, the set of \emph{$\gamma-$thick points}
of $h$ is defined in \cite{HMP} as\footnote{The definition of thick point presented here actually adopts a different
parametrization from the original version in \cite{HMP}.} 
\begin{equation}
T^{\gamma,h}:=\left\{ z\in D:\,\lim_{t\searrow0}\,\frac{\bar{h}_{t}\left(z\right)}{-\ln t}=\frac{\gamma}{\pi}\right\} .\label{eq:2D thick point}
\end{equation}
With $z$ fixed, the circular average process $\left\{ \bar{h}_{t}\left(z\right):\,t\in(0,1]\right\} $
has the same distribution as a Brownian motion $\left\{ B_{\tau}:\,\tau\geq0\right\} $
up to a deterministic time change $\tau=\tau\left(t\right)=\frac{-\ln t}{2\pi}$,
and as $t\searrow0$, $\bar{h}_{t}\left(z\right)$ behaves just like
$B_{\tau}$ as $\tau\nearrow\infty$. Then, for any given $z\in D$,
written in terms of the Brownian motion, the limit involved in (\ref{eq:2D thick point})
is equivalent to 
\[
\lim_{\tau\rightarrow\infty}\,\frac{B_{\tau}}{\tau}=2\gamma
\]
which occurs with probability zero for any $\gamma>0$. Therefore,
$\gamma-$thick points, so long as $\gamma>0$, are locations where
the field value is unusually large. The authors of \cite{HMP} prove
that, with probability one, if $\gamma>\sqrt{2\pi}$, then $T^{\gamma,h}=\emptyset$;
if $\gamma\in\left[0,\sqrt{2\pi}\right]$, then
\[
\dim_{\mathcal{H}}\left(T^{\gamma,h}\right)=2-\frac{\gamma^{2}}{\pi},
\]
 where ``$\dim_{\mathcal{H}}$'' refers to the Hausdorff dimension;
if $\gamma=0$, $z\in T^{\gamma,h}$ for almost every $z\in D$ under
the Lebesgue measure on $D$. 

\subsubsection{Thick Points of Polynomial-Correlated GFFs}

In $\mathbb{R}^{\nu}$ with $\nu\geq3$, if $\theta$ is a generic
element of the GFF associated with the operator\footnote{One can instead consider the GFF associated with $\Delta$ on a bounded
domain $D\subseteq\mathbb{R}^{\nu}$ equipped with the Dirichlet boundary
condition, and the same results as mentioned in this subsection will
hold. See Remark \ref{rem: about the choice of operator}.} $\left(I-\Delta\right)$ on $\mathbb{R}^{\nu}$, then $\theta$ is
more singular compared with the previous 2D log-correlated GFF element
$h$, because the Green's function in this case has a polynomial singularity
along the diagonal and the GFF is polynomial-correlated. Intuitively
speaking, compared with that of $h$, the graph of $\theta$ is ``rougher'',
and the higher the dimension $\nu$ is, the worse it becomes. But
no matter what the dimension is, it is always possible to average
$\theta$ over the codimension-1 sphere centered at any $x\in\mathbb{R}^{\nu}$
with radius $t>0$, and the spherical average, denoted by $\bar{\theta}_{t}\left(x\right)$,
approaches ``$\theta\left(x\right)$'' as $t\searrow0$ in the sense
of tempered distribution. In this setting, for $\gamma\geq0$, the
set of \emph{$\gamma-$thick points} of $\theta$ is defined in \cite{Chen_thick_point}
as 
\begin{equation}
T^{\gamma,\theta}:=\left\{ x\in\mathbb{R}^{\nu}:\,\limsup_{t\searrow0}\,\frac{\bar{\theta}_{t}\left(x\right)}{\sqrt{-G\left(t\right)\ln t}}\geq\sqrt{2\nu\gamma}\right\} \label{eq:thick point def}
\end{equation}
where $G\left(t\right):=\mathbb{E}\left[\left(\bar{\theta}_{t}\left(x\right)\right)^{2}\right]$
for every $t>0$. In a similar spirit as (\ref{eq:2D thick point}),
if $\gamma>0$, then a $\gamma-$thick point is a location where $\theta$
is unusually large. It is established in \cite{Chen_thick_point}
that, with probability one, if $\gamma>1$, then $T^{\gamma,\theta}=\emptyset$;
if $\gamma\in\left[0,1\right]$, then
\[
\dim_{\mathcal{H}}\left(T^{\gamma,\theta}\right)=\nu\left(1-\gamma\right).
\]
Clearly (\ref{eq:thick point def}) is not the most straightforward
analogue of (\ref{eq:2D thick point}), since ``$\limsup$'' is
considered instead of ``$\lim$'', but it turns out to be a more
suitable choice for the definition of thick point of the polynomial-correlated
GFF, because, with probability one, the ``perfect'' $\gamma-$thick
point, i.e., $x$ such that 
\[
\lim_{t\searrow0}\,\frac{\bar{\theta}_{t}\left(x\right)}{\sqrt{-G\left(t\right)\ln t}}=\sqrt{2\nu\gamma},
\]
does not exist. \cite{Chen_thick_point} also investigates the set
of \emph{sequential $\gamma-$thick points }given by 
\begin{equation}
ST^{\gamma,\theta}:=\left\{ x\in\mathbb{R}^{\nu}:\,\lim_{m\nearrow\infty}\,\frac{\bar{\theta}_{r_{m}}\left(x\right)}{\sqrt{-G\left(r_{m}\right)\ln r_{m}}}=\sqrt{2\nu\gamma}\right\} ,\label{eq:thick point along sequence}
\end{equation}
where $\left\{ r_{m}:m\geq1\right\} \subseteq(0,1]$ is a sequence
that $r_{m}\searrow0$ sufficiently fast as $m\nearrow\infty$, and
proves that, with probability one, if $\gamma>1$, then $ST^{\gamma,\theta}=\emptyset$;
if $\gamma\in\left[0,1\right]$, then
\[
\dim_{\mathcal{H}}\left(ST^{\gamma,\theta}\right)=\nu\left(1-\gamma\right).
\]

Compared with the case in the log-correlated setting, the higher-level
of singularity of $\theta$ makes its thick points ``rarer'' and
hence harder to find. In fact, the most involved part of the work
in \cite{Chen_thick_point} is to establish a lower bound for $\dim_{\mathcal{H}}\left(T^{\gamma,\theta}\right)$
and $\dim_{\mathcal{H}}\left(ST^{\gamma,\theta}\right)$. One would
expect that, for most problems related to the geometry of GFFs, it
is non-trivial to extend the study from the log-correlated setting
to the polynomial-correlated setting, due to the challenge posed by
the higher order of singularity in the latter case. 

\subsection{An Outline of the Article}

Generally speaking, in this article, instead of focusing on the regularized
GFF element ``$\bar{h}_{t}\left(z\right)$'' or ``$\bar{\theta}_{t}\left(x\right)$''
itself, we consider the integral of some test function $f\left(t\right)$,
integrated against the ``increment'' of the regularized GFF; instead
of focusing on how large the value of ``$\bar{h}_{t}\left(z\right)$''
or ``$\bar{\theta}_{t}\left(x\right)$'' becomes as $t\searrow0$,
we study how large the value of the concerned integral becomes when
$t$ is small, which reflects the ``steepness'' or the ``rate of
change'' of the regularized GFF with respect to $t$. Although setting
out to investigating a slightly different perspective of the ``landscape''
of the GFF, our work follows a similar general strategy as that in
\cite{HMP} and \cite{Chen_thick_point}. In $\mathsection2$, we
interpret GFFs in the framework of \emph{Abstract Wiener Space} and
adopt the regularization based on circular or spherical averages.
We also borrow, from the mentioned references, the results on the
continuity property of the regularized GFF to study the continuity
property of the integral of $f$ against the regularized GFF. In $\mathsection3$
we introduce the notion of ``$f-$steep point'' based on the considerations
above and carry out a careful analysis of the Hausdorff dimension
of the sets consisting of steep points. Below we give a brief description
of our main results.

In $\mathbb{R}^{\nu}$ with $\nu\geq2$, let $\left\{ \bar{\theta}_{t}\left(x\right):x\in\mathbb{R}^{\nu},t\in(0,1]\right\} $
be the regularized family based on circular or spherical averages
of $\theta$, same as introduced above, of the GFF associated with
$\left(I-\Delta\right)$ in $\mathbb{R}^{\nu}$, and let $f:(0,1]\rightarrow\mathbb{R}$
be a properly chosen test function (the requirements of $f$ will
be specified later). At any location $x\in\mathbb{R}^{\nu}$, we consider
a measurement of the steepness of $\bar{\theta}_{t}\left(x\right)$,
or the rate of change of $\bar{\theta}_{t}\left(x\right)$ with respect
to $t$, as given by the integral 
\[
X_{t}^{f,\theta}\left(x\right)=\int_{1}^{t}f\left(s\right)d\bar{\theta}_{s}\left(x\right),
\]
which, as we will show later, can be interpreted as a Riemann-Stieltjes
integral\footnote{For any $0<a<b\leq1$, ``$\int_{b}^{a}$'' refers to ``$-\int_{a}^{b}$''
in the sense of Riemann or Riemann-Stieltjes integral.}. Heuristically speaking, assuming $f$ is positive, the larger $X_{t}^{f,\theta}\left(x\right)$
gets as $t\searrow0$, the ``steeper'' $\bar{\theta}_{t}\left(x\right)$
is, or the faster $\bar{\theta}_{t}$$\left(x\right)$ changes with
respect to $t$, when weighted by $f$. Furthermore, if we define
\[
\Sigma_{t}^{f}:=\int_{1}^{t}f^{2}\left(s\right)dG\left(s\right)
\]
where $G\left(t\right):=\mathbb{E}\left[\left(\bar{\theta}_{t}\left(x\right)\right)^{2}\right]$
for every $t\in(0,1]$, then we can show that $\left\{ X_{t}^{f,\theta}\left(x\right):t\in(0,1]\right\} $
has the same distribution as a Brownian motion running by the ``clock''
$\Sigma_{t}^{f}$. Therefore, if $x\in\mathbb{R}^{\nu}$ is such that
\[
\lim_{t\searrow0}\,\frac{X_{t}^{f,\theta}\left(x\right)}{\Sigma_{t}^{f}}=\sqrt{2\nu},
\]
then $x$ is a location where $X_{t}^{f,\theta}\left(x\right)$ achieves
unusually large values and we will call $x$ an\emph{ $f-$steep point
}of $\theta$. Denote by $D^{f,\theta}$ the collection of all the
$f-$steep points of $\theta$. We study the Hausdorff dimension of
$D^{f,\theta}$ and find out that a key parameter is the limit range
of the ratio $\frac{\Sigma_{t}^{f}}{-\ln t}$ as $t\searrow0$. Namely,
if we set
\[
\bar{c}_{f}:=\limsup_{t\searrow0}\frac{\Sigma_{t}^{f}}{-\ln t}\text{ and }\underline{c}_{f}:=\liminf_{t\searrow0}\frac{\Sigma_{t}^{f}}{-\ln t}.
\]
then we prove (Theorem \ref{thm:main theorem hausdorff dimension})
that, with probability one, if $\bar{c}_{f}>1$, then $D^{f,\theta}=\emptyset$;
if $0<\underline{c}_{f}\leq\bar{c}_{f}\leq1$, then 
\[
\nu\left(1-2\bar{c}_{f}+\underline{c}_{f}\right)\leq\dim_{\mathcal{H}}\left(D^{f,\theta}\right)\leq\nu\left(1-\bar{c}_{f}\right);
\]
in particular, if $\bar{c}_{f}=\underline{c}_{f}=:c_{f}\in(0,1]$,
then 
\[
\dim_{\mathcal{H}}\left(D^{f,\theta}\right)=\nu\left(1-c_{f}\right).
\]
Besides, we also investigate some exceptional sets that are closely
related to $D^{f,\theta}$, including the set of the\emph{ super $f-$steep
points}, denoted by $D_{\limsup}^{f,\theta}$, consisting of $x$
such that
\[
\limsup_{t\searrow0}\,\frac{X_{t}^{f,\theta}\left(x\right)}{\Sigma_{t}^{f}}\geq\sqrt{2\nu},
\]
and the set of the \emph{sub $f-$steep points, }denoted by $D_{\liminf}^{f,\theta}$,
consisting of $x$ such that
\[
\liminf_{t\searrow0}\,\frac{X_{t}^{f,\theta}\left(x\right)}{\Sigma_{t}^{f}}\geq\sqrt{2\nu},
\]
as well as the set of the\emph{ sequential $f-$steep points, }denoted
by $SD^{f,\theta}$, consisting of $x$ such that
\[
\lim_{m\nearrow\infty}\,\frac{X_{r_{m}}^{f,\theta}\left(x\right)}{\Sigma_{r_{m}}^{f}}=\sqrt{2\nu},
\]
where $\left\{ r_{m}:m\geq1\right\} \subseteq(0,1]$ is a sequence
with $r_{m}\searrow0$ as $m\nearrow\infty$. We provide (Propositions
\ref{prop: upper bound on Hausdorff dim of limsup set} and \ref{prop:lower bound on D^f,theta})
upper bounds and lower bounds for the Hausdorff dimension of $D_{\limsup}^{f,\theta}$,
$D_{\liminf}^{f,\theta}$ and $SD^{f,\theta}$. 

We believe that analyzing steep points can be a useful approach in
studying the geometry of GFFs. In $\mathsection4$, by setting $f$
to be specific functions, we can apply the framework of $f-$steep
point to re-produce some of the existing results on thick points,
as reviewed in the previous subsection, for both log-correlated GFFs
and polynomial-correlated GFFs. Moreover, certain choices of $f$
lead to natural generalizations of thick point, to one of which we
refer as the \emph{oscillatory thick point}. Heuristically speaking,
an oscillatory thick point is a location $x\in\mathbb{R}^{\nu}$ where
the value of the regularized field element $\bar{\theta}_{t}\left(x\right)$
achieves unusually large values both in the positive direction and
in the negative direction as $t\searrow0$, i.e., an oscillatory behavior
with unusually large amplitude is exhibited by $\bar{\theta}_{t}\left(x\right)$
as $t\searrow0$. With the framework of steep point, we can determine
the exact Hausdorff dimension of the set of oscillatory thick points
for log-correlated GFFs (Proposition \ref{prop:oscil. log}), and
provide estimates for the Hausdorff dimension of the analogous exceptional
set for polynomial-correlated GFFs (Proposition \ref{prop:oscil. polyn}).
Besides, another generalization of thick point we will consider is
the\emph{ lasting thick point}, which, roughly speaking, is a thick
point where $\bar{\theta}_{t}\left(x\right)$ spends non-negligible
portion of the total time maintaining unusually large values. Again,
using the results on steep points, we establish (Proposition \ref{prop:lasting thick point})
bounds on the Hausdorff dimension of the set of lasting thick points,
showing that, although ``rarer'' than the standard thick points,
lasting thick points can still be ``detected''.

In $\mathsection5$ we briefly discuss some generalizations and problems
related to the notion of steep point, and possible directions in which
we would like to further our study. $\mathsection6$ is the Appendix,
in which we include the lengthy and technical proofs of some of the
results in $\mathsection2$ and $\mathsection3$, to avoid the tedious
and pedagogically unimportant computations from distracting readers,
and to minimize, in the main article, the overlapping with the arguments
used in \cite{HMP} and \cite{Chen_thick_point}. 

\section{Gaussian Free Fields and Circle/Sphere Averaging Regularization}

\subsection{Abstract Wiener Space and GFFs}

A general and mathematically accurate treatment of GFFs is provided
by the theory of \emph{Abstract Wiener Space} (AWS) (\foreignlanguage{american}{\cite{aws}}),
under whose framework not only can we define and construct GFFs rigorously,
we can also interpret any regularization procedure, as mentioned in
the introduction, in a natural way. The connection between AWS and
GFF is thoroughly explained in $\mathsection2$ of \cite{Chen_thick_point},
so in this section we will not repeat the entire theory but only review
main ideas for the sake of completeness. For readers who are interested
in the general theory of AWS, we refer to \foreignlanguage{american}{\cite{aws},
\cite{awsrevisited}, \cite{add_Gaus} and $\mathsection8$ of \cite{probability}}.
Same as in \cite{Chen_thick_point}, we define GFFs in a general setting.
Given $\nu\in\mathbb{N}$ and $p\in\mathbb{R}$, consider the Sobolev
space $H^{p}:=H^{p}\left(\mathbb{R}^{\nu}\right)$, which is the closure
of $C_{c}^{\infty}\left(\mathbb{R}^{\nu}\right)$, the space of $\mathbb{R}-$valued
compactly supported smooth functions on $\mathbb{R}^{\nu}$, under
the inner product given by, 
\[
\begin{split}\forall\phi,\psi\in C_{c}^{\infty}\left(\mathbb{R}^{\nu}\right),\quad\left(\phi,\,\psi\right)_{_{H^{p}}} & :=\left(\left(I-\Delta\right)^{p}\phi,\psi\right)_{L^{2}\left(\mathbb{R}^{\nu}\right)}\\
 & =\frac{1}{\left(2\pi\right)^{\nu}}\int_{\mathbb{R}^{\nu}}\left(1+\left|\xi\right|^{2}\right)^{p}\hat{\phi}\left(\xi\right)\overline{\hat{\psi}\left(\xi\right)}d\xi,
\end{split}
\]
where ``$\left(I-\Delta\right)^{p}$'' is the Bessel operator of
order $p$, and ``$\hat{\;\;}$'' refers to the Fourier transform.
$\left(H^{p},\,\left(\cdot,\cdot\right)_{H^{p}}\right)$ forms a separable
Hilbert space. One can identify $H^{-p}$ as the dual space of $H^{p}$,
and for every $\mu\in H^{-p}$, if $h_{\mu}:=\left(I-\Delta\right)^{-p}\mu$,
then $h_{\mu}\in H^{p}$. The theory of AWS guarantees that, there
exists a separable Banach space $\Theta^{p}:=\Theta^{p}\left(\mathbb{R}^{\nu}\right)$
with the Banach norm $\left\Vert \cdot\right\Vert _{\Theta^{p}}$,
and a centered Gaussian measure $\mathcal{W}^{p}:=\mathcal{W}^{p}\left(\mathbb{R}^{\nu}\right)$
on $\left(\Theta^{p},\mathcal{B}_{\Theta^{p}}\right)$ with $\mathcal{B}_{\Theta^{p}}$
being the Borel $\sigma-$algebra, such that \\
\\
(i) $\left(H^{p},\,\left(\cdot,\cdot\right)_{H^{p}}\right)$ is continuously
embedded in $\left(\Theta^{p},\left\Vert \cdot\right\Vert _{\Theta^{p}}\right)$
as a dense subspace, so $\Theta^{p}$ is also a space of $\mathbb{R}-$valued
functions or generalized functions;\\
\\
(ii) if $\lambda\in H^{-p}$ is a linear and bounded functional on
$\Theta^{p}$ with respect to the ``action'' $\left(\cdot,\cdot\right)_{L^{2}}$,
or in other words, $\lambda$ is an element of $\left(\Theta^{p}\right)^{*}$
the due space of $\Theta^{p}$, and $h_{\lambda}:=\left(I-\Delta\right)^{-p}\lambda$,
then the following mapping 
\[
\theta\in\Theta^{p}\rightsquigarrow\mathcal{I}\left(h_{\lambda}\right)\left(\theta\right):=\left(\theta,\lambda\right)_{L^{2}}\in\mathbb{R},
\]
as a random variable on $\left(\Theta^{p},\mathcal{B}_{\Theta^{p}},\mathcal{W}^{p}\right)$,
has the Gaussian distribution with $\mathbb{E}^{\mathcal{W}}\left[\mathcal{I}\left(h_{\lambda}\right)\right]=0$
and $\text{Var}\left(\mathcal{I}\left(h_{\lambda}\right)\right)=\left\Vert h_{\lambda}\right\Vert _{H^{p}}^{2}=\left\Vert \lambda\right\Vert _{H^{-p}}^{2}$.\\

In this setting, we refer to the probability space $\left(\Theta^{p},\mathcal{B}_{\Theta^{p}},\mathcal{W}^{p}\right)$
as the \emph{dim-$\nu$ order-$p$} \emph{GFF} \footnote{In physics literature, the term ``GFF'' only refers to the case
when $p=1$. Here we slightly extend the use of this term and continue
to use ``GFF'' when $p\neq1$.} and $\left(H^{p},\,\left(\cdot,\cdot\right)_{H^{p}}\right)$ is known
as the\emph{ Cameron-Martin space} associated with this GFF.\emph{
}Besides, (i) and (ii) also imply that the mapping 
\[
\mathcal{I}:\,h_{\lambda}\in H^{p}\mapsto\mathcal{I}\left(h_{\lambda}\right)\in L^{2}\left(\mathcal{W}^{p}\right)
\]
can be extended to the whole $H^{p}$ and gives rise to an isometry
$\mathcal{I}:\left(H^{p},\,\left(\cdot,\cdot\right)_{H^{p}}\right)\rightarrow L^{2}\left(\mathcal{W}^{p}\right)$,
and its image $\left\{ \mathcal{I}\left(h\right):h\in H^{p}\right\} $
forms a centered Gaussian family under $\mathcal{W}^{p}$ with the
covariance given by, 
\[
\forall h,g\in H^{p},\quad\mathbb{E}^{\mathcal{W}^{p}}\left[\mathcal{I}\left(h\right)\mathcal{I}\left(g\right)\right]=\left(h,g\right)_{H^{p}}.
\]
The isometry $\mathcal{I}$ is called the \emph{Paley-Wiener map }and
its images $\left\{ \mathcal{I}\left(h\right):h\in H^{p}\right\} $
are known as the \emph{Paley-Wiener integrals}. There are two facts
about the Paley-Wiener integrals that we will use in our later discussions.\\
\\
1. If $\left\{ h_{n}:n\geq1\right\} \subseteq H^{p}$ is an orthonormal
basis of $\left(H^{p},\,\left(\cdot,\cdot\right)_{H^{p}}\right)$,
then $\left\{ \mathcal{I}\left(h_{n}\right):n\geq1\right\} $, under
$\mathcal{W}^{p}$, is a family of independent standard Gaussian random
variables, and for $\mathcal{W}-$almost every $\theta\in\Theta^{p}$,
\begin{equation}
\theta=\sum_{n\geq1}\mathcal{I}\left(h_{n}\right)\left(\theta\right)h_{n}.\label{eq:H_basis expansion}
\end{equation}
2. Under $\mathcal{W}^{p}$, $\left\{ \mathcal{I}\left(h_{\mu}\right):\mu\in H^{-p}\right\} $
is again a family of centered Gaussian random variables with the covariance
given by, 
\begin{equation}
\begin{split}\forall\mu_{1},\mu_{2}\in H^{-p},\quad\mathbb{E}^{\mathcal{W}^{p}}\left[\mathcal{I}\left(h_{\mu_{1}}\right)\mathcal{I}\left(h_{\mu_{2}}\right)\right] & =\left(h_{\mu_{1}},h_{\mu_{2}}\right)_{H^{p}}=\left(\mu_{1},\left(I-\Delta\right)^{-p}\mu_{2}\right)_{L^{2}}\\
 & =\frac{1}{\left(2\pi\right)^{\nu}}\int_{\mathbb{R}^{\nu}}\frac{\widehat{\mu_{1}}\left(\xi\right)\overline{\widehat{\mu_{2}}\left(\xi\right)}}{\left(1+\left|\xi\right|^{2}\right)^{-p}}d\xi.
\end{split}
\label{eq:covariance structure}
\end{equation}
The formula (\ref{eq:covariance structure}) indicates that the covariance
structure of the GFF is determined by the Green's function of $\left(I-\Delta\right)^{p}$
on $\mathbb{R}^{\nu}$. \\
\begin{rem}
\label{rem: about the choice of operator}Instead of using the Bessel
operator $\left(I-\Delta\right)^{p}$ to construct GFFs on $\mathbb{R}^{\nu}$,
one can also use $\Delta^{p}$, equipped with proper boundary conditions,
to construct GFFs on bounded domains in $\mathbb{R}^{\nu}$, and this
is the case with the GFF treated in \cite{HMP,DS1,Shef} and many
other works. The field elements obtained in either way possess similar
properties locally in space. However, we adopt $\left(I-\Delta\right)^{p}$
in our project for technical reasons. Specifically, $\left(I-\Delta\right)^{p}$
allows the GFF to be defined on the entire space $\mathbb{R}^{\nu}$,
so we do not have to worry about any boundaries or boundary conditions,
and besides, we can carry out computations with the Fourier transforms
using Parseval's identity, which simplifies the task in many occasions.
\end{rem}

With different choices of $p$ and $\nu$, a generic element of the
GFF possesses different levels of singularity or regularity. For example
($\mathsection8$, \cite{probability}) , when $p=\frac{\nu+1}{2}$,
$\Theta^{\frac{\nu+1}{2}}$ can be taken as 
\[
\Theta^{\frac{\nu+1}{2}}:=\left\{ \theta\in C\left(\mathbb{R}^{\nu}\right):\lim_{\left|x\right|\rightarrow\infty}\,\frac{\left|\theta\left(x\right)\right|}{\log\left(e+\left|x\right|\right)}=0\right\} ,
\]
where $C\left(\mathbb{R}^{\nu}\right)$ is the space of $\mathbb{R}-$valued
continuous functions on $\mathbb{R}^{\nu}$, with the Banach norm
given by
\[
\left\Vert \theta\right\Vert _{\Theta^{\frac{\nu+1}{2}}}:=\sup_{x\in\mathbb{R}^{\nu}}\,\frac{\left|\theta\left(x\right)\right|}{\log\left(e+\left|x\right|\right)}.
\]
In other words, a generic element of the dim-$\nu$ order-$\frac{\nu+1}{2}$
GFF is a continuous function on $\mathbb{R}^{\nu}$ that grows slower
than logrithmically at infinity. In general, with $\nu$ fixed, the
larger $p$ is, the more regular the generic element is, and the smaller
$p$ is, the more singular the GFF becomes. In most of the cases that
are interesting to us, generic elements of the GFFs are only tempered
distributions and may not be point-wisely defined. For example, if
$p=\nu/2$, then the dim-$\nu$ order-$(\nu/2)$ GFF is a \emph{log-correlated
GFF} since the Green's function of $\left(I-\Delta\right)^{\nu/2}$
on $\mathbb{\mathbb{R}^{\nu}}$ has a logarithmic singularity along
the diagonal; in particular, the two-dimensional log-correlated GFFs
are most studied. On the other hand, if $p<\nu/2$ and $2p\in\mathbb{N}$
, then the field is a \emph{polynomial-correlated GFF }because in
this case, the corresponding Green's function has a polynomial singularity
of degree $\nu-2p$ along the diagonal. In this article, we aim to
explore new ways to study certain exceptional sets of GFFs with $p\in\mathbb{N}$
and $p\leq\nu/2$, and for the same reason as pointed out in $\mathsection3$
of \cite{Chen_thick_point}, we only need to treat the case when $p=1$
and $\nu\geq2$, without losing any generality. 

\subsection{Circular or Spherical Averages of GFFs}

For the rest of this article, we assume that $p=1$, $\nu\in\mathbb{N}$
and $\nu\geq2$, and write $H^{1}$, $\Theta^{1}$, and $\mathcal{W}^{1}$
as, respectively, $H$, $\Theta$ and $\mathcal{W}$ for simplification.
Denote by $\theta$ a generic element of the dim-$\nu$ order-$1$
GFF, i.e., $\theta\in\Theta$ is sampled under $\mathcal{W}$. We
have explained in the previous subsection that ``$\theta\left(x\right)$''
is not necessarily defined for every $x\in\mathbb{R}^{\nu}$, so we
will need to invoke a regularization procedure to study the behavior
of $\theta$ near $x$. As mentioned in the Introduction, in this
article we adopt the regularization based on the average of $\theta$
over a circle/sphere centered at $x$, which serves as an approximation
for ``$\theta\left(x\right)$'' as the radius tends to zero. Let
$B\left(x,t\right)$ and $\partial B\left(x,t\right)$ be the open
disc/ball and, respectively, the circle/sphere centered at $x\in\mathbb{R}^{\nu}$
with radius (under the Euclidean metric) $t\in(0,1]$, $\sigma_{x,t}$
be the length/surface measure on $\partial B\left(x,t\right)$, $\alpha_{\nu}:=\frac{2\pi^{\nu/2}}{\Gamma\left(\nu/2\right)}$
be the dimensional constant, and $\sigma_{x,t}^{ave}:=\frac{\sigma_{x,t}}{\alpha_{\nu}t^{\nu-1}}$
be the circle/sphere averaging measure over $\partial B\left(x,t\right)$.
By straightforward computations, we see that for every $x\in\mathbb{R}^{\nu}$
and $t\in(0,1]$, the Fourier transform of $\sigma_{x,t}^{ave}$ is
given by, 
\begin{equation}
\forall\xi\in\mathbb{R}^{\nu},\quad\widehat{\sigma_{x,t}^{ave}}\left(\xi\right)=\frac{\left(2\pi\right)^{\frac{\nu}{2}}}{\alpha_{\nu}}\,e^{i\left(x,\xi\right)_{\mathbb{R}^{\nu}}}\cdot\left(t\left|\xi\right|\right)^{\frac{2-\nu}{2}}J_{\frac{\nu-2}{2}}\left(t\left|\xi\right|\right)\label{eq:Fourier transform of spherical average}
\end{equation}
where $J_{\frac{\nu-2}{2}}$ is the standard Bessel function of the
first kind with index $\frac{\nu-2}{2}$. It is easy to check, using
(\ref{eq:Fourier transform of spherical average}) and the asymptotics
of $J_{\frac{\nu-2}{2}}$ at infinity, that $\sigma_{x,t}^{ave}\in H^{-1}\left(\mathbb{R}^{\nu}\right)$
and hence $h_{\sigma_{x,t}^{ave}}:=\left(I-\Delta\right)^{-1}\sigma_{x,t}^{ave}\in H$
and $\mathcal{I}\left(h_{\sigma_{x,t}^{ave}}\right)$ is a centered
Gaussian random variable under $\mathcal{W}$. This is to say that,
no matter how big $\nu$ is, no matter how singular the field element
$\theta$ is, one can always average $\theta$ over a circle/sphere
in $\mathbb{R}^{\nu}$ in the sense that the average exists as a centered
Gaussian random variable. Furthermore, $\left\{ \mathcal{I}\left(h_{\sigma_{x,t}^{ave}}\right):x\in\mathbb{R}^{\nu},t\in(0,1]\right\} $
forms a centered Gaussian family under $\mathcal{W}$ with the covariance
given by, for $x,y\in\mathbb{R}^{\nu}$, $t,s\in(0,1]$, when $x=y$,
\begin{equation}
\mathbb{E}^{\mathcal{W}}\left[\mathcal{I}\left(h_{\sigma_{x,t}^{ave}}\right)\mathcal{I}\left(h_{\sigma_{x,s}^{ave}}\right)\right]=\frac{1}{\alpha_{\nu}\left(ts\right)^{\frac{\nu-2}{2}}}\int_{0}^{\infty}\frac{\tau J_{\frac{\nu-2}{2}}\left(t\tau\right)J_{\frac{\nu-2}{2}}\left(s\tau\right)}{1+\tau^{2}}d\tau,\label{eq:covariance for (1-Delta)^s concentric}
\end{equation}
and when $x\neq y$,
\begin{eqnarray}
 &  & \begin{split} & \mathbb{E}^{\mathcal{W}}\left[\mathcal{I}\left(h_{\sigma_{x,t}^{ave}}\right)\mathcal{I}\left(h_{\sigma_{y,s}^{ave}}\right)\right]\\
 & \qquad=\frac{\left(2\pi\right)^{\nu/2}}{\alpha_{\nu}^{2}\left(ts\left|x-y\right|\right)^{\frac{\nu-2}{2}}}\int_{0}^{\infty}\frac{\tau^{2-\frac{\nu}{2}}J_{\frac{\nu-2}{2}}\left(t\tau\right)J_{\frac{\nu-2}{2}}\left(s\tau\right)J_{\frac{\nu-2}{2}}\left(\left|x-y\right|\tau\right)}{1+\tau^{2}}d\tau.
\end{split}
\label{eq:covariance for (1-Delta)^s-1}
\end{eqnarray}
The Gaussian family consisting of the circular/spherical averages
has been carefully treated and the integrals in (\ref{eq:covariance for (1-Delta)^s concentric})
and (\ref{eq:covariance for (1-Delta)^s-1}) have been computed explicitly
in $\mathsection3$ of \cite{Chen_thick_point}. Here we only cite
the results from \cite{Chen_thick_point} that are relevant to our
project, but do not repeat the calculations. Readers can turn to \cite{Chen_thick_point}
for details.
\begin{lem}
\label{lem:concentric cov} Let $x\in\mathbb{R}^{\nu}$ be fixed.
The distribution of the centered Gaussian family 
\[
\left\{ \mathcal{I}\left(h_{\sigma_{x,t}^{ave}}\right):t\in(0,1]\right\} 
\]
does not dependent on $x$ and\footnote{For two real numbers $a$ and $b$, ``$a\wedge b$'' refers to $\min\left\{ a,b\right\} $
and ``$a\vee b$'' refers to $\max\left\{ a,b\right\} .$}, 
\[
\forall t,s>0,\quad\mathbb{E}^{\mathcal{W}}\left[\mathcal{I}\left(h_{\sigma_{x,t}^{ave}}\right)\mathcal{I}\left(h_{\sigma_{x,s}^{ave}}\right)\right]=\frac{1}{\alpha_{\nu}\left(ts\right)^{\frac{\nu-2}{2}}}I_{\frac{\nu-2}{2}}\left(t\wedge s\right)K_{\frac{\nu-2}{2}}\left(t\vee s\right),
\]
where $I_{\frac{\nu-2}{2}}$ and $K_{\frac{\nu-2}{2}}$ are the modified
Bessel functions with index $\frac{\nu-2}{2}$. 

Further, if we renormalize the averages by defining 
\[
\forall t>0,\quad\bar{\sigma}_{x,t}:=\frac{\left(t/2\right)^{\frac{\nu-2}{2}}}{\Gamma\left(\nu/2\right)I_{\frac{\nu-2}{2}}\left(t\right)}\sigma_{x,t}^{ave},
\]
and set $\bar{\theta}_{t}\left(x\right):=\mathcal{I}\left(h_{\bar{\sigma}_{x,t}}\right)\left(\theta\right)$,
then $\left\{ \bar{\theta}_{t}\left(x\right):\,t\in(0,1]\right\} $
is a centered Gaussian process with the covariance give by, 
\begin{equation}
\forall0<s\leq t\leq1,\quad\mathbb{E}^{\mathcal{W}}\left[\bar{\theta}_{t}\left(x\right)\bar{\theta}_{s}\left(x\right)\right]=\frac{\alpha_{\nu}}{\left(2\pi\right)^{\nu}}\frac{K_{\frac{\nu-2}{2}}\left(t\right)}{I_{\frac{\nu-2}{2}}\left(t\right)}=:G\left(t\right).\label{eq:concentric cov renorm}
\end{equation}

In particular, $\left\{ \bar{\theta}_{t}\left(x\right):\,t\in(0,1]\right\} $
is a time-changed Brownian motion in the sense that if 
\[
\tau=\tau\left(t\right):=G\left(t\right)-G\left(1\right)\mbox{ for }t\in(0,1],
\]
then 
\[
\left\{ B_{\tau}:=\bar{\theta}_{G^{-1}\left(\tau+G\left(1\right)\right)}\left(x\right)-\bar{\theta}_{1}\left(x\right):\:\tau\geq0\right\} 
\]
has the same distribution as a standard Brownian motion. 
\end{lem}

One can verify that as $t\searrow0$,
\[
\lim_{t\searrow0}\frac{\left(t/2\right)^{\frac{\nu-2}{2}}}{\Gamma\left(\nu/2\right)I_{\frac{\nu-2}{2}}\left(t\right)}=1,
\]
so $\bar{\theta}_{t}\left(x\right)$ still is a ``legitimate'' approximation
of ``$\theta\left(x\right)$''. Moreover, by the asymptotics of
$K_{\frac{\nu-2}{2}}$ and $I_{\frac{\nu-2}{2}}$ near zero, the function
$G$ defined in (\ref{eq:concentric cov renorm}) is positive, smooth
and decreasing on $\left(0,\infty\right)$, and when $t$ is small,
\begin{equation}
G\left(t\right)=\begin{cases}
\frac{1}{2\pi}\left(-\ln t\right)+\mathcal{O}\left(1\right) & \mbox{ if }\nu=2,\\
\frac{1}{\alpha_{\nu}\left(\nu-2\right)}\cdot t^{2-\nu}+\mathcal{O}\left(t^{3-\nu}\right) & \mbox{ if }\nu\geq3,
\end{cases}\label{eq:asymptotic of G in 2D}
\end{equation}
which reflects the fact that the dim-$\nu$ order-1 GFF is log-correlated
in 2D and polynomial-correlated with degree $\nu-2$ in three and
higher dimensions. 

Throughout the rest of the article, we adopt 
\[
\left\{ \bar{\theta}_{t}\left(x\right):x\in\mathbb{R}^{\nu},t\in(0,1]\right\} 
\]
as the regularization of $\theta$. Not only does it reduce to a Brownian
motion (up to a time change) for the concentric family at every point
$x$, it also possesses favorable properties for the non-concentric
family under certain circumstances. 
\begin{lem}
\label{lem:nonconcentric cov} Assume that $x,y\in\mathbb{R}^{\nu}$and
$t,s\in(0,1]$.

(i) If $\left|x-y\right|\geq t+s$, i.e., if $B\left(x,t\right)\cap B\left(y,s\right)=\emptyset$,
then
\begin{equation}
\mathbb{E}^{\mathcal{W}}\left[\bar{\theta}_{t}\left(x\right)\bar{\theta}_{s}\left(y\right)\right]=\left(2\pi\right)^{-\nu/2}\frac{K_{\frac{\nu-2}{2}}\left(\left|x-y\right|\right)}{\left|x-y\right|^{\frac{\nu-2}{2}}}=:C_{disj}\left(\left|x-y\right|\right),\label{eq:cov non-concentric non-overlapping}
\end{equation}
In particular, when $\left|x-y\right|$ is small, 
\[
C_{disj}\left(\left|x-y\right|\right)=G\left(\left|x-y\right|\right)+\mathcal{O}\left(1\right),
\]
where $G$ is the same as in (\ref{eq:concentric cov renorm}).

(ii) If $t\geq\left|x-y\right|+s$, i.e., if $B\left(x,t\right)\supset B\left(y,s\right)$,
then
\begin{equation}
\mathbb{E}^{\mathcal{W}}\left[\bar{\theta}_{t}\left(x\right)\bar{\theta}_{s}\left(y\right)\right]=\left(2\pi\right)^{-\nu/2}\frac{I_{\frac{\nu-2}{2}}\left(\left|x-y\right|\right)}{\left|x-y\right|^{\frac{\nu-2}{2}}}\frac{K_{\frac{\nu-2}{2}}\left(t\right)}{I_{\frac{\nu-2}{2}}\left(t\right)}=:C_{incl}\left(t,\left|x-y\right|\right),\label{eq:cov non-concentric inclusion}
\end{equation}
In particular, when $t$ is small (and hence $\left|x-y\right|$ is
also small),
\[
C_{incl}\left(t,\left|x-y\right|\right)=G\left(t\right)+\mathcal{O}\left(1\right).
\]
\end{lem}

We would like to point out that (\ref{eq:cov non-concentric non-overlapping})
and (\ref{eq:cov non-concentric inclusion}) showcase the advantage
of this particular choice of regularization. Under the assumption
(i) or (ii) of Lemma \ref{lem:nonconcentric cov}, small radius (radii)
does not affect the covariance, which is a desirable property to have
when studying ``convergence'' in any reasonable sense as radius
(radii) tends to zero. 

The project we will carry out in this article only concerns local
behaviors of the GFF, and obviously, the distribution of ``$\theta\left(x\right)$''
is invariant under translations in the spatial variable $x$. So,
without loss of generality, we may only consider the GFF restricted
over $\overline{S\left(O,1\right)}$, the closed square/cube centered
at the origin with side length $2$ under the Euclidean metric\footnote{Similarly, for $x\in\mathbb{R}^{\nu}$ and $s>0$, $S\left(x,s\right)$
and $\overline{S\left(x,s\right)}$ are the Euclidean open and, respectively,
closed square/cube centered at $x$ with side length $2s$.}. An important factor in treating a GFF via a regularization is the
continuity property possessed by the regularized family. For the family
$\left\{ \bar{\theta}_{t}\left(x\right):x\in\overline{S\left(O,1\right)},t\in(0,1]\right\} $,
its continuity has been investigated in \cite{HMP} and \cite{Chen_thick_point},
via standard techniques, such as Kolmogorov's continuity criterion
(e.g., $\mathsection4$ in \cite{probability}) and the classical
entropy method (e.g., \cite{Dudley,Talagrand,AT07}). Here we review
some results on the continuity modulus of $\bar{\theta}_{t}\left(x\right)$,
and they will become important technical tools in our later discussions. 
\begin{lem}
\label{lem:estimate for d^2_t non-concentric} \label{lem:expectation of max non-concentric}Let
$\nu\geq2$. Consider the intrinsic metric $d$ associated with the
Gaussian family 
\[
\left\{ \bar{\theta}_{t}\left(x\right):x\in\overline{S\left(O,1\right)},\,t\in(0,1]\right\} 
\]
 given by
\[
\forall x,y\in\mathbb{R}^{\nu},\,\forall t,s\in(0,1],\quad d\left(x,t;\,y,s\right):=\sqrt{\mathbb{E}^{\mathcal{W}}\left[\left(\bar{\theta}_{t}\left(x\right)-\bar{\theta}_{s}\left(y\right)\right)^{2}\right]}.
\]

(i) There exists a constant\footnote{Throughout the article, $C$ refers to a constant that only depends
on the dimension $\nu$. $C$'s value may vary from line to line.} $C>0$ such that for every $t,s\in(0,1]$ and every $x,y\in\overline{S\left(O,1\right)}$
, 
\[
d^{2}\left(x,t;\,y,t\right)\leq Ct^{2-\nu}\sqrt{\frac{\left|x-y\right|}{t}}.
\]
and hence, 
\[
d^{2}\left(x,t;\,y,s\right)\leq C\left(t^{2-\nu}\sqrt{\frac{\left|x-y\right|}{t}}+\left|G\left(t\right)-G\left(s\right)\right|\right).
\]
Therefore, we may assume that for every $\theta\in\Theta$, the function
\[
\left(x,t\right)\in\overline{S\left(O,1\right)}\times(0,1]\rightsquigarrow\bar{\theta}_{t}\left(x\right)\in\mathbb{R}
\]
is continuous. 

(ii) There exists a constant $C>0$ such that for every $t\in(0,1]$
and every $0<\delta<\sqrt{G\left(t\right)}$, if
\[
\omega_{t}^{\theta}\left(\delta\right):=\sup\left\{ \left|\bar{\theta}_{s}\left(x\right)-\bar{\theta}_{s^{\prime}}\left(y\right)\right|:\,d\left(x,s;y,s^{\prime}\right)\leq\delta,\,x,y\in\overline{S\left(O,1\right)},\,s,s^{\prime}\in\left[t,1\right]\right\} ,
\]
then
\[
\mathbb{E}^{\mathcal{W}}\left[\omega_{t}^{\theta}\left(\delta\right)\right]\leq C\delta\sqrt{\ln\left(t^{\left(3-2\nu\right)/4}/\delta\right)}.
\]
\end{lem}

The proof of Lemma \ref{lem:expectation of max non-concentric} is
left in the Appendix $\mathsection6.1$, because the arguments are
based on straightforward calculations following the standard entropy
method, and are very similar to those in $\mathsection3$ of \cite{Chen_thick_point}.

\section{Steep Points of Gaussian Free Fields}

Let $\left(H,\Theta,\mathcal{W}\right)$ be the dim-$\nu$ order-$1$
GFF, $\nu\geq2$, $\theta\in\Theta$ be a generic element of the GFF,
and for each $t\in(0,1]$ and $x\in\overline{S\left(O,1\right)}$,
$\bar{\theta}_{t}\left(x\right)$ be the renormalized circular/spherical
average as introduced in the previous section. As defined in (\ref{eq:2D thick point})
and (\ref{eq:thick point def}) in the Introduction, thick points
of $\theta$ are, intuitively speaking, locations of ``high peaks''
in the graph of $\theta$; more rigorously, thick points are defined
as $x\in\overline{S\left(O,1\right)}$ such that the value of $\bar{\theta}_{t}\left(x\right)$
is unusually large for $t$ being sufficiently small. In this section,
we will focus on another perspective of the behavior of $\bar{\theta}_{t}\left(x\right)$,
that is, the rate of change of $\bar{\theta}_{t}\left(x\right)$ as
$t\searrow0$. If one could establish, in a proper sense, that for
some $x\in\overline{S\left(O,1\right)}$, the rate of change of $\bar{\theta}_{t}\left(x\right)$
is unusually large when $t$ is small, then one would expect that
the ``landscape'' of $\theta$ near $x$ is unusually steep since
$\bar{\theta}_{t}\left(x\right)$ is approximately the average of
$\theta$ over $\partial B\left(x,t\right)$. Taking into account
of this consideration, we refer to $x\in\overline{S\left(O,1\right)}$
where $\bar{\theta}_{t}\left(x\right)$ changes unusually fast in
$t$ as $t\searrow0$ as a ``steep point'' of $\theta$. 

Although ``$\frac{d}{dt}\left(\bar{\theta}_{t}\left(x\right)\right)$''
is the natural thing to consider when studying the rate of change
of $\bar{\theta}_{t}\left(x\right)$ with respect to $t$, it is clear
from the last statement of Lemma \ref{lem:concentric cov} that at
any given $x$, $\bar{\theta}_{t}\left(x\right)$ is almost surely
nowhere differentiable in $t$. To overcome the indifferentiability,
we will study the change of rate of $\bar{\theta}_{t}$ using certain
test function $f:(0,1]\rightarrow\mathbb{R}$, that is, study $\frac{d}{dt}\bar{\theta}_{t}\left(x\right)$
weighted by $f\left(t\right)$. The choice of such test function $f$
is rather general, provided $f$ satisfying some basic requirements.
First, in order to pair $f$ with $\frac{d}{dt}\bar{\theta}_{t}\left(x\right)$,
$f$ should have bounded variation at least locally on $(0,1]$. Second,
since $f$ has to overcome the singularity of the field element when
$t$ is small, it is natural to require $\left|f\right|$, the absolute
value of $f$, to decay to $0$ sufficiently fast as $t\searrow0$.
On the other hand, $\left|f\right|$ should not decay too fast so
that unusual behaviors of $\frac{d}{dt}\bar{\theta}_{t}\left(x\right)$
can still be captured. In addition, for technical reasons, we also
require $f$ to not ``jump'' too frequently. There is flexibility
in setting up the class of test functions to which the methods and
the results discussed in later sections apply. For the pedagogical
purpose, we adopt the following specific class of test functions. 
\begin{defn}
\label{def:class C }Define $\mathcal{C}$ to be the family of function
$f:(0,1]\rightarrow\mathbb{R}$ satisfying that 

\textbf{(a)} there exist constants\footnote{The constants $C_{f}$ and $\rho_{f}$ may depend on $f$, and the
values may be different from line to line. } $C_{f}>0$ and $\rho_{f}>0$ such that 
\[
\forall t\in(0,1],\quad\left|f\left(t\right)\sqrt{G\left(t\right)}\right|\leq C_{f}\left[\left(-\ln t\right)^{\rho_{f}}+1\right];
\]

\textbf{(b)} $f:(0,1]\rightarrow\mathbb{R}$ is left-continuous, $f$
has at most countably many jump discontinuities, and if
\[
\mathcal{J}:=\left\{ t_{j}:j\geq1,\;0<\cdots<t_{j+1}<t_{j}<\cdots<t_{1}<1\right\} 
\]
is the collection of the jump discontinuities of $f$ (in decreasing
order), then\footnote{For a discrete set $A$, ``$\#(A)$'' refers to the cardinality
of $A$.} 
\[
\forall t\in(0,1],\quad\#\left(\mathcal{J}\cap\left[t,1\right]\right)\leq C_{f}\left[\left(-\ln t\right)^{\rho_{f}}+1\right];
\]

\textbf{(c)} for each $j\geq1$, the absolute value function $\left|f\right|$
is non-decreasing on $(t_{j+1},t_{j}]$ (but $\left|f\right|$ does
not have to be non-decreasing on $(0,1]$), which, combined with \textbf{(a)}
and \textbf{(b)}, implies that $f\in BV_{loc}\left((0,1]\right)$; 

\textbf{(d)} if for every $t\in(0,1]$, $\Sigma_{t}^{f}:=\int_{1}^{t}\,f^{2}\left(s\right)dG\left(s\right)$
(defined as a Riemann-Stieltjes integral), then $\Sigma_{t}^{f}\nearrow\infty$
as $t\searrow0$. 
\end{defn}

Given $f\in\mathcal{C}$, since for every $\theta\in\Theta$, $x\in\overline{S\left(O,1\right)}$
and $t\in(0,1]$, $s\in\left[t,1\right]\rightsquigarrow\bar{\theta}_{s}\left(x\right)$
is continuous, and $f\in BV\left(\left[t,1\right]\right)$, we can
also define 
\[
X_{t}^{f,\theta}\left(x\right):=\int_{1}^{t}\,f\left(s\right)d\bar{\theta}_{s}\left(x\right)=f\left(t\right)\bar{\theta}_{t}\left(x\right)-f\left(1\right)\bar{\theta}_{1}\left(1\right)-\int_{1}^{t}\bar{\theta}_{s}\left(x\right)df\left(s\right)
\]
as a Riemann-Stieltjes integral. Again, by the last statement in Lemma
\ref{lem:concentric cov}, it is clear that for any fixed $x$, $\left\{ X_{t}^{f,\theta}\left(x\right):t\in(0,1]\right\} $
is a Gaussian process with independent increment (in the direction
of $t$ decreasing) and 
\[
\forall t\in(0,1],\quad\mathbb{E}^{\mathcal{W}}\left[\left(X_{t}^{f,\theta}\left(x\right)\right)^{2}\right]=\Sigma_{t}^{f}.
\]
In other words, $\left\{ X_{t}^{f,\theta}\left(x\right):t\in(0,1]\right\} $
can also be viewed as a Brownian motion running by the ``clock''
$\Sigma_{t}^{f}$, and according to\textbf{ (d)} in Definition \ref{def:class C },
the ``clock'' goes on forever. It follows immediately from the Law
of the Iterated Logarithm that for every $x\in\overline{S\left(O,1\right)}$,
\begin{equation}
\limsup_{t\searrow0}\frac{X_{t}^{f,\theta}\left(x\right)}{\sqrt{2\Sigma_{t}^{f}\,\ln\ln\Sigma_{t}^{f}}}=1\quad\mbox{for }\mbox{\ensuremath{\mathcal{W}}-}\mbox{almost every }\theta\in\Theta.\label{eq:LIL for X_t}
\end{equation}

Analogous to the idea of searching for thick points, we want to identify
points $x$ where $X_{t}^{f,\theta}\left(x\right)$ becomes unusually
large and will define such points as $f-$steep points of $\theta$.
\begin{defn}
\label{def:steep point}Given $f\in\mathcal{C}$, $x\in\overline{S\left(O,1\right)}$
is called an $f-$\emph{steep point }of $\theta\in\Theta$ if
\begin{equation}
\lim_{t\searrow0}\,\frac{X_{t}^{f,\theta}\left(x\right)}{\Sigma_{t}^{f}}=\sqrt{2\nu}.\label{eq:def of steep point}
\end{equation}
Let $D^{f,\theta}$ be the set of all the $f-$steep points of $\theta$. 

Related to $D^{f,\theta}$, we also introduce the set of\emph{ super
$f-$steep points }given by 
\[
D_{\limsup}^{f,\theta}:=\left\{ x\in\overline{S\left(O,1\right)}:\,\limsup_{t\searrow0}\,\frac{X_{t}^{f,\theta}\left(x\right)}{\Sigma_{t}^{f}}\geq\sqrt{2\nu}\right\} 
\]
as well as the set of\emph{ sub $f-$steep points }given by
\[
D_{\liminf}^{f,\theta}:=\left\{ x\in\overline{S\left(O,1\right)}:\,\liminf_{t\searrow0}\,\frac{X_{t}^{f,\theta}\left(x\right)}{\Sigma_{t}^{f}}\geq\sqrt{2\nu}\right\} .
\]
\end{defn}

Obviously,
\begin{equation}
D^{f,\theta}\subseteq D_{\liminf}^{f,\theta}\subseteq D_{\limsup}^{f,\theta},\label{eq:sets relation 1}
\end{equation}
and the simple observation (\ref{eq:LIL for X_t}) implies that for
every $x\in\overline{S\left(O,1\right)}$, 
\[
\mathcal{W}\left(x\in D^{f,\theta}\right)=\mathcal{W}\left(x\in D_{\liminf}^{f,\theta}\right)=\mathcal{W}\left(x\in D_{\limsup}^{f,\theta}\right)=0.
\]
In other words, the (super/sub) $f-$steep point sets are exceptional
sets of the GFF. Indeed, one can clearly see from (\ref{eq:def of steep point})
that the $f-$steep points of $\theta$ should be viewed as the thick
points corresponding to the process $\left\{ X_{t}^{f,\theta}:t\in(0,1]\right\} $.
As a consequence, one would expect that the results established on
the Hausdorff dimension of thick point sets, as reviewed in $\mathsection1$,
can be extended to steep point sets. This is our goal in this section,
and we summarize the main results in the following theorem.
\begin{thm}
\label{thm:main theorem hausdorff dimension}Given $f\in\mathcal{C}$,
set

\[
\bar{c}_{f}\,:=\limsup_{t\searrow0}\frac{\Sigma_{t}^{f}}{-\ln t}\;\mbox{and }\underline{c}_{f}:=\liminf_{t\searrow0}\frac{\Sigma_{t}^{f}}{-\ln t}.
\]

(i) For $\mathcal{W}-$almost every $\theta\in\Theta$, if $1<\bar{c}_{f}\leq\infty$,
then 
\[
D^{f,\theta}=D_{\liminf}^{f,\theta}=\emptyset;
\]
if $0<\underline{c}_{f}\leq\bar{c}_{f}\leq1$, then 

\[
\nu\left(1-2\bar{c}_{f}+\underline{c}_{f}\right)\leq\dim_{\mathcal{H}}\left(D^{f,\theta}\right)\leq\dim_{\mathcal{H}}\left(D_{\liminf}^{f,\theta}\right)\leq\nu\left(1-\bar{c}_{f}\right).
\]

(ii) For $\mathcal{W}-$almost every $\theta\in\Theta$, if $1<\underline{c}_{f}\leq\infty$,
then 
\[
D_{\limsup}^{f,\theta}=\emptyset;
\]
if $0<\underline{c}_{f}\leq1$, then 
\[
\nu\left(1-2\bar{c}_{f}+\underline{c}_{f}\right)\leq\dim_{\mathcal{H}}\left(D_{\limsup}^{f,\theta}\right)\leq\nu\left(1-\underline{c}_{f}\right).
\]
\end{thm}

Certainly the lower bound in the statements (i) and (ii) is only meaningful
if $2\bar{c}_{f}-\underline{c}_{f}<1$. However, as we will see later
that the estimates above, especially the lower bounds of the Hausdorff
dimension of the concerned sets, can be improved through imposing
further constraints on $f$. In particular, when $\bar{c}_{f}=\underline{c}_{f}$,
Theorem \ref{thm:main theorem hausdorff dimension} is reduced to
the following fact.
\begin{cor}
\label{cor:when sigma/ln limit exists}Given $f\in\mathcal{C}$, suppose
that
\[
\bar{c}_{f}=\underline{c}_{f}=c_{f}:=\lim_{t\searrow0}\frac{\Sigma_{t}^{f}}{-\ln t}.
\]
Then, for $\mathcal{W}-$almost every $\theta\in\Theta$, if $1<c_{f}\leq\infty$,
then 
\[
D^{f,\theta}=D_{\liminf}^{f,\theta}=D_{\limsup}^{f,\theta}=\emptyset;
\]
if $0<c_{f}\leq1$, then 
\[
\dim_{\mathcal{H}}\left(D^{f,\theta}\right)=\dim_{\mathcal{H}}\left(D_{\liminf}^{f,\theta}\right)=\dim_{\mathcal{H}}\left(D_{\limsup}^{f,\theta}\right)=\nu\left(1-c_{f}\right).
\]
\end{cor}

Besides $D^{f,\theta}$, $D_{\liminf}^{f,\theta}$ and $D_{\limsup}^{f,\theta}$,
we will also treat the set of \emph{sequential $f-$steep points}
given by
\begin{equation}
SD^{f,\theta}:=\left\{ x\in\overline{S\left(O,1\right)}:\lim_{m\nearrow\infty}\frac{X_{r_{m}}^{f,\theta}\left(x\right)}{\Sigma_{r_{m}}^{f}}=\sqrt{2\nu}\right\} ,\label{eq:def of sequential steep point}
\end{equation}
where $\left\{ r_{m}:m\geq1\right\} \subseteq(0,1]$ is a sequence
such that $r_{m}\searrow0$ as $m\nearrow\infty$. Obviously,
\begin{equation}
D^{f,\theta}\subseteq SD^{f,\theta}\subseteq D_{\limsup}^{f,\theta}.\label{eq:sets relation 2}
\end{equation}
When $f$ and $\left\{ r_{m}:m\geq1\right\} $ satisfy proper conditions,
on one hand, our method allows us to derive estimates for $\dim_{\mathcal{H}}\left(SD^{f,\theta}\right)$,
and on the other hand, as we will see in Proposition \ref{prop:lower bound on D^f,theta}(iii),
the result on $\dim_{\mathcal{H}}\left(SD^{f,\theta}\right)$ leads
to an improvement of the  lower bound of $\dim_{\mathcal{H}}\left(D_{\limsup}^{f,\theta}\right)$
given in Theorem \ref{thm:main theorem hausdorff dimension}. 

The proof of Theorem \ref{thm:main theorem hausdorff dimension} follows
a similar line of arguments as in \cite{HMP} and \cite{Chen_thick_point},
combined with an analysis of the continuity property of the family
$\left\{ X_{t}^{f,\theta}\left(x\right):x\in\overline{S\left(O,1\right)},t\in(0,1]\right\} $
which we will carry out with the help of Lemma \ref{lem:expectation of max non-concentric}.
Below we will study the Hausdorff dimension of $D^{f,\theta}$, $D_{\liminf}^{f,\theta}$,
$SD^{f,\theta}$ and $D_{\limsup}^{f,\theta}$ by establishing the
``upper bounds'' and the ``lower bounds'' separately.
\begin{rem}
The condition ``$\underline{c}_{f}>0$'' used in Theorem \ref{thm:main theorem hausdorff dimension},
or the condition ``$c_{f}>0$'' in Corollary \ref{cor:when sigma/ln limit exists},
can be dropped in certain circumstances. As we will see in $\mathsection3.2$,
this condition is only needed for technical reasons in the proof of
the lower bound for $\dim_{\mathcal{H}}\left(D^{f,\theta}\right)$.
We will discuss in $\mathsection3.2$ and $\mathsection4.2.3$ how
the methods and the results may still apply in certain cases even
if $\underline{c}_{f}=0$ or $c_{f}=0$. 
\end{rem}

\subsection{Upper Bounds}

This subsection is devoted to establishing the upper bounds for the
Hausdorff dimension of the concerned exceptional sets. To get started,
we need to develop estimates related to the modulus of continuity
for the family
\[
\left\{ \frac{X_{t}^{f,\theta}\left(x\right)}{\Sigma_{t}^{f}}:x\in\overline{S\left(O,1\right)},t\in(0,1]\right\} .
\]
Instead of studying this Gaussian family directly, our strategy is
to make use of the existing results on the modulus of continuity of
\[
\left\{ \bar{\theta}_{t}\left(x\right):x\in\overline{S\left(O,1\right)},t\in(0,1]\right\} .
\]
In particular, we have the following lemma.
\begin{lem}
\label{lem:modulus of X}For every $n\geq1$, let $B_{n}$ be the
subset of $\overline{S\left(O,1\right)}\times\overline{S\left(O,1\right)}\times(0,1]$
that 
\[
B_{n}:=\left\{ \left(x,y,t\right):\,x,y\in\overline{S\left(O,1\right)},\left|x-y\right|<2^{-\left(n+1\right)^{2}}2\sqrt{\nu},\,t\in\left[2^{-n^{2}},2^{-\left(n-1\right)^{2}}\right]\right\} .
\]
Then, for any $f\in\mathcal{C}$, when $n$ is sufficiently large,
\[
\mathbb{E}^{\mathcal{W}}\left[\sup_{\left(x,y,t\right)\in B_{n}}\left|\frac{X_{t}^{f,\theta}\left(y\right)}{\Sigma_{t}^{f}}-\frac{X_{t}^{f,\theta}\left(x\right)}{\Sigma_{t}^{f}}\right|\right]\leq2^{-\frac{n}{4}},
\]
and hence
\begin{equation}
\mathcal{W}\left(\sup_{\left(x,y,t\right)\in B_{n}}\left|\frac{X_{t}^{f,\theta}\left(y\right)}{\Sigma_{t}^{f}}-\frac{X_{t}^{f,\theta}\left(x\right)}{\Sigma_{t}^{f}}\right|>2^{-\frac{n}{8}}\;\mbox{ i.o. }\right)=0.\label{eq:modulus on X/Sigma}
\end{equation}
\end{lem}

Again, we skip the technical details for now, and leave the complete
proof in the Appendix $\mathsection6.2$. However, we would like to
point out that the proof of Lemma \ref{lem:modulus of X} makes use
of the conditions\textbf{ (a)(b)(c)} on $f\in\mathcal{C}$ as required
in Definition \ref{def:class C }, but it is clear from the proof
that those conditions are not unique, and they merely serve technical
purposes.

We are now ready to prove the upper bounds for $\dim_{\mathcal{H}}\left(D^{f,\theta}\right)$
and $\dim_{\mathcal{H}}$$\left(D_{\limsup}^{f,\theta}\right)$, and
the exactly same arguments also lead to an upper bound of $\dim_{\mathcal{H}}\left(SD^{f,\theta}\right)$
associated with any sequence that decays to zero.
\begin{prop}
\label{prop: upper bound on Hausdorff dim of limsup set}Given $f\in\mathcal{C}$,
let $\bar{c}_{f}$ and $\underline{c}_{f}$ be as defined in Theorem
\ref{thm:main theorem hausdorff dimension}. Then, for almost every
$\theta\in\Theta$, if $\underline{c}_{f}>1$, then 
\[
D^{f,\theta}=D_{\liminf}^{f,\theta}=D_{\limsup}^{f,\theta}=\emptyset;
\]
if $\underline{c}_{f}\leq1<\bar{c}_{f}$, then 
\[
D^{f,\theta}=D_{\liminf}^{f,\theta}=\emptyset\;\text{ and }\;\dim_{\mathcal{H}}\left(D_{\limsup}^{f,\theta}\right)\leq\nu\left(1-\underline{c}_{f}\right);
\]
if $\bar{c}_{f}\leq1$, then 
\[
\dim_{\mathcal{H}}\left(D^{f,\theta}\right)\leq\dim_{\mathcal{H}}\left(D_{\liminf}^{f,\theta}\right)\leq\nu\left(1-\bar{c}_{f}\right).
\]

Furthermore, suppose that $\left\{ r_{m}:m\geq1\right\} \subseteq(0,1]$
is a sequence such that $r_{m}\searrow0$ as $m\nearrow\infty$ and
\[
\limsup_{m\nearrow\infty}\frac{\Sigma_{r_{m}}^{f}}{-\ln r_{m}}=:c\in\left[0,\infty\right],
\]
and $SD^{f,\theta}$ is the sequential $f-$steep point set associated
with $\left\{ r_{m}:m\geq1\right\} $. Then for $\mathcal{W}-$almost
every $\theta$, if $c>1$, then 
\[
SD^{f,\theta}=\emptyset;
\]
if $c\leq1$, then 
\[
\dim_{\mathcal{H}}\left(SD^{f,\theta}\right)\leq\nu\left(1-c\right).
\]
\end{prop}

\begin{proof}
We will first prove the results concerning $D_{\liminf}^{f,\theta}$,
from which the claims about $D^{f,\theta}$ follow (\ref{eq:sets relation 1}),
and the rest of the statement can be proved by similar arguments with
minor changes. Assume $\bar{c}_{f}>0$. Otherwise the inequality on
$\dim_{\mathcal{H}}\left(D_{\liminf}^{f,\theta}\right)$ is satisfied
trivially. For each $n\geq0$, consider a finite lattice partition
of $\overline{S\left(O,1\right)}$ with cell size $2\cdot2^{-n^{2}}$
(i.e., the length, under the Euclidean metric, of each side of the
cell is $2\cdot2^{-n^{2}}$). Let $\left\{ x_{j}^{\left(n\right)}:\,j=1,\cdots,J_{n}\right\} $
be the collection of the lattice cell centers where $J_{n}=2^{\nu n^{2}}$
is the total number of the cells. Fix $c_{f}^{\prime}\in(0,\bar{c}_{f})$
and let $c_{f}^{\prime}$ be arbitrarily close to $\bar{c}_{f}$.
If $\bar{c}_{f}=\infty$, then we take $c_{f}^{\prime}$ to be arbitrarily
large. There exists a sequence $\left\{ s_{k}:k\geq1\right\} \subseteq\left(0,1\right)$
such that $s_{k}\searrow0$ as $k\nearrow\infty$, and 
\[
\Sigma_{s_{k}}^{f}>c_{f}^{\prime}\left(-\ln s_{k}\right)\mbox{ for all }k\geq1.
\]
For each $k\geq1$, set $n_{k}$ to be the unique positive integer
such that 
\[
2^{-n_{k}^{2}}<s_{k}\leq2^{-\left(n_{k}-1\right)^{2}}.
\]
According to (\ref{eq:modulus on X/Sigma}), for $\mathcal{W}-$almost
every $\theta\in\Theta$, there is an integer $N_{\theta}$ such that
for every $n\geq N_{\theta}$ and every $j=1,\cdots,J_{n+1}$
\begin{equation}
\sup_{\left(y,t\right)\in\overline{S\left(x_{j}^{(n+1)},2^{-\left(n+1\right)^{2}}\right)}\times\left[2^{-n^{2}},2^{-\left(n-1\right)^{2}}\right]}\,\left|\frac{X_{t}^{f,\theta}\left(y\right)}{\Sigma_{t}^{f}}-\frac{X_{t}^{f,\theta}\left(x_{j}^{(n+1)}\right)}{\Sigma_{t}^{f}}\right|\leq2^{-\frac{n}{8}}.\label{modulus estimate used in proof of upper bound}
\end{equation}
If $y_{0}\in D_{\liminf}^{f,\theta}$, then 
\[
\liminf_{k\nearrow\infty}\frac{X_{s_{k}}^{f,\theta}\left(y_{0}\right)}{\Sigma_{s_{k}}^{f}}\geq\sqrt{2\nu},
\]
and hence (\ref{modulus estimate used in proof of upper bound}) guarantees
that, with $\mathcal{W}-$probability one, for any $a>0$ arbitrarily
small and $k$ sufficiently large, 
\[
\frac{X_{s_{k}}^{f,\theta}\left(x_{j}^{(n_{k}+1)}\right)}{\Sigma_{s_{k}}^{f}}>\left(1-a\right)\sqrt{2\nu},\quad\left(\dagger\right)
\]
where $x_{j}^{(n_{k}+1)}$ is the center of the lattice cell (at the
$\left(n_{k}+1\right)$st level) where $y_{0}$ lies, i.e., $y_{0}\in\overline{S\left(x_{j}^{\left(n_{k}+1\right)},2^{-\left(n_{k}+1\right)^{2}}\right)}$.
Equivalently, if we denote by $\mathcal{J}_{k}^{\theta}$ the set
of $x_{j}^{(n_{k}+1)},$ $j=1,\cdots,J_{n_{k}+1}$, such that $(\dagger)$
holds, then $\mathcal{W}-$almost surely 
\begin{equation}
D_{\liminf}^{f,\theta}\subseteq\bigcup_{K\geq1}\,\bigcap_{k\geq K}\,\bigcup_{j=1}^{J_{n_{k}+1}}\left\{ \overline{S\left(x_{j}^{(n_{k}+1)},2^{-\left(n_{k}+1\right)^{2}}\right)}:\,x_{j}^{(n_{k}+1)}\in\mathcal{J}_{k}^{\theta}\right\} .\label{eq:covering of steep point set}
\end{equation}
Meanwhile, for all sufficiently large $k$'s,
\begin{equation}
\begin{split}\mathcal{W}\left(x_{j}^{(n_{k}+1)}\in\mathcal{J}_{k}^{\theta}\right) & \leq C\exp\left[-\nu\Sigma_{s_{k}}^{f}\left(1-a\right)^{2}\right]\leq Cs_{k}^{\nu c_{f}^{\prime}\left(1-a\right)^{2}}.\end{split}
\label{eq:prob of cell center ineq upper bound}
\end{equation}

If $\bar{c}_{f}\in(1,\infty]$, by choosing $c_{f}^{\prime}$ sufficiently
close to $\bar{c}_{f}$ and $a$ sufficiently small, one can always
make
\[
c_{f}^{\prime}\left(1-a\right)^{2}>1+a.
\]
Therefore, (\ref{eq:prob of cell center ineq upper bound}) implies
that for all sufficiently large $k$'s, 
\[
\begin{split}\mathcal{W}\left(D_{\liminf}^{f,\theta}\neq\emptyset\right) & \leq\mathcal{W}\left(\bigcup_{K\geq1}\,\bigcap_{k\geq K}\,\bigcup_{j=1}^{J_{n_{k}+1}}\left\{ x_{j}^{\left(n_{k}+1\right)}\in\mathcal{J}_{k}^{\theta}\right\} \right)\\
 & \leq\sum_{K\geq1}\limsup_{K\leq k\nearrow\infty}J_{n_{k}+1}\mathcal{W}\left(x_{j}^{\left(n_{k}+1\right)}\in\mathcal{J}_{k}^{\theta}\right)\\
 & \leq C\sum_{K\geq1}\limsup_{K\leq k\nearrow\infty}2^{\nu\left(n_{k}+1\right)^{2}}s_{k}^{\nu\left(1+a\right)}=0.
\end{split}
\]
That is, $D_{\liminf}^{f,\theta}=\emptyset$ with $\mathcal{W}-$probability
one when $\bar{c}_{f}>1$. 

Next, assume that $\bar{c}_{f}\in(0,1]$. Note that for $\mathcal{W}-$almost
every $\theta\in\Theta$, the right hand side of (\ref{eq:covering of steep point set})
forms a covering of $D_{\liminf}^{f,\theta}$, and the diameter (under
the Euclidean metric) of $\overline{S\left(x_{j}^{\left(n_{k}+1\right)},2^{-\left(n_{k}+1\right)^{2}}\right)}$
is $2\sqrt{\nu}2^{-\left(n_{k}+1\right)^{2}}$. Thus, if $\mathcal{H}^{\eta}$
is the Hausdorff$-\eta$ measure for $\eta>0$, then
\[
\begin{split}\mathcal{H}^{\eta}\left(D_{\liminf}^{f,\theta}\right) & \leq\liminf_{k\nearrow\infty}\,\sum_{j=1,\cdots,J_{n_{k}+1},\,x_{j}^{\left(n_{k}+1\right)}\in\mathcal{J}_{k}}\,\left[2\sqrt{\nu}2^{-\left(n_{k}+1\right)^{2}}\right]^{\eta}\\
 & =C_{\eta}\,\liminf_{k\nearrow\infty}\,2^{-\eta\left(n_{k}+1\right)^{2}}\#\left(\mathcal{J}_{k}^{\theta}\right)
\end{split}
\]
for some constant $C_{\eta}>0$. Again, it follows from (\ref{eq:prob of cell center ineq upper bound})
that
\[
\begin{split}\mathbb{E}^{\mathcal{W}}\left[\mathcal{H}^{\eta}\left(D_{\liminf}^{f,\theta}\right)\right] & \leq C_{\eta}\,\liminf_{k\nearrow\infty}\,2^{-\eta\left(n_{k}+1\right)^{2}}\mathbb{E}^{\mathcal{W}}\left[\#\left(\mathcal{J}_{k}^{\theta}\right)\right]\\
 & \leq C_{\eta}\,\liminf_{k\nearrow\infty}\,2^{\left(\nu-\eta\right)\left(n_{k}+1\right)^{2}}\mathcal{W}\left(x_{j}^{\left(n_{k}+1\right)}\in\mathcal{J}_{k}^{\theta}\right)\\
 & \leq C_{\eta}\,\liminf_{k\nearrow\infty}\,2^{\left(\nu-\eta\right)\left(n_{k}+1\right)^{2}-\nu c_{f}^{\prime}\left(1-a\right)^{2}\left(n_{k}-1\right)^{2}}.
\end{split}
\]
Given any $\eta>\nu\left(1-\bar{c}_{f}\right)$, so long as $c_{f}^{\prime}$
is sufficiently close to $\bar{c}_{f}$ and $a$ is sufficiently close
to zero, we can make
\[
\eta>\nu-\nu c_{f}^{\prime}\left(1-a\right)^{2},
\]
in which case $\mathcal{W}-$almost surely $\mathcal{H}^{\eta}\left(D_{\liminf}^{f,\theta}\right)=0$
and hence $\dim_{\mathcal{H}}\left(D_{\liminf}^{f,\theta}\right)\leq\eta$.
Since $\eta$ can be arbitrarily close to $\nu\left(1-\bar{c}_{f}\right)$,
we conclude that, for $\mathcal{W}-$almost every $\theta\in\Theta$,
\[
\dim_{\mathcal{H}}\left(D_{\liminf}^{f,\theta}\right)\leq\nu\left(1-\bar{c}_{f}\right).
\]

To prove the claims on the sequential steep point set $SD^{f,\theta}$,
we follow exactly the same line of arguments as above, replacing $\left\{ s_{k}:k\geq1\right\} $
by $\left\{ r_{m}:m\geq1\right\} $, $\bar{c}_{f}$ by $c$ and $c_{f}^{\prime}$
by $c^{\prime}$ where $c^{\prime}<c$ but is arbitrarily close to
$c$. We will omit the repetitive details and turn our attention to
$D_{\limsup}^{f,\theta}$. 

Again, without out loss of generality, we can assume $\underline{c}_{f}>0$.
Let $c_{f}^{\prime\prime}\in\left(0,\underline{c}_{f}\right)$ be
arbitrarily close to $\underline{c}_{f}$, and $a>0$ be arbitrarily
small. Obviously, $\Sigma_{t}^{f}>c_{f}^{\prime\prime}\left(-\ln t\right)$
for all sufficiently small $t$'s. Meanwhile, for every $\theta$
and $y_{0}\in D_{\limsup}^{f,\theta}$, one can find a sequence $\left\{ u_{k}:k\geq1\right\} \subseteq\left(0,1\right)$
such that $u_{k}\searrow0$ as $k\nearrow\infty$, and 
\[
\frac{X_{u_{k}}^{f,\theta}\left(y_{0}\right)}{\Sigma_{u_{k}}^{f}}>\left(1-\frac{a}{4}\right)\sqrt{2\nu}\mbox{ for all sufficiently large }k\geq1.
\]
Similarly, define $n_{k}$ to be the unique integer such that 
\[
2^{-n_{k}^{2}}<u_{k}\leq2^{-\left(n_{k}-1\right)^{2}}.
\]
Even though, this time, the choice of $\left\{ u_{k}:k\geq1\right\} $
and $\left\{ n_{k}:k\geq1\right\} $ will depend on $\theta$ and
$y_{0}$, we can still make the arguments above work. Namely, notice
that the estimate (\ref{modulus estimate used in proof of upper bound})
still applies, so when $k$ is sufficiently large, 
\[
\frac{X_{u_{k}}^{f,\theta}\left(x_{j}^{\left(n_{k}+1\right)}\right)}{\Sigma_{u_{k}}^{f}}>\left(1-\frac{a}{2}\right)\sqrt{2\nu}
\]
where, again, $x_{j}^{\left(n_{k}+1\right)}$ is the center of the
cell (at the $\left(n_{k}+1\right)$st level) that contains $y_{0}$,
and hence 
\[
\sup_{t\in\left[2^{-n_{k}^{2}},2^{-\left(n_{k}-1\right)^{2}}\right]}\frac{X_{t}^{f,\theta}\left(x_{j}^{\left(n_{k}+1\right)}\right)}{\Sigma_{t}^{f}}>\left(1-\frac{a}{2}\right)\sqrt{2\nu}.\quad\left(\dagger\dagger\right)
\]
For each $n\geq1$, Denote by $\mathcal{K}_{n}^{\theta}$ the set
of $x_{j}^{(n+1)},$ $j=1,\cdots,J_{n+1}$, such that $(\dagger\dagger)$
holds (with $n_{k}$ replaced by $n$). Then, 
\[
D_{\limsup}^{f,\theta}\subseteq\bigcup_{K\geq1}\,\bigcap_{k\geq K}\,\bigcup_{n\geq k}\,\bigcup_{j=1}^{J_{n+1}}\left\{ \overline{S\left(x_{j}^{(n+1)},2^{-\left(n+1\right)^{2}}\right)}:\,x_{j}^{(n+1)}\in\mathcal{K}_{n}^{\theta}\right\} .
\]
It is easy to check, for example, by applying the standard entropy
method to the process
\[
\left\{ \frac{X_{t}^{f,\theta}\left(x_{j}^{\left(n+1\right)}\right)}{\Sigma_{t}^{f}}:\,t\in\left[2^{-n^{2}},2^{-\left(n-1\right)^{2}}\right]\right\} ,
\]
that there is a constant $C>0$ such that, for all $n\geq1$,
\[
\mathbb{E}^{\mathcal{W}}\left[\sup_{t\in\left[2^{-n^{2}},2^{-\left(n-1\right)^{2}}\right]}\frac{X_{t}^{f,\theta}\left(x_{j}^{\left(n+1\right)}\right)}{\Sigma_{t}^{f}}\right]\leq C\left(\Sigma_{2^{-\left(n-1\right)^{2}}}^{f}\right)^{-\frac{1}{2}}
\]
which\footnote{If we identify $t\rightsquigarrow\frac{X_{t}^{f,\theta}\left(x_{j}^{\left(n_{k}+1\right)}\right)}{\Sigma_{t}^{f}}$
with $\tau\rightsquigarrow\frac{B_{\tau}}{\tau}$ where $B_{\tau}$
is the standard Brownian motion, then by formulas on the distribution
of running maximum of a drifted Brownian motion, we can compute this
expectation exactly, but knowing the exact value is not necessary
for our purpose.} tends to zero as $n\nearrow\infty$ according to\textbf{ (d)} in
Definition \ref{def:class C }. Then, by the Borell-TIS inequality
(e.g., $\mathsection2$ of \cite{AT07}), for all sufficiently large
$n$'s, 
\[
\begin{split}\mathcal{W}\left(x_{j}^{(n+1)}\in\mathcal{K}_{n}^{\theta}\right) & =\mathcal{W}\left(\sup_{t\in\left[2^{-n^{2}},2^{-\left(n-1\right)^{2}}\right]}\frac{X_{t}^{f,\theta}\left(x_{j}^{\left(n+1\right)}\right)}{\Sigma_{t}^{f}}>\left(1-\frac{a}{2}\right)\sqrt{2\nu}\right)\\
 & \leq\exp\left\{ -\frac{\Sigma_{2^{-\left(n-1\right)^{2}}}^{f}}{2}\left[\left(1-\frac{a}{2}\right)\sqrt{2\nu}+o\left(1\right)\right]^{2}\right\} \\
 & \leq C2^{-\left(n-1\right)^{2}\nu c_{f}^{\prime\prime}\left(1-a\right)^{2}}.
\end{split}
\]
From here, we proceed in exactly the same way as we did earlier when
proving the upper bound of $\dim_{\mathcal{H}}\left(D_{\liminf}^{f,\theta}\right)$.
Details are omitted.
\end{proof}

\subsection{Lower Bounds}

We now proceed to the lower bounds for the Hausdorff dimension of
the exceptional sets we have studied in the previous subsection. Let
us summarize in the following proposition the estimates we would like
to prove. 
\begin{prop}
\label{prop:lower bound on D^f,theta}Given $f\in\mathcal{C}$, let
$\bar{c}_{f}$ and $\underline{c}_{f}$ be as defined in Theorem \ref{thm:main theorem hausdorff dimension}. 

(i) If $0<\underline{c}_{f}\leq\bar{c}_{f}\leq1$ then $\mathcal{W}-$almost
surely,
\[
\dim_{\mathcal{H}}\left(D_{\limsup}^{f,\theta}\right)\geq\dim_{\mathcal{H}}\left(D_{\liminf}^{f,\theta}\right)\geq\dim_{\mathcal{H}}\left(D^{f,\theta}\right)\geq\nu\left(1-2\bar{c}_{f}+\underline{c}_{f}\right),
\]
which completes the proof of Theorem \ref{thm:main theorem hausdorff dimension}.

(ii) Suppose that there exists a sequence $\left\{ r_{m}:m\geq1\right\} \subseteq(0,1]$
with $r_{m}\searrow0$ as $m\nearrow\infty$ such that
\[
\liminf_{m\nearrow\infty}\frac{\Sigma_{r_{m}}^{f}}{-\ln r_{m}}>0\quad\text{(}\underline{c}_{f}\text{ does not have to be positive though)}.
\]
If $SD^{f,\theta}$ is the sequential $f-$steep point set associated
with $\left\{ r_{m}:m\geq1\right\} $, then $\mathcal{W}-$almost
surely 
\[
\dim_{\mathcal{H}}\left(D_{\limsup}^{f,\theta}\right)\geq\dim_{\mathcal{H}}\left(SD^{f,\theta}\right)\geq\nu\left(1-2\bar{c}_{f}+\underline{c}_{f}\right).
\]

(iii) Suppose that there exists a sequence $\left\{ r_{m}:m\geq1\right\} \subseteq(0,1]$
such that
\begin{equation}
\lim_{m\nearrow\infty}\,\frac{m}{\Sigma_{r_{m}}^{f}}=0\text{ and }\lim_{m\nearrow\infty}\,\frac{\Sigma_{r_{m+1}}^{f}}{\Sigma_{r_{m}}^{f}}=1.\label{eq:assump on r_m}
\end{equation}
If $SD^{f,\theta}$ is the sequential $f-$steep point set associated
with $\left\{ r_{m}:m\geq1\right\} $, then $\mathcal{W}-$almost
surely 
\[
\dim_{\mathcal{H}}\left(SD^{f,\theta}\right)=\dim_{\mathcal{H}}\left(D_{\liminf}^{f,\theta}\right)=\dim_{\mathcal{H}}\left(D^{f,\theta}\right)=\nu\left(1-\bar{c}_{f}\right)
\]
and hence
\[
\dim_{\mathcal{H}}\left(D_{\limsup}^{f,\theta}\right)\geq\nu\left(1-\bar{c}_{f}\right).
\]
\end{prop}

We only give a detailed proof of (i), because (ii) and (iii) can be
derived based on the same proof with minor changes. The main idea
in proving the lower bound in (i) is to create a setting in which
we can apply Frostman's lemma. To this end, we will need to carry
out several steps of preparations. We will start with a ``configuration''
of the problem that is easier to handle. Besides the condition $0<\underline{c}_{f}\leq\bar{c}_{f}\leq1$,
we will also assume that $1+\underline{c}_{f}>2\bar{c}_{f}$ (which
implies that $\bar{c}_{f}<1$) since otherwise the inequalities in
(i) hold trivially. Choose $\tilde{c}\in\left(\bar{c}_{f},1\right)$
and $\utilde{c}\in\left(0,\underline{c}_{f}\right)$ to be sufficiently
close to $\bar{c}_{f}$ and, respectively, $\underline{c}_{f}$, such
that
\begin{equation}
1+\utilde{c}>2\tilde{c}.\label{eq:assump on c_f and delta}
\end{equation}
Just as what we did for the upper bound, we will ``discretize''
the problem by considering the behaviors of the concerned quantities,
e.g., $X_{t}^{f,\theta}$ and $\Sigma_{t}^{f}$, when $t$ varies
along a specific sequence, say, $\left\{ 2^{-n^{2}}:n\geq0\right\} $.
Without loss of generality, we will assume that for all sufficiently
large $n$'s,
\[
\utilde{c}\left(n^{2}\ln2\right)\leq\Sigma_{2^{-n^{2}}}^{f}\leq\tilde{c}\left(n^{2}\ln2\right).
\]
To simplify the notation, we denote for every $\theta\in\Theta$,
$x\in\overline{S\left(O,1\right)}$ and $n\geq1$,
\[
\Delta\theta_{n}\left(x\right):=\bar{\theta}_{2^{-n^{2}}}\left(x\right)-\bar{\theta}_{2^{-\left(n-1\right)^{2}}}\left(x\right),
\]
 
\[
\Delta X_{n}^{f,\theta}\left(x\right):=X_{2^{-n^{2}}}^{f,\theta}\left(x\right)-X_{2^{-\left(n-1\right)^{2}}}^{f,\theta}\left(x\right),
\]
and 
\[
\Delta\Sigma_{n}^{f}:=\Sigma_{2^{-n^{2}}}^{f}-\Sigma_{2^{-\left(n-1\right)^{2}}}^{f}.
\]
Define $P_{x,n}^{f}$ to be the set of $\theta\in\Theta$ such that
\[
\sup_{t\in\left[2^{-n^{2}},2^{-\left(n-1\right)^{2}}\right]}\left|X_{t}^{f,\theta}\left(x\right)-X_{2^{-\left(n-1\right)^{2}}}^{f,\theta}\left(x\right)-\sqrt{2\nu}\left(\Sigma_{t}^{f}-\Sigma_{2^{-\left(n-1\right)^{2}}}^{f}\right)\right|\leq\sqrt{\Delta\Sigma_{n}^{f}},
\]
and set $\Phi_{x,n}^{f}:=\left(\bigcap_{i=1}^{n}P_{x,i}^{f}\right)$. 

The first ``ingredient'' we need is the probability estimate for
$P_{x,n}^{f}$ and $\Phi_{x,n}^{f}$, which can be obtained through
the Cameron-Martin formula.
\begin{lem}
\label{lem:analysis of the probability}For every $n\geq1$, $P_{x,i}^{f}$,
$i=1,\cdots,n$, are mutually independent. Moreover, there is a constant
$p\in\left(0,1\right)$ such that for every $n\geq1$, 
\begin{equation}
e^{-\nu\left(\Delta\Sigma_{n}^{f}\right)-\sqrt{2\nu}\sqrt{\Delta\Sigma_{n}^{f}}}p\leq\mathcal{W}\left(P_{x,n}^{f}\right)\leq e^{-\nu\left(\Delta\Sigma_{n}^{f}\right)+\sqrt{2\nu}\sqrt{\Delta\Sigma_{n}^{f}}}p\label{eq: estimate for P_x,n}
\end{equation}
and hence
\begin{equation}
e^{-\nu\Sigma_{2^{-n^{2}}}^{f}-\sqrt{2\nu}\sqrt{n\Sigma_{2^{-n^{2}}}^{f}}}p^{n}\leq\mathcal{W}\left(\Phi_{x,n}^{f}\right)\leq e^{-\nu\Sigma_{2^{-n^{2}}}^{f}+\sqrt{2\nu}\sqrt{n\Sigma_{2^{-n^{2}}}^{f}}}p^{n}.\label{eq:lower bound prob of Phi_x,n}
\end{equation}
\end{lem}

\begin{proof}
For each $n\geq1$, the independence of $P_{x,j}^{f}$, $j=1,\cdots,n$,
is obvious from the fact that $X_{t}^{f,\theta}\left(x\right)$ has
independent increments as $t$ decreases. We only need to show (\ref{eq: estimate for P_x,n}),
since (\ref{eq:lower bound prob of Phi_x,n}) follows trivially from
(\ref{eq: estimate for P_x,n}), the Cauchy inequality and the independence.
For each $n\geq1$ and $x\in\overline{S\left(O,1\right)}$, assume
that $h$ is the unique element in $H$ such that the corresponding
Paley-Wiener integral is given by $\mathcal{I}\left(h\right)\left(\theta\right)=-\Delta X_{n}^{f,\theta}\left(x\right)$.
Then, for $t\in\left[2^{-n^{2}},2^{-\left(n-1\right)^{2}}\right]$,
\[
\left(h,h_{\bar{\sigma}_{t}^{x}}\right)_{H}=\mathbb{E}^{\mathcal{W}}\left[\mathcal{I}\left(h\right)\left(\theta\right)\bar{\theta}_{t}\left(x\right)\right]=-\int_{2^{-\left(n-1\right)^{2}}}^{t}f\left(s\right)dG\left(s\right),
\]
which implies that 
\[
\begin{split}X_{t}^{f,\theta+\sqrt{2\nu}h}-X_{2^{-\left(n-1\right)^{2}}}^{f,\theta+\sqrt{2\nu}h} & =X_{t}^{f,\theta}-X_{2^{-\left(n-1\right)^{2}}}^{f,\theta}-\sqrt{2\nu}\int_{2^{-\left(n-1\right)^{2}}}^{t}f^{2}\left(s\right)dG\left(s\right)\\
 & =X_{t}^{f,\theta}-X_{2^{-\left(n-1\right)^{2}}}^{f,\theta}-\sqrt{2\nu}\left(\Sigma_{t}^{f}-\Sigma_{2^{-\left(n-1\right)^{2}}}^{f}\right).
\end{split}
\]
Moreover, 
\[
\begin{split}\left\Vert h\right\Vert _{H}^{2} & =\mathbb{E}^{\mathcal{W}}\left[\left(\mathcal{I}\left(h\right)\left(\theta\right)\right)^{2}\right]=\mathbb{E}^{\mathcal{W}}\left[\left(-\Delta X_{n}^{f,\theta}\left(x\right)\right)^{2}\right]=\Delta\Sigma_{n}^{f}\end{split}
.
\]
Therefore, by the Cameron-Martin formula (e.g., $\mathsection8$ in
\cite{probability}), we get that\footnote{For an integrable random variable $Z$ on $\Theta$ and a measurable
set $A\subseteq\Theta$, ``$\mathbb{E}^{\mathcal{W}}\left[Z;A\right]$''
refers to $\int_{A}Zd\mathcal{W}$.}
\[
\begin{split}\mathcal{W}\left(P_{x,n}^{f}\right) & =\mathcal{W}\left(\sup_{t\in\left[2^{-n^{2}},2^{-\left(n-1\right)^{2}}\right]}\left|X_{t}^{f,\theta+\sqrt{2\nu}h}-X_{2^{-\left(n-1\right)^{2}}}^{f,\theta+\sqrt{2\nu}h}\right|\leq\sqrt{\Delta\Sigma_{n}^{f}}\right)\\
 & =\mathbb{E}^{\mathcal{W}}\left[\exp\left(\sqrt{2\nu}\mathcal{I}\left(h\right)\left(\theta\right)-\nu\left\Vert h\right\Vert _{H}^{2}\right);\sup_{t\in\left[2^{-n^{2}},2^{-\left(n-1\right)^{2}}\right]}\left|X_{t}^{f,\theta}-X_{2^{-\left(n-1\right)^{2}}}^{f,\theta}\right|\leq\sqrt{\Delta\Sigma_{n}^{f}}\right]\\
 & =e^{-\nu\Delta\Sigma_{n}^{f}}\mathbb{E}^{\mathcal{W}}\left[e^{\sqrt{2\nu}\mathcal{I}\left(h\right)\left(\theta\right)};\sup_{t\in\left[2^{-n^{2}},2^{-\left(n-1\right)^{2}}\right]}\left|X_{t}^{f,\theta}-X_{2^{-\left(n-1\right)^{2}}}^{f,\theta}\right|\leq\sqrt{\Delta\Sigma_{n}^{f}}\right].
\end{split}
\]
When the constraint in the right hand side above is satisfied, 
\[
\mathcal{I}\left(h\right)\left(\theta\right)\in\left[-\sqrt{\Delta\Sigma_{n}^{f}},\sqrt{\Delta\Sigma_{n}^{f}}\right].
\]
Meanwhile, the distribution of $\left\{ X_{t}^{f,\theta}-X_{2^{-\left(n-1\right)^{2}}}^{f,\theta}:t\in\left[2^{-n^{2}},2^{-\left(n-1\right)^{2}}\right]\right\} $
is that of a standard Brownian motion $\left\{ B_{\tau}:0\leq\tau\leq T\right\} $
on a generic probability space, say, $\left(\Omega,\mathcal{F},\mathbb{P}\right)$,
with
\[
\tau:=\Sigma_{t}^{f}-\Sigma_{2^{-\left(n-1\right)^{2}}}^{f}\mbox{ and }T:=\Delta\Sigma_{n}^{f}.
\]
Thus, 
\[
\begin{split}\mathcal{W}\left(\sup_{t\in\left[2^{-n^{2}},2^{-\left(n-1\right)^{2}}\right]}\left|X_{t}^{f,\theta}-X_{2^{-\left(n-1\right)^{2}}}^{f,\theta}\right|\leq\sqrt{\Delta\Sigma_{n}^{f}}\right) & =\mathbb{P}\left(\sup_{\tau\in\left[0,T\right]}\left|B_{\tau}\right|\leq\sqrt{T}\right)\\
 & =\mathbb{P}\left(\sup_{\tau\in\left[0,1\right]}\left|B_{\tau}\right|\leq1\right):=p\in\left(0,1\right).
\end{split}
\]
We have finished the proof of (\ref{eq: estimate for P_x,n})
\end{proof}
The events $P_{x,n}^{f}$ and $\Phi_{x,n}^{f}$ concerning the discrete
family $\left\{ \Delta X_{n}^{f,\theta}\left(x\right):n\geq0\right\} $
help us ``design'' a specific collection of $f-$steep points, i.e.,
a subset of $D^{f,\theta}$, and whose Hausdorff measure or Hausdorff
dimension is ``convenient'' to study. Below we explain how to construct
such a subset of $D^{f,\theta}$, which is the second ``ingredient''
of the main proof. 

For every $n\ge0$, again we consider the lattice partition of $\overline{S\left(O,1\right)}$
with cell size $2^{-n^{2}}$ with 
\[
\left\{ x_{j}^{\left(n\right)}:\,j=1,\cdots,J_{n}:=2^{\nu n^{2}}\right\} 
\]
being the collection of all the cell centers. For every $\theta\in\Theta$,
we set 
\[
\Xi_{n}^{f,\theta}:=\left\{ x_{j}^{(n)}:\;1\leq j\leq J_{n},\,\theta\in\Phi_{x_{j}^{(n)},n}^{f}\right\} .
\]
\begin{lem}
\label{lem: subset of D^f,theta}Let $f\in\mathcal{C}$. For $\mathcal{W}-$almost
every $\theta\in\Theta$, 
\begin{equation}
D^{f,\theta}\supseteq\Upsilon^{f,\theta}:=\bigcap_{k\geq1}\,\overline{\bigcup_{n\geq k}\,\bigcup_{x\in\Xi_{n}^{f,\theta}}\,S\left(x,2^{-n^{2}}\right)}.\label{eq:subset of S}
\end{equation}
\end{lem}

\begin{proof}
Let $y$ be an element from the right hand side of (\ref{eq:subset of S}).
It is easy to see that one can always find a subsequence $\left\{ n_{j}:j\geq1\right\} \subseteq\mathbb{N}$
with $n_{j}\nearrow\infty$ as $j\nearrow\infty$ and a sequence of
cell centers $\left\{ x^{(n_{j})}\in\Xi_{n_{j}}^{f,\theta}:\,j\geq1\right\} $
such that $\lim_{j\nearrow\infty}\left|y-x^{(n_{j})}\right|=0$. For
any $t\in(0,1]$, assume that $\ell\in\mathbb{N}$ is the unique integer
such that $2^{-\ell^{2}}\leq t<2^{-\left(\ell-1\right)^{2}}$. Then
we have that for every $n_{j}\geq\ell$, since $x^{(n_{j})}\in\Xi_{n_{j}}^{f,\theta}$,
\begin{equation}
\left|\frac{X_{t}^{f,\theta}\left(x^{\left(n_{j}\right)}\right)}{\Sigma_{t}^{f}}-\sqrt{2\nu}\right|\leq\frac{\sum_{i=1}^{\ell}\sqrt{\Delta\Sigma_{i}^{f}}}{\Sigma_{2^{-\left(\ell-1\right)^{2}}}^{f}}\leq\frac{\sqrt{\ell\Sigma_{2^{-\ell^{2}}}^{f}}}{\Sigma_{2^{-\left(\ell-1\right)^{2}}}^{f}}\leq\frac{\sqrt{\tilde{c}\ln2}\,\ell^{3/2}}{\utilde{c}\left(\ell-1\right)^{2}\ln2},\label{eq:an estimate used in finding subset of steep point}
\end{equation}
which can be arbitrarily small when $\ell$ is sufficiently large,
or equivalently, when $t$ is sufficiently small. Moreover, with $t$
and $\theta$ fixed, the function $y\rightsquigarrow\frac{X_{t}^{f,\theta}\left(y\right)}{\Sigma_{t}^{f}}$
is continuous and hence absolutely continuous on $\overline{S\left(O,1\right)}$,
so one can also make 
\[
\left|\frac{X_{t}^{f,\theta}\left(x^{(n_{j})}\right)}{\Sigma_{t}^{f}}-\frac{X_{t}^{f,\theta}\left(y\right)}{\Sigma_{t}^{f}}\right|
\]
arbitrarily small by choosing sufficiently large $n_{j}$. This implies
that $y\in D^{f,\theta}$. 
\end{proof}
Now we are ready to embark on the proof of Proposition \ref{prop:lower bound on D^f,theta}(i).
Briefly speaking, our goal is to apply Frostman's lemma to bound $\dim_{\mathcal{H}}\left(\Upsilon^{f,\theta}\right)$
from below, which requires us to find a non-trivial Borel measure
$\mu^{f,\theta}$ supported on $\Upsilon^{f,\theta}$ and to study
the $\alpha-$energy of $\mu^{f,\theta}$ for certain $\alpha>0$.
We will achieve our goal following two steps: first consider a naturally
chosen family of Borel measures $\mu_{n}^{f,\theta}$ supported on
$\overline{S\left(O,1\right)}$ for $n\geq1$, and verify that $\left\{ \mu_{n}^{f,\theta}:n\geq1\right\} $
is ``nice'' in the sense that $\mu_{n}^{f,\theta}$'s have uniformly
bounded first and second moments in their total mass, as well as uniformly
bounded expectation of $\alpha-$energy for certain $\alpha>0$; next
we combine a ``compactness'' argument and the Hewitt-Savage 0-1
law to extract, for $\mathcal{W}-$almost every $\theta$, a limit
measure $\mu^{f,\theta}$, and confirm that $\mu^{f,\theta}$ inherits
the nice properties from $\left\{ \mu_{n}^{f,\theta}:n\geq1\right\} $
in the sense that $\mu^{f,\theta}$ is a non-trivial measure supported
on $\Upsilon^{f,\theta}$ with finite $\alpha-$energy for $\alpha$
in a proper range.\\

\noindent \emph{Proof of Proposition \ref{prop:lower bound on D^f,theta}(i):}
We consider a family of random finite measures on $\overline{S\left(O,1\right)}$:
for each $n\geq1$ and $\theta\in\Theta$, define the measure 
\begin{equation}
\forall B\in\mathcal{B}\left(\overline{S\left(O,1\right)}\right),\quad\mu_{n}^{f,\theta}\left(B\right):=\frac{1}{J_{n}}\sum_{j=1}^{J_{n}}\frac{\mathbb{I}_{\Xi_{n}^{f,\theta}}\left(x_{j}^{(n)}\right)}{\mathcal{W}\left(\Phi_{x_{j}^{(n)},n}^{f}\right)}\frac{\mbox{vol}\left(B\cap S\left(x_{j}^{(n)},2^{-n^{2}}\right)\right)}{\mbox{vol}\left(S\left(x_{j}^{(n)},2^{-n^{2}}\right)\right)},\label{eq:def of mu_n_theta}
\end{equation}
where ``vol'' refers to the volume under the Lebesgue measure on
$\mathbb{R}^{\nu}$. It is clear that 
\begin{equation}
\begin{split}\forall n\geq1,\quad & \begin{split}\mathbb{E}^{\mathcal{W}}\left[\mu_{n}^{f,\theta}\left(\overline{S\left(O,1\right)}\right)\right] & =1\end{split}
.\end{split}
\label{eq:1st moment of total mass under mu_theta}
\end{equation}
Besides the uniformity in the first moment of $\mu_{n}^{f,\theta}\left(\overline{S\left(O,1\right)}\right)$,
our next goal is to show that its second moment is also bounded in
$n$, i.e., 
\begin{equation}
\sup_{n\geq1}\mathbb{E}^{\mathcal{W}}\left[\left(\mu_{n}^{f,\theta}\left(\overline{S\left(O,1\right)}\right)\right)^{2}\right]<\infty.\label{eq:bounded of 2nd of moment of mu_n}
\end{equation}
To this end, we write 
\begin{equation}
\begin{split}\mathbb{E}^{\mathcal{W}}\left[\left(\mu_{n}^{f,\theta}\left(\overline{S\left(O,1\right)}\right)\right)^{2}\right] & =\frac{1}{J_{n}^{2}}\,\sum_{j,k=1}^{J_{n}}\,\frac{\mathcal{W}\left(\Phi_{x_{j}^{(n)},n}^{f}\bigcap\Phi_{x_{k}^{(n)},n}^{f}\right)}{\mathcal{W}\left(\Phi_{x_{j}^{(n)},n}^{f}\right)\mathcal{W}\left(\Phi_{x_{k}^{(n)},n}^{f}\right)}\end{split}
.\label{eq:2nd moment of the total mass under mu_theta}
\end{equation}
By ( \ref{eq:lower bound prob of Phi_x,n}), when $j=k$,
\begin{equation}
\frac{\mathcal{W}\left(\Phi_{x_{j}^{(n)},n}^{f}\bigcap\Phi_{x_{k}^{(n)},n}^{f}\right)}{\mathcal{W}\left(\Phi_{x_{j}^{(n)},n}^{f}\right)\mathcal{W}\left(\Phi_{x_{k}^{(n)},n}^{f}\right)}=\frac{1}{\mathcal{W}\left(\Phi_{x_{j}^{(n)},n}^{f}\right)}\leq\exp\left(\nu\tilde{c}n^{2}\ln2+Cn^{3/2}\right).\label{key estimate in lower bound diag}
\end{equation}
When $2\cdot2^{-\left(i+1\right)^{2}}\leq\left|x_{j}^{(n)}-x_{k}^{(n)}\right|<2\cdot2^{-i^{2}}$
for some $i$ where $0\leq i\le n-1$ (without loss of generality,
we assume that $i$ is large), since the family 
\[
\left\{ \Delta X_{l}^{f,\theta}\left(x_{j}^{(n)}\right),\,\Delta X_{l^{\prime}}^{f,\theta}\left(x_{k}^{(n)}\right):\,1\leq l\leq i-1,\,i+2\leq l\leq n,\,i+2\leq l^{\prime}\leq n\right\} 
\]
is independent, if follows from (\ref{eq: estimate for P_x,n}) and
(\ref{eq:lower bound prob of Phi_x,n}) that 
\begin{equation}
\begin{split}\frac{\mathcal{W}\left(\Phi_{x_{j}^{(n)},n}^{f}\bigcap\Phi_{x_{k}^{(n)},n}^{f}\right)}{\mathcal{W}\left(\Phi_{x_{j}^{(n)},n}^{f}\right)\mathcal{W}\left(\Phi_{x_{k}^{(n)},n}^{f}\right)} & \leq\frac{\mathcal{W}\left(\left(\bigcap_{l=1}^{i-1}P_{x_{j}^{(n)},l}^{f}\right)\bigcap\left(\bigcap_{l=i+2}^{n}P_{x_{j}^{(n)},l}^{f}\right)\bigcap\left(\bigcap_{l^{\prime}=i+2}^{n}P_{x_{k}^{(n)},l^{\prime}}^{f}\right)\right)}{\mathcal{W}\left(\Phi_{x_{j}^{(n)},n}^{f}\right)\mathcal{W}\left(\Phi_{x_{k}^{(n)},n}^{f}\right)}\\
 & \leq\frac{1}{\mathcal{W}\left(\Phi_{x_{k}^{(n)},i+1}^{f}\right)\mathcal{W}\left(P_{x_{j}^{(n)},i}^{f}\right)\mathcal{W}\left(P_{x_{j}^{(n)},i+1}^{f}\right)}\\
 & \leq\exp\left[\nu\tilde{c}\left(1+i\right)^{2}\ln2+\nu\left(\Sigma_{2^{-\left(i+1\right)^{2}}}^{f}-\Sigma_{2^{-\left(i-1\right)^{2}}}^{f}\right)+Ci^{3/2}\right]\\
 & \le\exp\left[\nu\left(2\tilde{c}-\utilde{c}\right)i^{2}\ln2+Ci^{3/2}\right].
\end{split}
\label{eq:key estimate in lower bound off diag}
\end{equation}
Combining (\ref{key estimate in lower bound diag}) and (\ref{eq:key estimate in lower bound off diag})
yields that the right hand side of (\ref{eq:2nd moment of the total mass under mu_theta})
is no greater than
\[
\begin{split} & \frac{1}{J_{n}^{2}}\,\sum_{j=1}^{J_{n}}2^{\nu\tilde{c}n^{2}+Cn^{3/2}}+\frac{1}{J_{n}^{2}}\,\sum_{i=0}^{n-1}\sum_{\left\{ (j,k):2\cdot2^{-\left(i+1\right)^{2}}\leq\left|x_{j}^{(n)}-x_{k}^{(n)}\right|<2\cdot2^{-i^{2}}\right\} }2^{\nu\left(2\tilde{c}-\utilde{c}\right)i^{2}+Ci^{3/2}}\\
\leq & 2^{-\nu\left(1-\tilde{c}\right)n^{2}+Cn^{3/2}}+\sum_{i=0}^{n-1}2^{-\nu\left(1-2\tilde{c}+\utilde{c}\right)i^{2}+Ci^{3/2}},
\end{split}
\]
which is uniformly bounded in $n$ under the assumption (\ref{eq:assump on c_f and delta}).
So we have proved (\ref{eq:bounded of 2nd of moment of mu_n}).

Next, we turn our attention to the $\alpha-$energy, $\alpha>0$,
of the measure $\mu_{n}^{f,\theta}$ for every $\theta\in\Theta$
and every $n\geq1$, i.e., 
\[
I_{\alpha}\left(\mu_{n}^{f,\theta}\right):=\int_{\overline{S\left(O,1\right)}}\int_{\overline{S\left(O,1\right)}}\left|y-w\right|^{-\alpha}\mu_{n}^{f,\theta}\left(dy\right)\mu_{n}^{f,\theta}\left(dw\right).
\]
We need to verify that, whenever $\alpha$ is smaller than a critical
value which will be determined later, $\mu_{n}^{f,\theta}$ has uniformly
bounded expected $\alpha-$energy, i.e.,
\begin{equation}
\sup_{n\geq1}\mathbb{E}^{\mathcal{W}}\left[I_{\alpha}\left(\mu_{n}^{f,\theta}\right)\right]<\infty.\label{eq:boundedness of alpha energy}
\end{equation}
By (\ref{eq:def of mu_n_theta}), $\mathbb{E}^{\mathcal{W}}\left[I_{\alpha}\left(\mu_{n}^{f,\theta}\right)\right]$
is equal to 
\begin{equation}
\frac{1}{J_{n}^{2}}\sum_{j,k=1}^{J_{n}}\frac{\mathcal{W}\left(\Phi_{x_{j}^{(n)},n}^{f}\bigcap\Phi_{x_{k}^{(n)},n}^{f}\right)}{\mathcal{W}\left(\Phi_{x_{j}^{(n)},n}^{f}\right)\mathcal{W}\left(\Phi_{x_{k}^{(n)},n}^{f}\right)}\frac{\int_{\overline{S\left(x_{j}^{(n)},2^{-n^{2}}\right)}}\int_{\overline{S\left(x_{k}^{(n)},2^{-n^{2}}\right)}}\left|y-w\right|^{-\alpha}dydw}{\mbox{vol}\left(S\left(x_{j}^{(n)},2^{-n^{2}}\right)\right)\mbox{vol}\left(S\left(x_{k}^{(n)},2^{-n^{2}}\right)\right)}.\label{eq:expection of energy}
\end{equation}
Assume for now $\alpha<\nu\left(1-2\tilde{c}+\utilde{c}\right)$.
Obviously, for the diagonal terms in the summation in (\ref{eq:expection of energy}),
i.e., when $j=k$, we have that
\[
\frac{\int_{\overline{S\left(x_{j}^{(n)},2^{-n^{2}}\right)}}\int_{\overline{S\left(x_{k}^{(n)},2^{-n^{2}}\right)}}\left|y-w\right|^{-\alpha}dydw}{\mbox{vol}\left(S\left(x_{j}^{(n)},2^{-n^{2}}\right)\right)\mbox{vol}\left(S\left(x_{k}^{(n)},2^{-n^{2}}\right)\right)}=C2^{\alpha n^{2}}.
\]
If $j\neq k$, when $\left|x_{j}^{(n)}-x_{k}^{(n)}\right|\leq4\sqrt{\nu}2^{-n^{2}}$,
by possibly enlarging $C$, we can make
\[
\frac{\int_{\overline{S\left(x_{j}^{(n)},2^{-n^{2}}\right)}}\int_{\overline{S\left(x_{k}^{(n)},2^{-n^{2}}\right)}}\left|y-w\right|^{-\alpha}dydw}{\mbox{vol}\left(S\left(x_{j}^{(n)},2^{-n^{2}}\right)\right)\mbox{vol}\left(S\left(x_{k}^{(n)},2^{-n^{2}}\right)\right)}\leq C2^{\alpha n^{2}}\leq C\left|x_{j}^{(n)}-x_{k}^{(n)}\right|^{-\alpha};
\]
when $\left|x_{j}^{(n)}-x_{k}^{(n)}\right|>4\sqrt{\nu}2^{-n^{2}}$,
since for every $y^{\prime},w^{\prime}\in\overline{S\left(O,2^{-n^{2}}\right)}$,
\[
\left|y^{\prime}-w^{\prime}\right|\leq2\sqrt{\nu}2^{-n^{2}}\leq\frac{1}{2}\left|x_{j}^{(n)}-x_{k}^{(n)}\right|,
\]
we have that
\[
\left|x_{j}^{\left(n\right)}-x_{k}^{\left(n\right)}-(y^{\prime}-w^{\prime})\right|\geq\frac{1}{2}\left|x_{j}^{(n)}-x_{k}^{(n)}\right|,
\]
and hence
\[
\begin{split}\frac{\int_{\overline{S\left(x_{j}^{(n)},2^{-n^{2}}\right)}}\int_{\overline{S\left(x_{k}^{(n)},2^{-n^{2}}\right)}}\left|y-w\right|^{-\alpha}dydw}{\mbox{vol}\left(S\left(x_{j}^{(n)},2^{-n^{2}}\right)\right)\mbox{vol}\left(S\left(x_{k}^{(n)},2^{-n^{2}}\right)\right)} & =\frac{\int_{\overline{S\left(O,2^{-n^{2}}\right)}}\int_{\overline{S\left(O,2^{-n^{2}}\right)}}\left|x_{j}^{\left(n\right)}-x_{k}^{\left(n\right)}-(y^{\prime}-w^{\prime})\right|^{-\alpha}dy^{\prime}dw^{\prime}}{\mbox{vol}\left(S\left(x_{j}^{(n)},2^{-n^{2}}\right)\right)\mbox{vol}\left(S\left(x_{k}^{(n)},2^{-n^{2}}\right)\right)}\\
 & \leq C\left|x_{j}^{\left(n\right)}-x_{k}^{\left(n\right)}\right|^{-\alpha}.
\end{split}
\]
Also, recall from (\ref{key estimate in lower bound diag}) and (\ref{eq:key estimate in lower bound off diag})
that, when $j=k$, 
\[
\frac{\mathcal{W}\left(\Phi_{x_{j}^{(n)},n}^{f}\bigcap\Phi_{x_{k}^{(n)},n}^{f}\right)}{\mathcal{W}\left(\Phi_{x_{j}^{(n)},n}^{f}\right)\mathcal{W}\left(\Phi_{x_{k}^{(n)},n}^{f}\right)}\leq\exp\left(\nu\tilde{c}n^{2}\ln2+Cn^{3/2}\right);
\]
when $2\cdot2^{-\left(i+1\right)^{2}}\leq\left|x_{j}^{(n)}-x_{k}^{(n)}\right|<2\cdot2^{-i^{2}}$
for some $i=0,1,\cdots,n-1$,
\[
\begin{split}\frac{\mathcal{W}\left(\Phi_{x_{j}^{(n)},n}^{f}\bigcap\Phi_{x_{k}^{(n)},n}^{f}\right)}{\mathcal{W}\left(\Phi_{x_{j}^{(n)},n}^{f}\right)\mathcal{W}\left(\Phi_{x_{k}^{(n)},n}^{f}\right)} & \leq\exp\left[\nu\left(2\tilde{c}-\utilde{c}\right)i^{2}\ln2+Ci^{3/2}\right]\\
 & \leq\exp\left[-\nu\left(2\tilde{c}-\utilde{c}\right)\ln\left|x_{j}^{(n)}-x_{k}^{(n)}\right|+o\left(-\ln\left|x_{j}^{(n)}-x_{k}^{(n)}\right|\right)\right].
\end{split}
\]
As a result, the summation in (\ref{eq:expection of energy}) is no
greater than a constant multiple of
\[
\begin{split} & \frac{1}{J_{n}^{2}}\sum_{j=1}^{J_{n}}2^{\left(\alpha+\nu\tilde{c}\right)n^{2}+Cn^{3/2}}+\frac{1}{J_{n}^{2}}\sum_{\left\{ 1\leq j,k\leq J_{n},\,j\neq k\right\} }\left|x_{j}^{(n)}-x_{k}^{(n)}\right|^{-\alpha-\nu\left(2\tilde{c}-\utilde{c}\right)-o(1)}\\
\leq & 2^{-\left(\nu-\nu\tilde{c}-\alpha\right)n^{2}\ln2+Cn^{3/2}}+\frac{1}{J_{n}^{2}}\sum_{\left\{ 1\leq j,k\leq J_{n},\,j\neq k\right\} }\left|x_{j}^{(n)}-x_{k}^{(n)}\right|^{-\alpha-\nu\left(2\tilde{c}-\utilde{c}\right)-o(1)}\\
\longrightarrow & \iint_{\overline{S\left(O,1\right)}\times\overline{S\left(O,1\right)}}\left|x-y\right|^{-\alpha-\nu\left(2\tilde{c}-\utilde{c}\right)}dxdy<\infty\text{ as }n\nearrow\infty.
\end{split}
\]
Therefore, we have shown that (\ref{eq:boundedness of alpha energy})
holds whenever $\alpha<\nu\left(1-2\tilde{c}+\utilde{c}\right)$.

Now fix any $\alpha\in\left(0,\nu\left(1-2\tilde{c}+\utilde{c}\right)\right)$.
After showing (\ref{eq:bounded of 2nd of moment of mu_n}) and (\ref{eq:boundedness of alpha energy}),
we have two positive real numbers
\[
A_{1}:=\sup_{n\geq1}\,\mathbb{E}^{\mathcal{W}}\left[\left(\mu_{n}^{f,\theta}\left(\overline{S\left(O,1\right)}\right)\right)^{2}\right]\mbox{ and }A_{2}:=\sup_{n\geq1}\,\mathbb{E}^{\mathcal{W}}\left[I_{\alpha}\left(\mu_{n}^{f,\theta}\right)\right].
\]
For constants $c_{1}>1$, $c_{2}>0$, define the measurable subset
of $\Theta$ 
\[
\Lambda_{n}^{\alpha:}:=\left\{ \theta\in\Theta:\,\frac{1}{c_{1}}\leq\mu_{n}^{f,\theta}\left(\overline{S\left(O,1\right)}\right)\leq c_{1},\,I_{\alpha}\left(\mu_{n}^{f,\theta}\right)\leq c_{2}\right\} 
\]
and $\Lambda^{\alpha}:=\limsup_{n\rightarrow\infty}\Lambda_{n}^{\alpha}$.
Clearly, 
\[
\sup_{n\geq1}\,\mathcal{W}\left(I_{\alpha}\left(\mu_{n}^{f,\theta}\right)>c_{2}\right)\leq\frac{A_{2}}{c_{2}}\mbox{ and }\sup_{n\geq1}\,\mathcal{W}\left(\mu_{n}^{f,\theta}\left(\overline{S\left(O,1\right)}\right)>c_{1}\right)\leq\frac{1}{c_{1}}.
\]
Moreover, by (\ref{eq:1st moment of total mass under mu_theta}) and
the Paley-Zygmund inequality, 
\[
\begin{split}\sup_{n\geq1}\,\mathcal{W}\left(\mu_{n}^{f,\theta}\left(\overline{S\left(O,1\right)}\right)<\frac{1}{c_{1}}\right) & \leq1-\frac{\left(1-\frac{1}{c_{1}}\right)^{2}}{A_{1}}\end{split}
.
\]
As a consequence, by choosing $c_{1}$ and $c_{2}$ sufficiently large,
we can make 
\[
\mathcal{W}\left(\Lambda_{n}^{\alpha}\right)>\frac{\left(1-\frac{1}{c_{1}}\right)^{2}}{A_{1}}-\frac{1}{c_{1}}-\frac{A_{2}}{c_{2}}>\frac{1}{2A_{1}}
\]
for every $n\geq1$, and hence $\mathcal{W}\left(\Lambda^{\alpha}\right)\geq\frac{1}{2A_{1}}$. 

Finally, we are ready to extract a limit measure $\mu^{f,\theta}$
from the family $\left\{ \mu_{n}^{f,\theta}:n\geq0\right\} $. For
every $\theta\in\Lambda^{\alpha}$, there exists a subsequence $\left\{ n_{k}:k\geq0\right\} $
such that 
\[
\frac{1}{c_{1}}\leq\mu_{n_{k}}^{f,\theta}\left(\overline{S\left(O,1\right)}\right)\leq c_{1},\,I_{\alpha}\left(\mu_{n_{k}}^{f,\theta}\right)\leq c_{2}\mbox{ for all }k\geq0.
\]
Because $I_{\text{\ensuremath{\alpha}}}$, as a mapping from the space
of finite measures on $\overline{S\left(O,1\right)}$ to $\left[0,\infty\right]$,
is lower semi-continuous with respect to the weak topology, 
\[
\mathcal{M}:=\left\{ \mu\mbox{ Borel measure on }\overline{S\left(O,1\right)}:\,\frac{1}{c_{1}}\leq\mu\left(\overline{S\left(O,1\right)}\right)\leq c_{1},\,I_{\alpha}\left(\mu\right)\leq c_{2}\right\} 
\]
is compact, and hence there exists a Borel measure $\mu^{f,\theta}$
on $\overline{S\left(O,1\right)}$ such that $\mu_{n_{k}}^{f,\theta}$
weakly converges to $\mu^{f,\theta}$ along a subsequence of $\left\{ n_{k}:k\geq0\right\} $.
Thus, 
\[
\frac{1}{c_{1}}\leq\mu^{f,\theta}\left(\overline{S\left(O,1\right)}\right)\leq c_{1},\,I_{\alpha}\left(\mu^{f,\theta}\right)\leq c_{2}.
\]
Moreover, if $\Upsilon^{f,\theta}$ is the set defined in (\ref{eq:subset of S}),
then the weak convergence relation between $\left\{ \mu_{n_{k}}^{f,\theta}:k\geq1\right\} $
and $\mu^{f,\theta}$, combined with the fact that $\mu_{n_{k}}^{f,\theta}$
is supported on $\overline{\bigcup_{x\in\Xi_{n_{k}}^{f,\theta}}\,S\left(x,2^{-n_{k}^{2}}\right)}$
for every $k\geq1$, implies that 
\begin{align*}
\mu^{f,\theta}\left(\Upsilon^{f,\theta}\right) & \geq\limsup_{k\nearrow\infty}\mu_{n_{k}}^{f,\theta}\left(\overline{\bigcup_{x\in\Xi_{n_{k}}^{f,\theta}}\,S\left(x,2^{-n_{k}^{2}}\right)}\right)\geq\frac{1}{c_{1}}.
\end{align*}
This means that $\Upsilon^{f,\theta}$ has strictly positive $\alpha-$capacity,
i.e., 
\[
\sup\left\{ \left(\iint_{\Upsilon^{f,\theta}\times\Upsilon^{f,\theta}}\frac{\mu\times\mu\left(dydw\right)}{\left|y-w\right|^{\alpha}}\right)^{-1}:\,\mu\mbox{ is a probability measure on }\Upsilon^{f,\theta}\right\} >0.
\]
By Frostman's lemma, $\dim_{\mathcal{H}}\left(\Upsilon^{f,\theta}\right)\geq\alpha$
and hence $\dim_{\mathcal{H}}\left(D^{f,\theta}\right)\geq\alpha$.
Thus, we have established that
\[
\mathcal{W}\left(\dim_{\mathcal{H}}\left(D^{f,\theta}\right)\geq\alpha\right)\geq\mathcal{W}\left(\Lambda^{\alpha}\right)\geq\frac{1}{2A_{1}}.
\]
Recall from (\ref{eq:H_basis expansion}) that for $\mathcal{W}-$
almost every $\theta\in\Theta$, 
\[
\theta=\sum_{n\ge1}\mathcal{I}\left(h_{n}\right)\left(\theta\right)h_{n}
\]
where $\left\{ h_{n}:n\geq1\right\} $ is an orthonormal basis of
the Cameron-Martin space $H$ and $\left\{ \mathcal{I}\left(h_{n}\right):n\geq1\right\} $
under $\mathcal{W}$ forms a sequence of independent standard Gaussian
random variables. By a simple application of the Hewitt-Savage 0-1
law, we have that 
\[
\mathcal{W}\left(\dim_{\mathcal{H}}\left(D^{f,\theta}\right)\geq\alpha\right)=1.
\]

Finally, since $\alpha$ is arbitrary in $\left(0,\nu\left(1-2\tilde{c}+\utilde{c}\right)\right)$
with $\tilde{c}$ and $\utilde{c}$ being arbitrarily close to $\bar{c}_{f}$
and, respectively, $\underline{c}_{f}$ , we get the desired lower
bound 
\[
\mathcal{W}\left(\dim_{\mathcal{H}}\left(D^{f,\theta}\right)\geq\nu\left(1-2\bar{c}_{f}+\underline{c}_{f}\right)\right)=1.
\]
This completes the proof of Proposition \ref{prop:lower bound on D^f,theta}(i).$\qquad\qquad\qquad\qquad\hfill\hfill\hfill\square$

As for Proposition \ref{prop:lower bound on D^f,theta}(ii), we follow
exactly the same proof as above, except that, in the proof of Lemma
\ref{lem: subset of D^f,theta}, we argue that if for $m\geq1$ sufficiently
large, $\ell\in\mathbb{N}$ is the unique integer such that $2^{-\ell^{2}}\leq r_{m}<2^{-\left(\ell-1\right)^{2}}$,
then (\ref{eq:an estimate used in finding subset of steep point})
will be replaced by 
\[
\left|\frac{X_{r_{m}}^{f,\theta}\left(x^{\left(n_{j}\right)}\right)}{\Sigma_{r_{m}}^{f}}-\sqrt{2\nu}\right|\leq\frac{\sum_{i=1}^{\ell}\sqrt{\Delta\Sigma_{i}^{f}}}{\Sigma_{r_{m}}^{f}}\leq C\frac{\sqrt{\ell\cdot\Sigma_{2^{-\ell^{2}}}^{f}}}{-\ln r_{m}}\leq\frac{\sqrt{\tilde{c}\ln2}\ell^{3/2}}{\left(\ell-1\right)^{2}\ln2},
\]
where the second inequality relies on the hypothesis in Proposition
\ref{prop:lower bound on D^f,theta}(ii) that 
\[
\liminf_{m\nearrow\infty}\frac{\Sigma_{r_{m}}^{f}}{-\ln r_{m}}>0
\]
but does not require $\underline{c}_{f}>0$. Hence, Lemma \ref{lem: subset of D^f,theta}
still holds in the sense that
\[
SD^{f,\theta}\supseteq\Upsilon^{f,\theta}:=\bigcap_{k\geq1}\,\overline{\bigcup_{n\geq k}\,\bigcup_{x\in\Xi_{n}^{f,\theta}}\,S\left(x,2^{-n^{2}}\right)}.
\]
Therefore, the lower bound of $\dim_{\mathcal{H}}\left(\Upsilon^{f,\theta}\right)$
established in (i) also serves as a lower bound of $\dim_{\mathcal{H}}\left(SD^{f,\theta}\right)$.

As for Proposition \ref{prop:lower bound on D^f,theta}(iii), we first
observe that for every $m\geq1$ and $t\in(r_{m},r_{m-1}]$,
\[
\frac{\Sigma_{t}^{f}}{-\ln t}\leq\frac{\Sigma_{r_{m}}^{f}}{-\ln r_{m-1}}=\frac{\Sigma_{r_{m-1}}^{f}}{-\ln r_{m-1}}\cdot\frac{\Sigma_{r_{m}}^{f}}{\Sigma_{r_{m-1}}^{f}},
\]
so, by (\ref{eq:assump on r_m}),
\[
\limsup_{m\nearrow\infty}\frac{\Sigma_{r_{m}}^{f}}{-\ln r_{m}}=\bar{c}_{f}.
\]
Therefore, we only need to show that
\[
\dim_{\mathcal{H}}\left(D^{f,\theta}\right)\geq\nu\left(1-\bar{c}_{f}\right)
\]
and the rest follows from Proposition \ref{prop: upper bound on Hausdorff dim of limsup set}
and the relations (\ref{eq:sets relation 1}) and (\ref{eq:sets relation 2}).
To this end, we replace, everywhere in the proof of Proposition \ref{prop:lower bound on D^f,theta}(i),
$\left\{ 2^{-n^{2}}:n\geq0\right\} $ by $\left\{ r_{n}:n\geq0\right\} $
(with $r_{0}:=1$). The same line of arguments still works in this
case, even if we do not know specific values of $\left\{ r_{n}:n\geq0\right\} $.
Instead of repeating the entire proof, we will only point out the
steps that require changes, which also shows how the extra constraint
(\ref{eq:assump on r_m}) on $f$ and $\left\{ r_{n}:n\geq0\right\} $
can help. For example, this time (\ref{eq:an estimate used in finding subset of steep point})
will become, for sufficiently large $n$'s,
\[
\sup_{t\in(r_{n},r_{n-1}]}\left|\frac{X_{t}^{f,\theta}\left(x^{\left(n_{j}\right)}\right)}{\Sigma_{t}^{f}}-\sqrt{2\nu}\right|\leq\frac{\sum_{i=1}^{n}\sqrt{\Delta\Sigma_{i}^{f}}}{\Sigma_{r_{n-1}}^{f}}\leq C\frac{\sqrt{n\Sigma_{r_{n}}^{f}}}{\Sigma_{r_{n-1}}^{f}},
\]
which is arbitrarily small according to (\ref{eq:assump on r_m}),
so Lemma \ref{lem: subset of D^f,theta} concludes that 
\[
D^{f,\theta}\supseteq\Upsilon^{f,\theta}:=\bigcap_{k\geq1}\,\overline{\bigcup_{n\geq k}\,\bigcup_{x\in\Xi_{n}^{f,\theta}}\,S\left(x,r_{n}\right)}.
\]
Another change takes place in the estimate on
\[
\frac{\mathcal{W}\left(\Phi_{x_{j}^{(n)},n}^{f}\bigcap\Phi_{x_{k}^{(n)},n}^{f}\right)}{\mathcal{W}\left(\Phi_{x_{j}^{(n)},n}^{f}\right)\mathcal{W}\left(\Phi_{x_{k}^{(n)},n}^{f}\right)},\quad\text{ where }2r_{i+1}<\left|x_{j}^{\left(n\right)}-x_{k}^{\left(n\right)}\right|\leq2r_{i}\text{ for some }i=0,1,\cdots,n-1,
\]
and this is the key factor to the whole proof. In this case, when
$i$ is sufficiently large, according to our earlier derivation, we
have that 
\begin{align*}
\frac{\mathcal{W}\left(\Phi_{x_{j}^{(n)},n}^{f}\bigcap\Phi_{x_{k}^{(n)},n}^{f}\right)}{\mathcal{W}\left(\Phi_{x_{j}^{(n)},n}^{f}\right)\mathcal{W}\left(\Phi_{x_{k}^{(n)},n}^{f}\right)} & \leq\exp\left[\nu\Sigma_{r_{i+1}}^{f}+\nu\left(\Sigma_{r_{i+1}}^{f}-\Sigma_{r_{i-1}}^{f}\right)+o\left(\Sigma_{r_{i+1}}^{f}\right)\right]\\
 & \leq\exp\left[\left(\nu+o\left(1\right)\right)\Sigma_{r_{i}}^{f}\right]\leq r_{i}^{-\nu\tilde{c}+o\left(1\right)},
\end{align*}
which, again, is guaranteed by (\ref{eq:assump on r_m}). The rest
of the proof is exactly the same. 

\section{Examples of Steep Points of Gaussian Free Fields}

By varying the choices of $f\in\mathcal{C}$, the definition (\ref{eq:def of steep point})
of $f-$steep point leads to various exceptional sets of GFFs, including
the classical thick point sets for log-correlated GFFs as defined
in (\ref{eq:2D thick point}), as well as the counterpart for polynomial-correlated
GFFs as defined in (\ref{eq:thick point def}) and (\ref{eq:thick point along sequence}).
Theorem \ref{thm:main theorem hausdorff dimension} offers information
on the Hausdorff dimension of such sets.

\subsection{For Log-Correlated GFFs}

When $\nu=2$, the GFF we have studied in the previous sections is
log-correlated. Considering that the framework developed in $\mathsection3$
applies well when $\Sigma_{t}^{f}$ is ``comparable'' with $\left(-\ln t\right)$
as $t\searrow0$, a natural choice of $f$ is a constant function.
This is the main scheme in which we will discuss certain generalized
thick point sets for log-correlated GFFs.

\subsubsection{Thick Points, Revisited }

By setting $f\equiv\gamma$ for $\gamma\in\mathbb{R}\backslash\left\{ 0\right\} $,
it is clear that for every $\theta\in\Theta$, $t\in(0,1]$ and $x\in\overline{S\left(O,1\right)}$,
\[
X_{t}^{f,\theta}\left(x\right)=\gamma\left(\bar{\theta}_{t}\left(x\right)-\bar{\theta}_{1}\left(x\right)\right),
\]
and 
\[
\Sigma_{t}^{f}=\gamma^{2}\left(G\left(t\right)-G\left(1\right)\right).
\]
By (\ref{eq:asymptotic of G in 2D}), 
\[
c_{f}=\lim_{t\searrow0}\frac{\Sigma_{t}^{f}}{-\ln t}=\frac{\gamma^{2}}{2\pi},
\]
so $x$ is an $f-$steep point of $\theta$ if and only if 
\[
\lim_{t\searrow0}\frac{\bar{\theta}_{t}\left(x\right)}{-\ln t}=\frac{\gamma}{\pi}.
\]
In other words, when $\gamma>0$, according to the definition (\ref{eq:2D thick point}),
the set $D^{f.\theta}$ of $f-$steep points coincides with the set
$T^{\gamma,\theta}$ of \emph{$\gamma-$thick points.} 

Corollary \ref{cor:when sigma/ln limit exists} implies that for $\mathcal{W}-$almost
every $\theta\in\Theta$, if $\gamma^{2}>2\pi$, then
\begin{equation}
T^{\gamma,\theta}=D^{f,\theta}=D_{\liminf}^{f,\theta}=D_{\limsup}^{f,\theta}=\emptyset;\label{eq:log GFF thick point empty}
\end{equation}
if $0<\gamma^{2}\leq2\pi$, then
\begin{equation}
\dim_{\mathcal{H}}\left(T^{\gamma,\theta}\right)=\dim_{\mathcal{H}}\left(D^{f,\theta}\right)=\dim_{\mathcal{H}}\left(D_{\liminf}^{f,\theta}\right)=\dim_{\mathcal{H}}\left(D_{\limsup}^{f,\theta}\right)=2-\frac{\gamma^{2}}{\pi},\label{eq: log GFF thick point H dim}
\end{equation}
which agrees with the results obtained in \cite{HMP}. 

\subsubsection{Oscillatory Thick Points}

Besides the standard thick point set, the general framework of steep
point also allows us to study certain variations of this exceptional
set. Again, we fix a constant $\gamma\in(0,\sqrt{2\pi}]$. We have
already seen in (\ref{eq: log GFF thick point H dim}) that $\mathcal{W}-$almost
surely the set of $x\in\overline{S\left(O,1\right)}$ where
\[
\limsup_{t\searrow0}\frac{\bar{\theta}_{t}\left(x\right)}{-\ln t}\geq\frac{\gamma}{\pi}
\]
has Hausdorff dimension $2-\frac{\gamma^{2}}{\pi}$, and the same
fact holds for the set of $x\in\overline{S\left(O,1\right)}$ where
\[
\liminf_{t\searrow0}\frac{\bar{\theta}_{t}\left(x\right)}{-\ln t}\leq-\frac{\gamma}{\pi}
\]
due to the fact that $\mathcal{W}$ is invariant under the transformation
$\theta\rightsquigarrow-\theta$. However, if we require the two conditions
above to be met at the same time and set
\[
T_{oscil.}^{\gamma,\theta}:=\left\{ x\in\overline{S\left(O,1\right)}:\,\limsup_{t\searrow0}\frac{\bar{\theta}_{t}\left(x\right)}{-\ln t}\geq\frac{\gamma}{\pi}\text{ and }\liminf_{t\searrow0}\frac{\bar{\theta}_{t}\left(x\right)}{-\ln t}\le-\frac{\gamma}{\pi}\right\} ,
\]
then, intuitively, $T_{oscil.}^{\gamma,\theta}$ contains \emph{oscillatory
thick points} where $\bar{\theta}_{t}$ oscillates and achieves an
unusually large magnitude in both the positive and the negative directions.
One would expect that $T_{oscil.}^{\gamma,\theta}$ is much smaller
than either of the two sets mentioned above by imposing only one condition.
But we will show that, at least in terms of the Hausdorff dimension,
$T_{oscil.}^{\gamma,\theta}$ is as ``big'' as either of the two
sets.
\begin{prop}
\label{prop:oscil. log}If $\gamma\in\left[0,\sqrt{2\pi}\right]$,
then for $\mathcal{W}-$almost every $\theta\in\Theta$,
\[
\dim_{\mathcal{H}}\left(T_{oscil.}^{\gamma,\theta}\right)=2-\frac{\gamma^{2}}{\pi}.
\]
\end{prop}

\begin{proof}
It is obvious that $T_{oscil.}^{\gamma,\theta}$ has a Hausdorff dimension
no larger than $2-\frac{\gamma^{2}}{\pi}$. To show the other direction,
we consider a sequence $\left\{ r_{n}:n\geq1\right\} \subseteq\left(0,1\right)$
such that $r_{n}\searrow0$ as $n\nearrow\infty$ and
\begin{equation}
\lim_{n\nearrow\infty}\frac{n\left(-\ln r_{n-1}\right)}{-\ln r_{n}}=0.\label{eq: assumption on decay rate of r_m}
\end{equation}
Set $r_{0}=1$, and define the piece-wise constant function
\[
f_{oscil.}:t\in(0,1]\mapsto f_{oscil.}\left(t\right):=\sum_{n\geq1}\mathbb{I}_{(r_{n},r_{n-1}]}\left(t\right)\left(-1\right)^{n}\gamma.
\]
It is still the case that $\Sigma_{t}^{f_{oscil.}}=\gamma^{2}\left(G\left(t\right)-G\left(1\right)\right)$,
but 
\[
X_{t}^{f_{oscil.},\theta}\left(x\right)=\gamma\sum_{n\geq0}\left(-1\right)^{n}\left(\bar{\theta}_{t\vee r_{n}}\left(x\right)-\bar{\theta}_{t\vee r_{n-1}}\left(x\right)\right).
\]

According to (\ref{eq:log GFF thick point empty}), for $\mathcal{W}-$almost
every $\theta$, 
\[
\sup_{x\in\overline{S\left(O,1\right)}}\limsup_{t\searrow0}\frac{\left|\bar{\theta}_{t}\left(x\right)\right|}{G\left(t\right)}\leq2\gamma.
\]
Combining the above with (\ref{eq: assumption on decay rate of r_m}),
it is obvious that, for every $x\in\overline{S\left(O,1\right)}$,
\[
\lim_{n\nearrow\infty}\frac{\sum_{j=1}^{n-1}\left|\bar{\theta}_{r_{j}}\left(x\right)\right|}{G\left(r_{n}\right)}=0.
\]
Therefore, if $x\in D^{f_{oscil.},\theta}$, then
\[
\lim_{k\nearrow\infty}\frac{X_{r_{2k}}^{f_{oscil.},\theta}}{\Sigma_{r_{2k}}^{f_{oscil.}}}=\lim_{k\nearrow\infty}\frac{X_{r_{2k-1}}^{f_{oscil.},\theta}}{\Sigma_{r_{2k-1}}^{f_{oscil.}}}=2,
\]
which, one can easily verify that, is equivalent to 
\[
\lim_{k\nearrow\infty}\frac{\bar{\theta}_{r_{2k}}\left(x\right)}{-\ln r_{2k}}=\frac{\gamma}{\pi}=-\lim_{k\nearrow\infty}\frac{\bar{\theta}_{r_{2k-1}}\left(x\right)}{-\ln r_{2k-1}}.
\]
We conclude that $\mathcal{W}-$almost surely $D^{f_{oscil.},\theta}\subseteq T_{oscil.}^{\gamma,\theta}$,
and hence
\[
\dim_{\mathcal{H}}\left(T_{oscil.}^{\gamma,\theta}\right)\geq2-\frac{\gamma^{2}}{\pi}.
\]
\end{proof}
The results above also apply to ``asymmetric'' oscillations of $\bar{\theta}_{t}\left(x\right)$.
Namely, for $\gamma_{1},\gamma_{2}>0$ and $\gamma_{1}\neq\gamma_{2}$,
we set 
\[
T_{oscil.}^{(\gamma_{1},-\gamma_{2}),\theta}:=\left\{ x\in\overline{S\left(O,1\right)}:\,\limsup_{t\searrow0}\frac{\bar{\theta}_{t}\left(x\right)}{-\ln t}\geq\frac{\gamma_{1}}{\pi}\text{ and }\liminf_{t\searrow0}\frac{\bar{\theta}_{t}\left(x\right)}{-\ln t}\le-\frac{\gamma_{2}}{\pi}\right\} .
\]
Then we have the following fact.
\begin{cor}
If $\gamma_{1},\gamma_{2}\in\left[0,\sqrt{2\pi}\right]$, then for
$\mathcal{W}-$almost every $\theta\in\Theta$,
\[
\dim_{\mathcal{H}}\left(T_{oscil.}^{(\gamma_{1},-\gamma_{2}),\theta}\right)=2-\frac{\gamma_{1}^{2}\vee\gamma_{2}^{2}}{\pi}.
\]
\end{cor}

\begin{proof}
On one hand, it is clear that $\mathcal{W}-$almost surely
\[
\dim_{\mathcal{H}}\left(T_{oscil.}^{(\gamma_{1},-\gamma_{2}),\theta}\right)\leq\dim_{\mathcal{H}}\left(T_{oscil.}^{\gamma_{1},\theta}\right)\wedge\dim_{\mathcal{H}}\left(T_{oscil.}^{\gamma_{2},\theta}\right)=2-\frac{\gamma_{1}^{2}\vee\gamma_{2}^{2}}{\pi}.
\]
On the other hand, since $T_{oscil.}^{\left(\gamma_{1},-\gamma_{2}\right),\theta}\supseteq T_{oscil.}^{\gamma,\theta}$
where $\gamma:=\gamma_{1}\vee\gamma_{2}$, $\mathcal{W}-$almost surely
\[
\dim_{\mathcal{H}}\left(T_{oscil.}^{(\gamma_{1},-\gamma_{2}),\theta}\right)\geq\dim_{\mathcal{H}}\left(T_{oscil.}^{\gamma,\theta}\right)=2-\frac{\gamma^{2}}{\pi}.
\]
\end{proof}

\subsection{For Polynomial-Correlated GFFs}

In the case when $\nu\geq3$, the GFF is polynomial-correlated with
the degree of the polynomial being $\nu-2$. In order to make $\Sigma_{t}^{f}$
comparable with $\left(-\ln t\right)$ as $t\searrow0$, the natural
choice of $f$ is a constant multiple of $\frac{1}{\sqrt{G\left(t\right)}}$.
Namely, if for some $c\in\mathbb{R}\backslash\left\{ 0\right\} $,
\[
f:t\in(0,1]\mapsto f\left(t\right):=\frac{c}{\sqrt{G\left(t\right)}},
\]
then for every $\theta\in\Theta$, $t\in(0,1]$ and $x\in\overline{S\left(O,1\right)}$,
\[
\Sigma_{t}^{f}=c^{2}\ln\frac{G\left(t\right)}{G\left(1\right)}\text{ and }X_{t}^{f,\theta}\left(x\right)=\int_{1}^{t}\frac{c}{\sqrt{G\left(s\right)}}d\bar{\theta}_{s}\left(x\right),
\]
and hence by (\ref{eq:asymptotic of G in 2D}), 
\[
c_{f}=\lim_{t\searrow0}\frac{\Sigma_{t}^{f}}{-\ln t}=c^{2}\left(\nu-2\right).
\]
Since, 
\begin{equation}
\lim_{t\searrow0}\frac{X_{t}^{f,\theta}\left(x\right)}{\Sigma_{t}^{f}}=\sqrt{2\nu}\text{ if and only if }\lim_{t\searrow0}\frac{\int_{1}^{t}\frac{d\bar{\theta}_{s}\left(x\right)}{\sqrt{G\left(s\right)}}}{-\ln t}=c\sqrt{2\nu}\left(\nu-2\right),\label{eq: thin point}
\end{equation}
the result in Corollary \ref{cor:when sigma/ln limit exists} implies
that $\mathcal{W}-$almost surely,
\begin{equation}
\sup_{x\in\overline{S\left(O,1\right)}}\limsup_{t\searrow0}\frac{\int_{1}^{t}\frac{d\bar{\theta}_{s}\left(x\right)}{\sqrt{G\left(s\right)}}}{-\ln t}\leq\sqrt{2\nu\left(\nu-2\right)}\label{eq: thin point limsup upper bound}
\end{equation}
and for $c$ such that $c^{2}<\frac{1}{\nu-2}$,
\begin{equation}
\dim_{\mathcal{H}}\left(\left\{ x\in\overline{S\left(O,1\right)}:\lim_{t\searrow0}\frac{\int_{1}^{t}\frac{d\bar{\theta}_{s}\left(x\right)}{\sqrt{G\left(s\right)}}}{-\ln t}=c\sqrt{2\nu}\left(\nu-2\right)\right\} \right)=\nu\left[1-c^{2}\left(\nu-2\right)\right].\label{eq: thin point H dim}
\end{equation}
A point $x$ where the limit (\ref{eq: thin point}) is achieved is
certainly a location where $\bar{\theta}_{t}$ behaves unusually,
but it does not correspond to the behavior of $\bar{\theta}_{t}$
attaining unusually large values. In fact, we will argue in $\mathsection5.1$
that such a location is where the value of $\bar{\theta}_{t}$ ``tends''
to remain unusually low when measured by certain ``clock'' in $t$.
This is to say that $x$ should not be considered as a ``thick point'',
although the result (\ref{eq: thin point limsup upper bound}) can
help us acquire information on the thick point sets. 

\subsubsection{Thick Points, Revisited}

Recall from (\ref{eq:thick point def}) that, for $\gamma\geq0$,
the $\gamma-$ \emph{thick point set} of $\theta$ is 
\[
T^{\gamma,\theta}:=\left\{ x\in\overline{S\left(O,1\right)}:\;\limsup_{t\searrow0}\frac{\bar{\theta}_{t}\left(x\right)}{\sqrt{-G\left(t\right)\ln t}}\geq\sqrt{2\nu}\gamma\right\} .
\]
As we have reviewed in the Introduction, it is proven in \cite{Chen_thick_point}
that $\mathcal{W}-$almost surely, $T^{\gamma,\theta}=\emptyset$
when $\gamma>1$, and $\dim_{\mathcal{H}}\left(T^{\gamma,\theta}\right)=\nu\left(1-\gamma^{2}\right)$
when $\gamma\in[0,1]$. We also explained earlier that, due to the
higher-order singularity of the covariance function, the proof of
these results, especially the lower bound of $\dim_{\mathcal{H}}\left(T^{\gamma,\theta}\right)$,
was considerably more technical and involved than that in the log-correlated
case; in fact, the lower bound on $\dim_{\mathcal{H}}\left(T^{\gamma,\theta}\right)$
was established indirectly through treating the sequential $\gamma-$thick
point\emph{ }set\emph{ }
\[
ST^{\gamma,\theta}:=\left\{ x\in\overline{S\left(O,1\right)}:\;\lim_{n\nearrow\infty}\frac{\bar{\theta}_{r_{n}}\left(x\right)}{\sqrt{-G\left(r_{n}\right)\ln r_{n}}}=\sqrt{2\nu}\gamma\right\} ,
\]
where $\left\{ r_{n}:n\geq1\right\} \subseteq(0,1]$ is a sequence
satisfying (\ref{eq: assumption on decay rate of r_m}). A lower bound
of $\dim_{\mathcal{H}}\left(ST^{\gamma,\theta}\right)$ would lead
to a lower bound of $\dim_{\mathcal{H}}\left(T^{\gamma,\theta}\right)$. 

Below we will revisit $T^{\gamma,\theta}$ and $ST^{\gamma,\theta}$
under the framework of steep point. In particular, we can provide
a lower bound for $\dim_{\mathcal{H}}\left(ST^{\gamma,\theta}\right)$
as well as $\dim_{\mathcal{H}}\left(T^{\gamma,\theta}\right)$ using
Proposition \ref{prop:lower bound on D^f,theta} with a much shorter
and easier proof than the one given in \cite{Chen_thick_point}. However,
the lower bound we obtain here is not as tight as the one provided
in \cite{Chen_thick_point}, which suggests that, in order to obtain
the exact Hausdorff dimension of $T^{\gamma,\theta}$ and $ST^{\gamma,\theta}$,
one does need to carry out a careful analysis as in \cite{Chen_thick_point}. 
\begin{rem}
In fact, we believe that the same procedure as adopted in the study
of $ST^{\gamma,\theta}$ can be applied to give more accurate treatments
to the sequential steep point set $SD^{f,\theta}$ for any fast decaying
sequence $\left\{ r_{m}:m\geq1\right\} $; in other words, with (possibly)
heavier technicality, it is possible to determine the Hausdorff dimension
of $SD^{f,\theta}$ for $\left\{ r_{m}:m\geq1\right\} $ that is more
general than the one considered in Proposition \ref{prop:lower bound on D^f,theta}(iii),
from which one will produce an improved lower bound of $\dim_{\mathcal{H}}\left(D_{\limsup}^{f,\theta}\right)$
and may recover the tight lower bound of $\dim_{\mathcal{H}}\left(ST^{\gamma,\theta}\right)$.
This problem is currently being investigated in a separate work, and
we will not get into details here. 
\end{rem}

Take any sequence $\left\{ r_{n}:n\geq1\right\} \subseteq(0,1]$ such
that $r_{n}\searrow0$ as $n\nearrow\infty$ and (\ref{eq: assumption on decay rate of r_m})
is satisfied, and set $r_{0}=1$. Fix $\gamma\in\left[0,1/\sqrt{2}\right]$,
and define a function $g:(0,1]\rightarrow(0,\infty)$ as
\begin{equation}
t\in(0,1]\rightsquigarrow g\left(t\right):=\gamma\sum_{n\geq1}\sqrt{\frac{-\ln r_{n}}{G\left(r_{n}\right)}}\mathbb{I}_{(r_{n},r_{n-1}]}\left(t\right).\label{eq:def of g piecewise}
\end{equation}
It is easy to check that $g\in\mathcal{C}$ with $\bar{c}_{g}=\gamma^{2}$
and $\underline{c}_{g}=0$. In fact, due to (\ref{eq: assumption on decay rate of r_m}),
when $n$ is sufficiently large,
\begin{equation}
\begin{split}\Sigma_{r_{n}}^{g} & =\left(\gamma^{2}+o\left(1\right)\right)\left(-\ln r_{n}\right).\end{split}
\label{eq:sigma^g _rn-1}
\end{equation}
Therefore, by Proposition \ref{prop:lower bound on D^f,theta}(ii),
if $SD^{g,\theta}$ is the sequential $g-$steep point set associated
with the sequence $\left\{ r_{n}:n\geq1\right\} $, then $\mathcal{W}-$almost
surely
\[
\dim_{\mathcal{H}}\left(SD^{g,\theta}\right)\geq\nu\left(1-2\gamma^{2}\right).
\]

On the other hand, for every $\theta\in\Theta$, $x\in\overline{S\left(O,1\right)}$
and $n\geq1$,
\[
X_{r_{n}}^{g,\theta}\left(x\right)=\gamma\sum_{j=1}^{n}\sqrt{\frac{-\ln r_{j}}{G\left(r_{j}\right)}}\left(\bar{\theta}_{r_{j}}\left(x\right)-\bar{\theta}_{r_{j-1}}\left(x\right)\right).
\]
Using the result established in \cite{Chen_thick_point} that $T^{\gamma,\theta}=\emptyset$
with probability one for any $\gamma>1$, as well as the invariance
of $\mathcal{W}$ under the transformation $\theta\rightsquigarrow-\theta$,
we know that $\mathcal{W}-$almost surely, for every $x\in\overline{S\left(O,1\right)}$,
\begin{equation}
\liminf_{t\searrow0}\frac{\bar{\theta}_{t}\left(x\right)}{\sqrt{-G\left(t\right)\ln t}}\geq-\sqrt{2\nu}\text{ and }\limsup_{t\searrow0}\frac{\bar{\theta}_{t}\left(x\right)}{\sqrt{-G\left(t\right)\ln t}}\leq\sqrt{2\nu},\label{eq:thick point upper bound for polyn GFF-1}
\end{equation}
and hence 
\[
\left|\bar{\theta}_{r_{n}}\left(x\right)\right|\leq\left(\sqrt{2\nu}+1\right)\sqrt{-G\left(r_{n}\right)\ln r_{n}}
\]
for all but finitely many $n$'s. Combining (\ref{eq: assumption on decay rate of r_m}),
(\ref{eq:sigma^g _rn-1}) and (\ref{eq:thick point upper bound for polyn GFF-1})
leads to, for every $x\in\overline{S\left(O,1\right)}$,
\[
\lim_{n\nearrow\infty}\frac{1}{\Sigma_{r_{n}}^{g}}\left(\sqrt{\frac{-\ln r_{n}}{G\left(r_{n}\right)}}\left|\bar{\theta}_{r_{n-1}}\left(x\right)\right|+\sum_{j=1}^{n-1}\sqrt{\frac{-\ln r_{j}}{G\left(r_{j}\right)}}\left|\bar{\theta}_{r_{j}}\left(x\right)-\bar{\theta}_{r_{j-1}}\left(x\right)\right|\right)=0.
\]
It becomes clear that if $x\in SD^{g,\theta}$, then 
\[
\sqrt{2\nu}=\lim_{n\nearrow\infty}\frac{X_{r_{n}}^{g,\theta}\left(x\right)}{\Sigma_{r_{n}}^{g}}=\lim_{n\nearrow\infty}\gamma\frac{\sqrt{-\ln r_{n}}\bar{\theta}_{r_{n}}\left(x\right)}{\sqrt{G\left(r_{n}\right)}\Sigma_{r_{n}}^{g}},
\]
which, by (\ref{eq:sigma^g _rn-1}), implies that 
\[
\lim_{n\nearrow\infty}\frac{\bar{\theta}_{r_{n}}\left(x\right)}{\sqrt{-G\left(r_{n}\right)\ln r_{n}}}=\sqrt{2\nu}\gamma.
\]
Therefore, we conclude that $\mathcal{W}-$almost surely $ST^{\gamma,\theta}\supseteq D^{g,\theta}$
and 
\[
\dim_{\mathcal{H}}\left(T^{\gamma,\theta}\right)\geq\dim_{\mathcal{H}}\left(ST^{\gamma,\theta}\right)\geq\dim_{\mathcal{H}}\left(SD^{g,\theta}\right)\geq\nu\left(1-2\gamma^{2}\right).
\]

\subsubsection{Oscillatory Thick Points}

Similarly as discussed in the log-correlated case, we can also consider,
for $\theta$ being the generic element of a polynomial-correlated
GFF, the exceptional set given by the \emph{oscillatory thick points
}as, for $\gamma_{1},\gamma_{2}>0$,
\[
T_{oscil.}^{\left(\gamma_{1},-\gamma_{2}\right),\theta}:=\left\{ x\in\overline{S\left(O,1\right)}:\;\limsup_{t\searrow0}\frac{\bar{\theta}_{t}\left(x\right)}{\sqrt{-G\left(t\right)\ln t}}\geq\sqrt{2\nu}\gamma_{1}\text{ and }\liminf_{t\searrow0}\frac{\bar{\theta}_{t}\left(x\right)}{\sqrt{-G\left(t\right)\ln t}}\leq-\sqrt{2\nu}\gamma_{2}\right\} .
\]
\begin{prop}
\label{prop:oscil. polyn}If $\gamma_{1},\gamma_{2}\in\left[0,1/\sqrt{2}\right]$,
then for $\mathcal{W}-$almost every $\theta\in\Theta$, 
\[
\nu\left[1-2\left(\gamma_{1}^{2}\vee\gamma_{2}^{2}\right)\right]\leq\dim_{\mathcal{H}}\left(T_{oscil.}^{\left(\gamma_{1},-\gamma_{2}\right),\theta}\right)\leq\nu\left[1-\left(\gamma_{1}^{2}\vee\gamma_{2}^{2}\right)\right].
\]
\end{prop}

\begin{proof}
First, recall that it is shown in \cite{Chen_thick_point} that for
any $\gamma\in[0,1]$, $\mathcal{W}-$almost surely
\[
\dim_{\mathcal{H}}\left(T^{\gamma,\theta}\right)\leq\nu\left(1-\gamma^{2}\right),
\]
which, combined with the invariance of $\mathcal{W}$ under $\theta\rightsquigarrow-\theta$,
implies that $\mathcal{W}-$almost surely
\[
\dim_{\mathcal{H}}\left(T_{oscil.}^{\left(\gamma_{1},-\gamma_{2}\right),\theta}\right)\leq\dim_{\mathcal{H}}\left(T^{\gamma_{1},\theta}\right)\wedge\dim_{\mathcal{H}}\left(T^{\gamma_{2},\theta}\right)\leq\nu\left[1-\left(\gamma_{1}^{2}\vee\gamma_{2}^{2}\right)\right].
\]
Next, set $\gamma:=\gamma_{1}\vee\gamma_{2}$. Choose the same $\left\{ r_{n}:n\geq0\right\} $
as in $\mathsection4.2.1$, consider the function
\[
g_{oscil.}:t\in(0,1]\mapsto g_{oscil.}\left(t\right):=\sum_{n=1}^{\infty}\left(-1\right)^{n}\gamma\sqrt{\frac{-\ln r_{n}}{G\left(r_{n}\right)}}\mathbb{I}_{(r_{n},r_{n-1}]}\left(t\right).
\]
and let $SD^{g_{oscil.},\theta}$ be the set of sequential $g_{oscil.}-$steep
points of $\theta$ associated with $\left\{ r_{n}:n\geq1\right\} $.
Following exactly the same arguments as above, one can show that,
for $\mathcal{W}-$almost every $\theta\in\Theta$, if $x\in SD^{g_{oscil.},\theta}$,
then 
\[
\lim_{k\nearrow\infty}\frac{\bar{\theta}_{r_{2k}}\left(x\right)}{\sqrt{-G\left(r_{2k}\right)\ln r_{2k}}}=\sqrt{2\nu}\gamma\text{ and }\lim_{k\nearrow\infty}\frac{\bar{\theta}_{r_{2k-1}}\left(x\right)}{\sqrt{-G\left(r_{2k-1}\right)\ln r_{2k-1}}}=-\sqrt{2\nu}\gamma,
\]
which means that $SD^{g_{oscil.},\theta}\subseteq T_{oscil.}^{\left(\gamma_{1},-\gamma_{2}\right),\theta}$
and hence
\[
\dim_{\mathcal{H}}\left(T_{oscil.}^{\left(\gamma_{1},-\gamma_{2}\right),\theta}\right)\geq\dim\left(SD^{g_{oscil.},\theta}\right)\geq\nu\left(1-2\gamma^{2}\right)=\nu\left[1-2\left(\gamma_{1}^{2}\vee\gamma_{2}^{2}\right)\right].
\]
\end{proof}

\subsubsection{Lasting Thick Points}

As we have seen so far in the treatment of thick points of $\theta$,
we have only invoked Proposition \ref{prop:lower bound on D^f,theta}(ii),
the results on $SD^{g,\theta}$. It is natural for one to ask whether
the analysis of $D^{g,\theta}$ leads to any new exceptional sets
of $\theta$, and if so, whether we can use Theorem \ref{thm:main theorem hausdorff dimension}
to get information on such sets. The answer to both questions is positive.
We propose to study the following exceptional set of $\theta$. For
$\gamma>0$, we call $x\in\overline{S\left(O,1\right)}$ a \emph{lasting
$\gamma-$thick point} of $\theta$ if 
\[
\limsup_{t\searrow0}\frac{\int_{1}^{t}\mathbb{I}_{[\sqrt{2\nu}\gamma,\infty)}\left(\frac{\bar{\theta}_{s}\left(x\right)}{\sqrt{-G\left(s\right)\ln s}}\right)dG\left(s\right)}{G\left(t\right)}>0,
\]
where ``$\mathbb{I}_{A}$'' refers to the indicator function of
a set $A\subseteq(0,\infty)$. Denote by $LT^{\gamma,\theta}$ the
collection of the lasting $\gamma-$thick points of $\theta$. Heuristically
speaking, a lasting $\gamma-$thick point is a location where the
behavior of $\text{\ensuremath{\bar{\theta}}}_{t}\left(x\right)$
achieving unusually large values, i.e., greater than or equal to $\sqrt{2\nu}\gamma\sqrt{-G\left(t\right)\ln t}$,
``lasts'' for a cumulative period of time that is a non-negligible
fraction of the total duration of the process. Of course, $LT^{\gamma,\theta}$
is an exceptional set and
\[
LT^{\gamma,\theta}\subseteq T^{\gamma,\theta},
\]
so $\mathcal{W}-$almost surely $LT^{\gamma,\theta}=\emptyset$ when
$\gamma>1$, and
\[
\dim_{\mathcal{H}}\left(LT^{\gamma,\theta}\right)\leq\dim_{\mathcal{H}}\left(T^{\gamma,\theta}\right)\leq\nu\left(1-\gamma^{2}\right)
\]
when $\gamma\in[0,1]$. Below we will derive a lower bound for $\dim_{\mathcal{H}}\left(LT^{\gamma,\theta}\right)$
by drawing the connection between $LT^{\gamma,\theta}$ and the steep
point set $D^{g,\theta}$ considered earlier. 
\begin{prop}
\label{prop:lasting thick point}If $\gamma\in\left[0,1/\sqrt{2}\right]$,
then for $\mathcal{W}-$almost every $\theta\in\Theta$, 
\[
\nu\left(1-2\gamma^{2}\right)\leq\dim_{\mathcal{H}}\left(LT^{\gamma,\theta}\right)\leq\nu\left(1-\gamma^{2}\right)
\]
\end{prop}

\begin{proof}
Only the lower bound requires proof. After a quick examination of
the arguments in $\mathsection4.2.1$, we realize that, for the piece-wise
constant function $g$ defined in (\ref{eq:def of g piecewise}),
we cannot apply Theorem \ref{thm:main theorem hausdorff dimension}
because $\underline{c}_{g}=0$. To overcome this problem, we consider
a perturbation of $g$. Namely, let $\gamma^{\,\prime}>\gamma$ be
arbitrarily close to $\gamma$, $\epsilon>0$ be arbitrarily small
and $\left\{ r_{n}:n\geq0\right\} $ be the same as in $\mathsection4.2.1$,
and define the function
\[
g^{\epsilon}:t\in(0,1]\mapsto g^{\epsilon}\left(t\right):=\gamma^{\,\prime}\sum_{n\geq1}\sqrt{\frac{-\ln r_{n}}{G\left(r_{n}\right)}}\mathbb{I}_{(r_{n},r_{n-1}]}\left(t\right)+\epsilon\frac{1}{\sqrt{G\left(t\right)}}.
\]
Again, $g^{\epsilon}\in\mathcal{C}$. It is straightforward to check,
by (\ref{eq: assumption on decay rate of r_m}), that when $n$ is
sufficiently large,
\[
\begin{split}\Sigma_{r_{n}}^{g^{\epsilon}} & =\left[\left(\gamma^{\,\prime}\right)^{2}+o\left(1\right)\right]\left(-\ln r_{n}\right)+\epsilon^{2}\ln G\left(r_{n}\right),\end{split}
\]
and if $t\in(r_{n},r_{n-1}]$, then $\Sigma_{t}^{g^{\epsilon}}$ is
equal to
\begin{equation}
\begin{split} & \Sigma_{r_{n-1}}^{g^{\epsilon}}+\int_{r_{n-1}}^{t}\left(\gamma^{\,\prime}\sqrt{\frac{-\ln r_{n}}{G\left(r_{n}\right)}}+\epsilon\frac{1}{\sqrt{G\left(s\right)}}\right)^{2}dG\left(s\right)\\
= & \left(\gamma^{\,\prime}\right)^{2}\left(-\ln r_{n}\right)\frac{G\left(t\right)}{G\left(r_{n}\right)}+4\gamma^{\,\prime}\epsilon\sqrt{-\ln r_{n}}\sqrt{\frac{G\left(t\right)}{G\left(r_{n}\right)}}+\left[\left(\gamma^{\,\prime}\right)^{2}+o\left(1\right)\right]\left(-\ln r_{n-1}\right)+\epsilon^{2}\ln G\left(t\right)\\
= & \underset{\varphi_{n,t}}{\underbrace{\left(\gamma^{\,\prime}\right)^{2}\left(-\ln r_{n}\right)\frac{G\left(t\right)}{G\left(r_{n}\right)}+\left(\gamma^{\,\prime}\right)^{2}\left(-\ln r_{n-1}\right)}}+\epsilon^{2}\ln G\left(t\right)+o\left(-\ln t\right)
\end{split}
\label{eq:computation of sigma_t}
\end{equation}
Since 
\[
\limsup_{t\searrow0}\sum_{n=1}^{\infty}\mathbb{I}_{(r_{n},r_{n-1}]}\left(t\right)\frac{\varphi_{n,t}}{-\ln t}=\left(\gamma^{\,\prime}\right)^{2}\text{ and }\liminf_{t\searrow0}\sum_{n=1}^{\infty}\mathbb{I}_{(r_{n},r_{n-1}]}\left(t\right)\frac{\varphi_{n,t}}{-\ln t}=0,
\]
we see that
\[
\bar{c}_{g^{\epsilon}}=\left(\gamma^{\,\prime}\right)^{2}+\epsilon^{2}\left(\nu-2\right)\text{ and }\underline{c}_{g^{\epsilon}}=\epsilon^{2}\left(\nu-2\right).
\]
It becomes clear that by including the term ``$\epsilon\frac{1}{\sqrt{G\left(t\right)}}$''
in the definition of $g^{\epsilon}$, we have made $\underline{c}_{g^{\epsilon}}>0$,
to which case the main theorem can apply. Applying Theorem \ref{thm:main theorem hausdorff dimension},
we know that 
\[
\dim_{\mathcal{H}}\left(D^{g^{\epsilon},\theta}\right)\geq\nu\left[1-2\left(\gamma^{\,\prime}\right)^{2}-\epsilon^{2}\left(\nu-2\right)\right]
\]
provided that $\left(\gamma^{\,\prime}\right)^{2}<\frac{1}{2}$ and
$\epsilon>0$ is sufficiently small. 

On the other hand, for every $\theta\in\Theta$, $x\in\overline{S\left(O,1\right)}$
and $n\geq1$, if $t\in(r_{n},r_{n-1}]$, then
\begin{align*}
X_{t}^{g^{\epsilon},\theta}\left(x\right) & =\gamma^{\,\prime}\sqrt{\frac{-\ln r_{n}}{G\left(r_{n}\right)}}\left(\bar{\theta}_{t}\left(x\right)-\bar{\theta}_{r_{n-1}}\left(x\right)\right)\\
 & \hfill\hfill+\gamma^{\,\prime}\sum_{j=1}^{n-1}\sqrt{\frac{-\ln r_{j}}{G\left(r_{j}\right)}}\left(\bar{\theta}_{r_{j}}\left(x\right)-\bar{\theta}_{r_{j-1}}\left(x\right)\right)+\epsilon\int_{1}^{t}\frac{1}{\sqrt{G\left(s\right)}}d\bar{\theta}_{s}\left(x\right).
\end{align*}
We write $\gamma^{\,\prime\prime}:=\frac{\gamma+\gamma^{\,\prime}}{2}$
and restrict $t$ to $(r_{n},Ar_{n}]$ where $A:=\left(\gamma^{\,\prime}/\gamma^{\,\prime\prime}\right)^{\frac{2}{\nu-2}}$
. It is clear from (\ref{eq:computation of sigma_t}) that 
\begin{align*}
\Sigma_{t}^{g^{\epsilon}}\geq\Sigma_{Ar_{n}}^{g^{\epsilon}} & \geq\left[\left(\gamma^{\,\prime}\right)^{2}A^{2-\nu}+o\left(1\right)\right]\left(-\ln r_{n}\right).
\end{align*}
Following the same arguments as earlier, we see that when $n$ is
large,
\[
-\gamma^{\,\prime}\sqrt{\frac{-\ln r_{n}}{G\left(r_{n}\right)}}\bar{\theta}_{r_{n-1}}\left(x\right)+\gamma^{\,\prime}\sum_{j=1}^{n-1}\sqrt{\frac{-\ln r_{j}}{G\left(r_{j}\right)}}\left(\bar{\theta}_{r_{j}}\left(x\right)-\bar{\theta}_{r_{j-1}}\left(x\right)\right)=o\left(-\ln r_{n}\right)=o\left(\Sigma_{t}^{g^{\epsilon}}\right),
\]
so for $t\in(r_{n},Ar_{n}]$, 
\[
\begin{split}\frac{X_{t}^{g^{\epsilon},\theta}\left(x\right)}{\Sigma_{t}^{g^{\epsilon}}} & =\gamma^{\,\prime}\frac{\sqrt{-\ln r_{n}}\bar{\theta}_{t}\left(x\right)}{\sqrt{G\left(r_{n}\right)}\Sigma_{t}^{g^{\epsilon}}}+\epsilon\frac{\int_{1}^{t}\frac{1}{\sqrt{G\left(s\right)}}d\bar{\theta}_{s}\left(x\right)}{\Sigma_{t}^{g^{\epsilon}}}+o\left(1\right).\end{split}
\]
Furthermore, (\ref{eq: thin point limsup upper bound}) tells us that,
when $t\in(r_{n},Ar_{n}]$ and $n$ is large,
\[
\epsilon\frac{\left|\int_{1}^{t}\frac{1}{\sqrt{G\left(s\right)}}d\bar{\theta}_{s}\left(x\right)\right|}{\Sigma_{t}^{g^{\epsilon}}}\leq\epsilon2\sqrt{\nu\left(\nu-2\right)}\frac{\left(-\ln t\right)}{\Sigma_{t}^{g^{\epsilon}}}\leq\frac{\epsilon2\sqrt{\nu\left(\nu-2\right)}}{\left(\gamma^{\,\prime}\right)^{2}A^{2-\nu}}+o\left(1\right).
\]
All in all, we have that, if $x\in D^{g^{\epsilon},\theta}$, $n$
is sufficiently large, and $t\in(r_{n},Ar_{n}]$, then 
\[
\frac{\sqrt{-\ln r_{n}}\bar{\theta}_{t}\left(x\right)}{\sqrt{G\left(r_{n}\right)}\Sigma_{t}^{g^{\epsilon}}}\geq\frac{1}{\gamma^{\,\prime}}\left[\sqrt{2\nu}+o\left(1\right)-\frac{\epsilon2\sqrt{\nu\left(\nu-2\right)}}{\left(\gamma^{\,\prime}\right)^{2}A^{2-\nu}}\right],
\]
which implies that 
\[
\begin{split}\frac{\bar{\theta}_{t}\left(x\right)}{\sqrt{-G\left(t\right)\ln t}} & \geq\frac{\sqrt{-\ln r_{n}}\bar{\theta}_{t}\left(x\right)}{\sqrt{G\left(r_{n}\right)}\Sigma_{t}^{g^{\epsilon}}}\cdot\frac{\Sigma_{t}^{g^{\epsilon}}}{\left(-\ln r_{n}\right)}\\
 & \geq\left[\sqrt{2\nu}+o\left(1\right)-\frac{\epsilon2\sqrt{\nu\left(\nu-2\right)}}{\left(\gamma^{\,\prime}\right)^{2}A^{2-\nu}}\right]\left(\gamma^{\,\prime}A^{2-\nu}+o\left(1\right)\right)\\
 & =\sqrt{2\nu}\frac{\left(\gamma^{\,\prime\prime}\right)^{2}}{\gamma^{\,\prime}}+o\left(1\right)-\epsilon\frac{2\sqrt{\nu\left(\nu-2\right)}}{\gamma^{\,\prime}}\geq\sqrt{2\nu}\gamma,
\end{split}
\]
provided that $\epsilon$ is sufficiently small. This is to say that
\[
\limsup_{n\nearrow\infty}\frac{\int_{1}^{r_{n}}\mathbb{I}_{[\sqrt{2\nu}\gamma,\infty)}\left(\frac{\bar{\theta}_{s}\left(x\right)}{\sqrt{-G\left(s\right)\ln s}}\right)dG\left(s\right)}{G\left(r_{n}\right)}\geq\lim_{n\nearrow\infty}\frac{G\left(r_{n}\right)-G\left(Ar_{n}\right)}{G\left(r_{n}\right)}=1-\left(\frac{\gamma^{\,\prime\prime}}{\gamma^{\,\prime}}\right)^{2}>0.
\]
We can conclude that $\mathcal{\mathcal{W}}-$almost surely $D^{g^{\epsilon},\theta}\subseteq LT^{\gamma,\theta}$
and hence
\[
\dim_{\mathcal{H}}\left(LT^{\gamma,\theta}\right)\geq\dim_{\mathcal{H}}\left(D^{g^{\epsilon},\theta}\right)\geq\nu\left[1-2\left(\gamma^{\,\prime}\right)^{2}-\epsilon^{2}\left(\nu-2\right)\right].
\]
Finally, since $\gamma^{\,\prime}>\gamma$ is arbitrarily close to
$\gamma$ and $\epsilon>0$ is arbitrarily small, we have that $\mathcal{W}-$almost
surely
\[
\dim_{\mathcal{H}}\left(LT^{\gamma,\theta}\right)\geq\nu\left(1-2\gamma^{2}\right).
\]
\end{proof}

\section{Generalizations and Further Questions}

At the end of the article, we briefly allude to a few related problems
and directions in which we would like to further our study.

\subsection{``Thin Points''}

As mentioned in $\mathsection4.2$, when $\nu\geq3$, a natural choice
of $f:\,t\in(0,1]\mapsto f\left(t\right)\in\mathbb{R}$ to which we
can apply Theorem \ref{thm:main theorem hausdorff dimension} is that
$f\left(t\right)$ being a constant multiple of $\frac{1}{\sqrt{G\left(t\right)}}$,
say, $f\left(t\right)=\frac{c}{\sqrt{G\left(t\right)}}$ for some
$c\in\mathbb{R}\backslash\left\{ 0\right\} $. In this case we have
pointed out in (\ref{eq: thin point}) that $x\in D^{f,\theta}$ if
and only if 
\begin{equation}
\lim_{t\searrow0}\frac{\int_{1}^{t}\frac{d\bar{\theta}_{s}\left(x\right)}{\sqrt{G\left(s\right)}}}{-\ln t}=c\sqrt{2\nu}\left(\nu-2\right).\label{eq:thin point reparametrized}
\end{equation}
Let us take a more careful look at the limit involved in (\ref{eq:thin point reparametrized}).
Since for every $\theta\in\Theta$, $x\in\overline{S\left(O,1\right)}$
and $t\in(0,1]$, 
\[
\int_{1}^{t}\frac{d\bar{\theta}_{s}\left(x\right)}{\sqrt{G\left(t\right)}}=\frac{\bar{\theta}_{t}\left(x\right)}{\sqrt{G\left(t\right)}}-\frac{\bar{\theta}_{1}\left(x\right)}{\sqrt{G\left(1\right)}}+\frac{1}{2}\int_{1}^{t}\frac{\bar{\theta}_{s}\left(x\right)}{\sqrt{G\left(s\right)}}\frac{dG\left(s\right)}{G\left(s\right)},
\]
and, again, as we have pointed out earlier, $\mathcal{W}-$almost
surely,
\[
\sup_{x\in\overline{S\left(O,1\right)}}\limsup_{t\searrow0}\frac{\left|\bar{\theta}_{t}\left(x\right)\right|}{\sqrt{-G\left(t\right)\ln t}}\leq\sqrt{2\nu},
\]
it is clear that the limit concerned in (\ref{eq:thin point reparametrized})
is equivalent to 
\[
\lim_{t\searrow0}\frac{\int_{1}^{t}\frac{\bar{\theta}_{s}\left(x\right)}{\sqrt{G\left(s\right)}}\frac{dG\left(s\right)}{G\left(s\right)}}{\ln G\left(t\right)}=2c\sqrt{2\nu}\text{ or }\lim_{t\searrow0}\frac{\int_{1}^{t}\left(\frac{\bar{\theta}_{s}\left(x\right)}{\sqrt{G\left(s\right)}}-2c\sqrt{2\nu}\right)\frac{dG\left(s\right)}{G\left(s\right)}}{\ln G\left(t\right)}=0.
\]
If the limit above is achieved, then it suggests that, at least when
measured by the measure ``$\frac{dG\left(t\right)}{G\left(t\right)}$'',
$\bar{\theta}_{t}\left(x\right)/\sqrt{G\left(t\right)}$ tends to
stay ``close'' to the level of $2c\sqrt{2\nu}$ when $t$ is small,
which means that $\bar{\theta}_{t}\left(x\right)$'s value is unusually
small. If we, tentatively, call such a location $x$ a \emph{``thin
point'' }of $\theta$, then (\ref{eq: thin point H dim}) tells us
that the Hausdorff dimension of the set of ``thin points'' is $\nu\left[1-c^{2}\left(\nu-2\right)\right]$
for $\mathcal{W}-$almost every $\theta\in\Theta$, provided that
$c^{2}\leq\frac{1}{\nu-2}$. 

We think the characterization of ``thin points'' of the GFF can
be improved since the version described there is indirect and restricted
(having to be measured by ``$\frac{dG\left(t\right)}{G\left(t\right)}$'').
Therefore, we hope to further analyze the phenomenon of $\bar{\theta}_{t}$
maintaining unusually low values, by devising a more explicit scheme
to compare or connect $\bar{\theta}_{t}\left(x\right)$ with a constant
multiple of $\sqrt{G\left(t\right)}$.

\subsection{Dependence or Independence on the Choice of $f$}

As we have mentioned in the Introduction, to overcome the singularity
of GFFs in general, various regularization procedures have been introduced
and adopted in the study of GFFs. Although different regularization
procedures may work equally well in the study of certain properties
of GFFs, it is unclear, in most cases, whether an obtained result
is dependent on the specific regularization, or it is intrinsic about
the GFF itself and independent of the choice of regularization. For
example, it remains open, in the general setting, whether two thick
point sets obtained through two different regularizations have any
connection, as well as whether there is an intrinsic way to define
thick points without the use of any regularization.

In our project it is clear that, if $f_{1}$ are $f_{2}$ are two
different choices from $\mathcal{C}$ with $c_{f_{1}}=c_{f_{2}}:=c\in\left[0,1\right]$,
then $\mathcal{W}-$almost surely

\[
\dim_{\mathcal{H}}\left(D^{f_{1},\theta}\right)=\dim_{\mathcal{H}}\left(D^{f_{2},\theta}\right)=\nu\left(1-c\right).
\]
So, when the two choices of test functions have the same key parameter,
at least the Hausdorff dimension of the corresponding steep point
sets are identical. We are interested in further studying the relation
between $D^{f_{1},\theta}$ and $D^{f_{2},\theta}$, In particular,
we hope to use the framework developed in this article to determine
the conditions on $f_{1}$ and $f_{2}$ under which the difference
set between $D^{f_{1},\theta}$ and $D^{f_{2},\theta}$ is small,
as well as to design examples of $f_{1}$ and $f_{2}$ such that the
difference set between $D^{f_{1},\theta}$ and $D^{f_{2},\theta}$
is big.

\subsection{Liouville Quantum Gravity Measures in $\mathbb{R}^{\nu}$ for $\nu\geq3$}

In the Introduction we briefly alluded to the Liouville Quantum Gravity
(LQG) measure on a planar domain, which is a random measure that formally
takes the form of ``$e^{h\left(z\right)}dz$'' where $h$ is a generic
element of the 2D log-correlated GFF and $dz$ is the Lebesgue measure
on the domain. Since, formally, the ``density'' with respect to
the Lebesgue measure is always positive, the LQG measure can be thought
as the induced measure of the Lebesgue measure under a random conformal
transformation, providing a model of 2D random geometry. The fact
that the covariance function of the GFF has a logarithmic (and no
worse than logarithmic) singularity plays an essential role in the
mathematical construction of the LQG measure. Therefore, the straightforward
analog of the LQG measure in $\mathbb{R}^{\nu}$ for $\nu\geq3$,
i.e., ``$e^{\theta\left(x\right)}dx$'' where $\theta$ is a generic
element of the polynomial-correlated GFF on $\mathbb{R}^{\nu}$, is
not accessible in the same way.

On the other hand, if one is interested in modeling random geometry
in $\mathbb{R}^{\nu}$ for $\nu\geq3$ using $\theta$, then a possible
approach is to construct the analog of the LQG measure with the regularized
family of $\theta$ replaced by $\left\{ X_{t}^{f,\theta}\left(x\right):x\in\mathbb{R}^{\nu},t\in(0,1]\right\} $
for some $f\in\mathcal{C}$. The family of $X_{t}^{f,\theta}\left(x\right)$
has the desired logarithmic singularity. Besides, since the LQG measure
has a thick point set as its support, one can expect that an analogous
random measure will be supported on the corresponding $f-$steep point
set. It is also possible to extend further results on the LQG measure
to the proposed random measure, such as the Knizhnik-Polyakov-Zamolodchikov
formula which governs the correspondence between the scaling dimension
of the random measure and that of the Lebesgue measure. We will investigate
this matter in the upcoming work.

\section{Appendix}

In the Appendix we include the complete proofs of Lemma \ref{lem:expectation of max non-concentric}
and Lemma \ref{lem:modulus of X}.

\subsection{Proof of Lemma \ref{lem:expectation of max non-concentric}}

Recall that we want to show that, there exists $C>0$ such that for
every $\nu\geq2$, $x,y\in\overline{S\left(O,1\right)}$, $t\in(0,1]$
and $\delta\in\left(0,\sqrt{G\left(t\right)}\right)$, we have \\
\begin{equation}
d^{2}\left(x,t;\,y,t\right):=\mathbb{E}^{\mathcal{W}}\left[\left|\bar{\theta}_{t}\left(x\right)-\bar{\theta}_{s}\left(y\right)\right|^{2}\right]\leq Ct^{2-\nu}\sqrt{\frac{\left|x-y\right|}{t}},\label{eq:estimate intrinsic metric non-concentric}
\end{equation}
and
\begin{equation}
\mathbb{E}^{\mathcal{W}}\left[\omega_{t}^{\theta}\left(\delta\right)\right]\leq C\delta\sqrt{\ln\left(t^{\left(3-2\nu\right)/4}/\delta\right)},\label{eq:expectation of modulus of continuity}
\end{equation}
where 
\begin{equation}
\omega{}_{t}^{\theta}\left(\delta\right):=\sup\left\{ \left|\bar{\theta}_{s}\left(x\right)-\bar{\theta}_{s^{\prime}}\left(y\right)\right|:\,d\left(x,s;y,s^{\prime}\right)\leq\delta,\,x,y\in\overline{S\left(O,1\right)},\,s,s^{\prime}\in\left[t,1\right]\right\} .\label{eq:def of cont modulus}
\end{equation}
\begin{proof}
Assume $x\neq y$. By (\ref{eq:covariance for (1-Delta)^s concentric})
and (\ref{eq:covariance for (1-Delta)^s-1}), we have that
\[
d^{2}\left(x,t;\,y,t\right)=\frac{2\alpha_{\nu}}{\left(2\pi\right)^{\nu}I_{\frac{\nu-2}{2}}^{2}\left(t\right)}\int_{0}^{\infty}\frac{\tau}{1+\tau^{2}}J_{\frac{\nu-2}{2}}^{2}\left(t\tau\right)\Psi\left(\tau\left|x-y\right|\right)d\tau
\]
where $\Psi$ is the function given by 
\[
w\in\left(0,\infty\right)\rightsquigarrow\Psi\left(w\right):=1-\frac{\left(2\pi\right)^{\nu/2}}{\alpha_{\nu}}w^{\frac{2-\nu}{2}}J_{\frac{\nu-2}{2}}\left(w\right).
\]
It follows from the properties of $J_{\frac{\nu-2}{2}}$ that $\Psi$
is analytic and 
\[
\Psi\left(w\right)=\Gamma\left(\nu/2\right)\sum_{m=1}^{\infty}\frac{\left(-1\right)^{m-1}2^{-2m}w^{2m}}{m!\Gamma\left(\frac{\nu}{2}+m\right)}.
\]
Clearly, there exists $C>0$ such that $\left|\Psi\left(w\right)\right|\leq C\sqrt{w}$
for all $w\in[0,\infty)$. Therefore, 
\[
d^{2}\left(x,t;\,y,t\right)\leq C\frac{1}{I_{\frac{\nu-2}{2}}^{2}\left(t\right)}\sqrt{\left|x-y\right|}\int_{0}^{\infty}\frac{\tau^{3/2}}{1+\tau^{2}}J_{\frac{\nu-2}{2}}^{2}\left(t\tau\right)d\tau.
\]
Assuming $t$ is small, we can estimate the integral in the right
hand side of above as follows:
\[
\begin{split}\int_{0}^{\infty}\frac{\tau^{3/2}}{1+\tau^{2}}J_{\frac{\nu-2}{2}}^{2}\left(t\tau\right)d\tau & =\left(\int_{0}^{1}+\int_{1}^{1/t}+\int_{1/t}^{\infty}\right)\frac{\tau^{3/2}}{1+\tau^{2}}J_{\frac{\nu-2}{2}}^{2}\left(t\tau\right)d\tau\\
 & \leq C+C\int_{1}^{1/t}\tau^{-1/2}d\tau+Ct^{-1}\int_{1/t}^{\infty}\tau^{-3/2}d\tau\\
 & \leq Ct^{-1/2},
\end{split}
\]
which leads to the desired inequality (\ref{eq:estimate intrinsic metric non-concentric}).

We will apply the metric entropy method (e.g., \cite{Dudley,Talagrand,AT07})
to prove (\ref{eq:expectation of modulus of continuity}). For every
compact subset $\mathcal{A}\subseteq\overline{S\left(O,1\right)}\times(0,1]$,
let $\mbox{diam}_{d}\left(\mathcal{A}\right)$ be the diameter of
$\mathcal{A}$ under the metric $d$. $\mathcal{A}$ is also compact
under $d$, so $\mathcal{A}$ can be finitely covered under $d$.
For $\epsilon>0$ and $\mathbf{x}\in\overline{S\left(O,1\right)}\times(0,1]$,
let $B_{d}\left(\mathbf{x},\epsilon\right)$ be the open disc/ball
centered at $\mathbf{x}$ with radius $\epsilon$ under $d$, and
$N\left(\epsilon,\mathcal{A}\right)$ be the smallest number of such
discs/balls $B_{d}\left(\mathbf{x},\epsilon\right)$ required to cover
$\mathcal{A}$. Then $N$ is the metric entropy function with respect
to $d$. For any fixed $t\in\left(0,1\right)$, set
\[
\mathcal{A}_{t}:=\,\overline{S\left(O,1\right)}\times\left[t,1\right],
\]
and let $\omega_{t}^{\theta}$ be as in (\ref{eq:def of cont modulus}).
Then $\omega_{t}^{\theta}$ is the modulus of continuity of the Gaussian
family $\left\{ \bar{\theta}_{s}\left(x\right):\left(x,s\right)\in\mathcal{A}_{t}\right\} $
under the metric $d$, i.e., for $\delta>0$, 
\[
\omega{}_{t}^{\theta}\left(\delta\right)=\sup\left\{ \left|\bar{\theta}_{s}\left(x\right)-\bar{\theta}_{s^{\prime}}\left(y\right)\right|:\,\left(x,s\right),\left(y,s^{\prime}\right)\in\mathcal{A}_{t},\:d\left(x,s;y,s^{\prime}\right)\leq\delta\right\} .
\]
Then, according to the standard metric entropy theory (e.g., Theorem
1.3.5 of \cite{AT07}), there is a universal constant $K>0$ such
that 
\begin{equation}
\mathbb{E}^{\mathcal{W}}\left[\omega{}_{t}^{\theta}\left(\delta\right)\right]\leq K\int_{0}^{\delta}\sqrt{\ln N\left(\epsilon,\mathcal{A}_{t}\right)}d\epsilon.\label{eq:entropy result}
\end{equation}
Below we describe a specific finite covering of $\mathcal{A}_{t}$
for every $\epsilon>0$ sufficiently small. 

First, set
\[
s_{\epsilon}:=\frac{1}{2}\left(\frac{\epsilon^{2}}{9}C^{-1}t^{\nu-3/2}\right)^{2}
\]
where $C$, for the moment, is the same constant as in (\ref{eq:estimate intrinsic metric non-concentric}),
and let
\[
\left\{ B\left(y_{l},s_{\epsilon}\right):\,l=1,\cdots,L_{\epsilon}\right\} 
\]
be a finite covering of $\overline{S\left(O,1\right)}$ where $y_{l}\in S\left(O,1\right)$
and $L_{\epsilon}$ be the smallest number of discs/balls $B\left(y_{l},s_{\epsilon}\right)$
needed to cover $\overline{S\left(O,1\right)}$ and hence 
\[
L_{\epsilon}=\mathcal{O}\left(s_{\epsilon}^{-\nu}\right)\leq C\left[\epsilon^{-1}t^{\left(3-2\nu\right)/4}\right]^{4\nu}.
\]
By (\ref{eq:estimate intrinsic metric non-concentric}), the choice
of $s_{\epsilon}$ is such that, for every $y,w\in B\left(y_{l},s_{\epsilon}\right)$
and every $s\in\left[t,1\right]$, 
\[
d^{2}\left(y,s;\,y^{\prime},s\right)\leq Cs^{3/2-\nu}\sqrt{2s_{\epsilon}}\leq\epsilon^{2}/9.
\]
Next, take $\tau_{0}:=2$ and define $\tau_{m}$ inductively such
that 
\[
G\left(\tau_{m}\right)-G\left(\tau_{m-1}\right)=\epsilon^{2}/9
\]
for $m=1,\cdots,M_{\epsilon}+1$, where $M_{\epsilon}$ is the smallest
integer such that $\tau_{M_{\epsilon}}\leq t$ and hence 
\[
M_{\epsilon}\leq C\left(G\left(t\right)\right)/\epsilon^{2}.
\]
Consider the covering of $\overline{S\left(O,1\right)}\times\left[t,1\right]$
that consists of the ``cylinders'' 
\[
\left\{ B\left(y_{l},s_{\epsilon}\right)\times(\tau_{m+1},\tau_{m-1}):\,l=1,\cdots,L_{\epsilon},\,m=1,\cdots,M_{\epsilon}\right\} .
\]
Any pair of points $\left(\left(y,t\right),\,\left(w,s\right)\right)$
that lies in one of the ``cylinders'' above, e.g., $B\left(y_{l},s_{\epsilon}\right)\times(\tau_{m+1},\tau_{m-1})$,
satisfies that 
\[
\begin{split}d\left(y,t;\,w,s\right) & \leq d\left(y,t;\,y,\tau_{m}\right)+d\left(y,\tau_{m};\,w,\tau_{m}\right)+d\left(w,\tau_{m};\,w,s\right)\\
 & \leq\epsilon/3+\epsilon/3+\epsilon/3=\epsilon.
\end{split}
\]
This implies that 
\[
N\left(\epsilon,\text{\ensuremath{\mathcal{A}}}_{t}\right)\leq L_{\epsilon}\cdot\left(M_{\epsilon}+1\right),
\]
and hence by (\ref{eq:entropy result}),
\[
\begin{split}\mathbb{E}^{\mathcal{W}}\left[\omega_{t}^{\theta}\left(\delta\right)\right] & \leq C\int_{0}^{\delta}\left(\sqrt{\ln L_{\epsilon}}+\sqrt{\ln M_{\epsilon}}\right)d\epsilon.\end{split}
\]
Therefore, we only need to compute the two integrals in the right
hand side above. 

By the estimates we derived for $L_{\epsilon}$ above and a simple
change of variable $u=\sqrt{\ln\left(\epsilon^{-1}t^{\frac{3-2\nu}{4}}\right)}$,
we get that
\[
\begin{split}\int_{0}^{\delta}\sqrt{\ln L_{\epsilon}}d\epsilon & \leq C\int_{0}^{\delta}\sqrt{\ln\left(\epsilon^{-1}t^{\frac{3-2\nu}{4}}\right)}d\epsilon\leq Ct^{\frac{3-2\nu}{4}}\int_{\sqrt{\ln\left(\delta^{-1}t^{\frac{3-2\nu}{4}}\right)}}^{\infty}u^{2}e^{-u^{2}}du.\end{split}
\]
Since$\int_{a}^{\infty}e^{-u^{2}}u^{2}du=\mathcal{O}\left(ae^{-a^{2}}\right)$
when $a>0$ is sufficiently large, we arrive that 
\[
\int_{0}^{\delta}\sqrt{\ln L_{\epsilon}}d\epsilon\leq C\delta\sqrt{\ln\left(t^{\left(3-2\nu\right)/4}/\delta\right)}.
\]
Similarly, one can derive that
\[
\int_{0}^{\delta}\sqrt{\ln M_{\epsilon}}d\epsilon\leq C\delta\sqrt{\ln\left(\sqrt{G\left(t\right)}/\delta\right)}.
\]
Combining the inequalities above, we have proven (\ref{eq:expectation of modulus of continuity}).
\end{proof}

\subsection{Proof of Lemma \ref{lem:modulus of X}}

Recall that for every $n\geq1$, $B_{n}$ is the Borel set in $\overline{S\left(O,1\right)}\times\overline{S\left(O,1\right)}\times(0,1]$
that 
\[
B_{n}:=\left\{ \left(x,y,t\right):\,x,y\in\overline{S\left(O,1\right)},\left|x-y\right|<2^{-\left(n+1\right)^{2}}2\sqrt{\nu},\,t\in\left[2^{-n^{2}},2^{-\left(n-1\right)^{2}}\right]\right\} .
\]
We want to show that for every sufficiently large $n$,
\[
\mathbb{E}^{\mathcal{W}}\left[\sup_{\left(x,y,t\right)\in B_{n}}\left|\frac{X_{t}^{f,\theta}\left(y\right)}{\Sigma_{t}^{f}}-\frac{X_{t}^{f,\theta}\left(x\right)}{\Sigma_{t}^{f}}\right|\right]\leq2^{-\frac{n}{4}},
\]
as well as 
\[
\mathcal{W}\left(\sup_{\left(x,y,t\right)\in B_{n}}\left|\frac{X_{t}^{f,\theta}\left(y\right)}{\Sigma_{t}^{f}}-\frac{X_{t}^{f,\theta}\left(x\right)}{\Sigma_{t}^{f}}\right|>2^{-\frac{n}{8}}\;\mbox{ i.o. }\right)=0.
\]
\begin{proof}
We only need to prove the first statement, since the second statement
is an immediate consequence of the first one by the Borel-Cantelli
lemma. 

To facilitate the proof, we first make the following observations.
For every $\theta\in\Theta$, $n\geq1$ and $t\in\left[2^{-n^{2}},1\right]$,
we define
\begin{equation}
m_{n}^{\theta}\left(t\right):=\sup\left\{ \left|\bar{\theta}_{s}\left(x\right)-\bar{\theta}_{s}\left(y\right)\right|:\,x,y\in\overline{S\left(O,1\right)},\,\left|x-y\right|\leq2^{-\left(n+1\right)^{2}}2\sqrt{\nu},\,s\in[t,1]\right\} .\label{eq:def of m^theta(t)}
\end{equation}
By (\ref{eq:estimate intrinsic metric non-concentric}), whenever
$\left|x-y\right|\leq2^{-\left(n+1\right)^{2}}2\sqrt{\nu}$ and $s\in[t,1]$, 

\[
d\left(x,s;y,s\right)\leq Cs^{\left(3-2\nu\right)/4}\left|x-y\right|^{1/4}\leq Ct^{\left(3-2\nu\right)/4}2^{-\left(n+1\right)^{2}/4},
\]
and hence 
\[
m_{n}^{\theta}\left(t\right)\leq\omega_{t}^{\theta}\left(Ct^{\left(3-2\nu\right)/4}2^{-\left(n+1\right)^{2}/4}\right)
\]
where $\omega_{t}^{\theta}$ is as defined in (\ref{eq:def of cont modulus}).
Therefore, by (\ref{eq:expectation of modulus of continuity}), it
is easy to check that
\begin{equation}
\begin{split}\mathbb{E}^{\mathcal{W}}\left[m_{n}^{\theta}\left(t\right)\right] & \leq\mathbb{E}^{\mathcal{W}}\left[\omega_{t}^{\theta}\left(Ct^{\left(3-2\nu\right)/4}2^{-\left(n+1\right)^{2}/4}\right)\right]\leq C\sqrt{G\left(t\right)}2^{-n/2}n.\end{split}
\label{eq:an estimate on expectation of max diff up to t}
\end{equation}

Now let us turn our attention to the desired statement. Recall that
$\mathcal{J}:=\left\{ t_{j}:j\geq1\right\} $ are all the jump discontinuities
of $f$. Set $t_{0}\equiv1$. Assume that $J$ and $J^{\prime}$ are
the two integers, $J^{\prime}\leq J$, such that
\[
t_{J^{\prime}-1}>2^{-\left(n-1\right)^{2}}\geq t_{J^{\prime}}>\cdots\cdots>t_{J}\geq2^{-n^{2}}>t_{J+1}.
\]
For fixed $t\in\left[2^{-n^{2}},2^{-\left(n-1\right)^{2}}\right]$,
assume that $K:=K\left(t\right)$ is the unique integer, $J^{\prime}\leq K\leq J$,
such that
\[
t_{J^{\prime}-1}>2^{-\left(n-1\right)^{2}}\geq t_{J^{\prime}}>\cdots>t_{K}\geq t>t_{K+1}>\cdots>t_{J}\geq2^{-n^{2}}>t_{J+1}.
\]
Then, by rewriting it as a telescoping sum and using the triangle
inequality, we have that
\[
\sup_{\left(x,y,t\right)\in B_{n}}\left|\frac{X_{t}^{f,\theta}\left(y\right)-X_{t}^{f,\theta}\left(x\right)}{\Sigma_{t}^{f}}\right|
\]
is no greater than 
\begin{equation}
\begin{split} & \sup_{\left(x,y,t\right)\in B_{n}}\frac{1}{\Sigma_{t}^{f}}\left|X_{t}^{f,\theta}\left(y\right)-X_{t_{K}}^{f,\theta}\left(y\right)-\left(X_{t}^{f,\theta}\left(x\right)-X_{t_{K}}^{f,\theta}\left(x\right)\right)\right|\\
 & \hfill\qquad+\sup_{\left(x,y,t\right)\in B_{n}}\frac{1}{\Sigma_{t_{K}}^{f}}\sum_{i=1}^{K}\left|X_{t_{i}}^{f,\theta}\left(y\right)-X_{t_{i-1}}^{f,\theta}\left(y\right)-\left(X_{t_{i}}^{f,\theta}\left(x\right)-X_{t_{i-1}}^{f,\theta}\left(x\right)\right)\right|\\
\leq & \sup_{\left(x,y,t\right)\in B_{n}}\frac{1}{\Sigma_{t}^{f}}\left|X_{t}^{f,\theta}\left(y\right)-X_{t_{K}}^{f,\theta}\left(y\right)-\left(X_{t}^{f,\theta}\left(x\right)-X_{t_{K}}^{f,\theta}\left(x\right)\right)\right|\\
 & \qquad\hfill+\sum_{j=J^{\prime}}^{J}\frac{1}{\Sigma_{t_{j}}^{f}}\left(\sum_{i=1}^{j}\,\sup_{\left|x-y\right|\leq2^{-\left(n+1\right)^{2}}2\sqrt{\nu}}\left|X_{t_{i}}^{f,\theta}\left(y\right)-X_{t_{i-1}}^{f,\theta}\left(y\right)-\left(X_{t_{i}}^{f,\theta}\left(x\right)-X_{t_{i-1}}^{f,\theta}\left(x\right)\right)\right|\right).
\end{split}
\label{eq:estimate of modulus of continuity for X/Sigma}
\end{equation}
To treat the first term on the right hand side in (\ref{eq:estimate of modulus of continuity for X/Sigma}),
notice that $f$ is continuous and does not change sign on $[t,t_{K}]$
and $\left|f\right|$ is non-decreasing on $[t,t_{K}]$. For every
$s\in[t,1]$, let $m_{n}^{\theta}\left(s\right)$ be as defined in
(\ref{eq:def of m^theta(t)}). Then $m_{n}^{\theta}:\left[t,1\right]\rightarrow\left(0,\infty\right)$
is a non-increasing function on $\left[t,1\right]$. Thus, 
\[
\begin{split} & \sup_{\left(x,y,t\right)\in B_{n}}\frac{1}{\Sigma_{t}^{f}}\left|X_{t}^{f,\theta}\left(y\right)-X_{t_{K}}^{f,\theta}\left(y\right)-\left(X_{t}^{f,\theta}\left(x\right)-X_{t_{K}}^{f,\theta}\left(x\right)\right)\right|\\
\leq & \sup_{\left(x,y,t\right)\in B_{n}}\frac{1}{\Sigma_{t}^{f}}\left|f\left(t\right)\left(\bar{\theta}_{t}\left(y\right)-\bar{\theta}_{t}\left(x\right)\right)-f\left(t_{K}\right)\left(\bar{\theta}_{t_{K}}\left(y\right)-\bar{\theta}_{t_{K}}\left(x\right)\right)-\int_{t_{K}}^{t}\left(\bar{\theta}_{s}\left(y\right)-\bar{\theta}_{s}\left(x\right)\right)df\left(s\right)\right|\\
\leq & \sup_{2^{-n^{2}}\leq t\leq2^{-\left(n-1\right)^{2}}}\frac{1}{\Sigma_{t}^{f}}\left[\left|f\left(t\right)\right|m_{n}^{\theta}\left(t\right)+\left|f\left(t_{K}\right)\right|m_{n}^{\theta}\left(t_{K}\right)-\int_{t_{K}}^{t}m_{n}^{\theta}\left(s\right)d\left|f\left(s\right)\right|\right]\\
\leq & \sup_{2^{-n^{2}}\leq t\leq2^{-\left(n-1\right)^{2}}}\frac{1}{\Sigma_{t}^{f}}\left[\int_{t_{K}}^{t}\left|f\left(s\right)\right|dm_{n}^{\theta}\left(s\right)+2\left|f\left(t_{K}\right)\right|m_{n}^{\theta}\left(t_{K}\right)\right]\\
\leq & \sup_{2^{-n^{2}}\leq t\leq2^{-\left(n-1\right)^{2}}}\left[\int_{t_{K}}^{t}\frac{\left|f\left(s\right)\right|}{\Sigma_{s}^{f}}dm_{n}^{\theta}\left(s\right)+2\frac{\left|f\left(t_{K}\right)\right|}{\Sigma_{t_{K}}^{f}}m_{n}^{\theta}\left(t_{K}\right)\right]\\
\leq & \sum_{j=J^{\prime}+1}^{J}\int_{t_{j-1}}^{t_{j}}\frac{\left|f\left(s\right)\right|}{\Sigma_{s}^{f}}dm_{n}^{\theta}\left(s\right)+\int_{t_{J}}^{2^{-n^{2}}}\frac{\left|f\left(s\right)\right|}{\Sigma_{s}^{f}}dm_{n}^{\theta}\left(s\right)+2\sum_{j=J^{\prime}}^{J}\frac{\left|f\left(t_{j}\right)\right|}{\Sigma_{t_{j}}^{f}}m_{n}^{\theta}\left(t_{j}\right)
\end{split}
\]
Again, for each integral above, $m_{n}^{\theta}$ is non-increasing
and $\left|f\left(s\right)\right|/\Sigma_{s}^{f}$ is non-decreasing
in $s$ within the relevant region. So, taking expectation and using
(\ref{eq:an estimate on expectation of max diff up to t}) yields
that for $j=J^{\prime}+1,\cdots,J$,
\[
\begin{split} & \mathbb{E}^{\mathcal{W}}\left[\int_{t_{j-1}}^{t_{j}}\frac{\left|f\left(s\right)\right|}{\Sigma_{s}^{f}}dm_{n}^{\theta}\left(s\right)\right]\\
\leq & \frac{\left|f\left(t_{j}\right)\right|}{\Sigma_{t_{j}}^{f}}\mathbb{E}^{\mathcal{W}}\left[m_{n}^{\theta}\left(t_{j}\right)\right]+\frac{\left|f\left(t_{j-1}\right)\right|}{\Sigma_{t_{j-1}}^{f}}\mathbb{E}^{\mathcal{W}}\left[m_{n}^{\theta}\left(t_{j-1}\right)\right]-\int_{t_{j-1}}^{t_{j}}\mathbb{E}^{\mathcal{W}}\left[m_{n}^{\theta}\left(s\right)\right]d\frac{\left|f\left(s\right)\right|}{\Sigma_{s}^{f}}\\
\leq & Cn2^{-n/2}\left[\frac{\left|f\left(t_{j}\right)\right|}{\Sigma_{t_{j}}^{f}}\sqrt{G\left(t_{j}\right)}+\frac{\left|f\left(t_{j-1}\right)\right|}{\Sigma_{t_{j-1}}^{f}}\sqrt{G\left(t_{j-1}\right)}-\int_{t_{j-1}}^{t_{j}}\sqrt{G\left(s\right)}d\frac{\left|f\left(s\right)\right|}{\Sigma_{s}^{f}}\right]\\
\leq & Cn2^{-n/2}\left[2\frac{\left|f\left(t_{j-1}\right)\right|}{\Sigma_{t_{j-1}}^{f}}\sqrt{G\left(t_{j-1}\right)}+\int_{t_{j-1}}^{t_{j}}\frac{\left|f\left(s\right)\right|}{\Sigma_{s}^{f}}\frac{1}{2\sqrt{G\left(s\right)}}dG\left(s\right)\right],
\end{split}
\]
and similarly, 
\[
\begin{split}\mathbb{E}^{\mathcal{W}}\left[\int_{t_{J}}^{2^{-n^{2}}}\frac{\left|f\left(s\right)\right|}{\Sigma_{s}^{f}}dm_{n}^{\theta}\left(s\right)\right] & \leq Cn2^{-\frac{n}{2}}\left[2\frac{\left|f\left(t_{J}\right)\right|}{\Sigma_{t_{J}}^{f}}\sqrt{G\left(t_{J}\right)}+\int_{t_{J}}^{2^{-n^{2}}}\frac{\left|f\left(s\right)\right|}{\Sigma_{s}^{f}}\frac{1}{2\sqrt{G\left(s\right)}}dG\left(s\right)\right].\end{split}
\]
Therefore, 
\[
\begin{split} & \mathbb{E}^{\mathcal{W}}\left[\sum_{j=J^{\prime}+1}^{J}\int_{t_{j-1}}^{t_{j}}\frac{\left|f\left(s\right)\right|}{\Sigma_{s}^{f}}dm_{n}^{\theta}\left(s\right)+\int_{t_{J}}^{2^{-n^{2}}}\frac{\left|f\left(s\right)\right|}{\Sigma_{s}^{f}}dm_{n}^{\theta}\left(s\right)+2\sum_{j=J^{\prime}}^{J}\frac{\left|f\left(t_{j}\right)\right|}{\Sigma_{t_{j}}^{f}}m_{n}^{\theta}\left(t_{j}\right)\right]\\
\leq & Cn2^{-\frac{n}{2}}\left[2\sum_{j=J^{\prime}+1}^{J+1}\frac{\left|f\left(t_{j-1}\right)\right|}{\Sigma_{t_{j-1}}^{f}}\sqrt{G\left(t_{j-1}\right)}+\int_{t_{J^{\prime}}}^{2^{-n^{2}}}\frac{\left|f\left(s\right)\right|}{\Sigma_{s}^{f}}\frac{1}{2\sqrt{G\left(s\right)}}dG\left(s\right)+2\sum_{j=J^{\prime}}^{J}\frac{\left|f\left(t_{j}\right)\right|}{\Sigma_{t_{j}}^{f}}\sqrt{G\left(t_{j}\right)}\right]\\
\leq & C_{f}n2^{-\frac{n}{2}}\left[J\left(-\ln t_{J}\right)^{\rho_{f}}+\sqrt{\int_{t_{J^{\prime}}}^{2^{-n^{2}}}\frac{f^{2}\left(s\right)}{\left(\Sigma_{s}^{f}\right)^{2}}dG\left(s\right)}\cdot\sqrt{\int_{t_{J^{\prime}}}^{2^{-n^{2}}}\frac{dG\left(s\right)}{G\left(s\right)}}\right]\\
\leq & C_{f}n2^{-\frac{n}{2}}\left[J\left(-\ln t_{J}\right)^{\rho_{f}}+\sqrt{\Sigma_{2^{-n^{2}}}^{f}}\sqrt{\ln G\left(2^{-n^{2}}\right)}\right]\leq C_{f}n^{3+4\rho_{f}}2^{-\frac{n}{2}},
\end{split}
\]
where we used the conditions \textbf{(a)(b)(c)} imposed on $f\in\mathcal{C}$
as in Definition \ref{def:class C } as well as the simple observation
that
\[
\Sigma_{2^{-n^{2}}}^{f}=\int_{1}^{2^{-n^{2}}}f^{2}\left(s\right)dG\left(s\right)\leq C_{f}\int_{1}^{2^{-n^{2}}}\left(-\ln s\right)^{2\rho_{f}}\frac{dG\left(s\right)}{G\left(s\right)}\leq C_{f}n^{4\rho_{f}+2}.
\]

As for the second term on the right hand side of (\ref{eq:estimate of modulus of continuity for X/Sigma}),
by a similar argument, for each $i=1,\cdots,J$, 
\[
\begin{split} & \mathbb{E}^{\mathcal{W}}\left[\sup_{\left|x-y\right|\leq2^{-\left(n+1\right)^{2}}2\sqrt{\nu}}\,\left|X_{t_{i}}^{f,\theta}\left(y\right)-X_{t_{i-1}}^{f,\theta}\left(y\right)-\left(X_{t_{i}}^{f,\theta}\left(x\right)-X_{t_{i-1}}^{f,\theta}\left(x\right)\right)\right|\right]\\
\leq & \left|f\left(t_{i}\right)\right|\mathbb{E}^{\mathcal{W}}\left[m_{n}^{\theta}\left(t_{i}\right)\right]+\left|f\left(t_{i-1}\right)\right|\mathbb{E}^{\mathcal{W}}\left[m_{n}^{\theta}\left(t_{i-1}\right)\right]-\int_{t_{i-1}}^{t_{i}}\mathbb{E}^{\mathcal{W}}\left[m_{n}^{\theta}\left(s\right)\right]d\left|f\left(s\right)\right|\\
\leq & Cn2^{-\frac{n}{2}}\left[2\left|f\left(t_{i-1}\right)\right|\sqrt{G\left(t_{i-1}\right)}+\int_{t_{i-1}}^{t_{i}}\frac{\left|f\left(s\right)\right|}{2\sqrt{G\left(s\right)}}dG\left(s\right)\right]\\
\leq & Cn2^{-\frac{n}{2}}\left[2\left|f\left(t_{i-1}\right)\right|\sqrt{G\left(t_{i-1}\right)}+\frac{1}{2}\sqrt{\Sigma_{t_{i}}^{f}-\Sigma_{t_{i-1}}^{f}}\cdot\sqrt{\ln G\left(t_{i}\right)-\ln G\left(t_{i-1}\right)}\right],
\end{split}
\]
and hence, 
\[
\begin{split} & \sum_{j=J^{\prime}}^{J}\frac{1}{\Sigma_{t_{j}}^{f}}\sum_{i=1}^{j}\mathbb{E}^{\mathcal{W}}\left[\sup_{\left|x-y\right|\leq2^{-\left(n+1\right)^{2}}2\sqrt{\nu}}\left|X_{t_{i}}^{f,\theta}\left(y\right)-X_{t_{i-1}}^{f,\theta}\left(y\right)-\left(X_{t_{i}}^{f,\theta}\left(x\right)-X_{t_{i-1}}^{f,\theta}\left(x\right)\right)\right|\right]\\
\leq & C_{f}n2^{-\frac{n}{2}}\sum_{j=J^{\prime}}^{J}\frac{1}{\Sigma_{t_{j}}^{f}}\left(j\left(-\ln t_{j-1}\right)^{\rho_{f}}+\sqrt{\Sigma_{t_{j}}^{f}\cdot\ln G\left(t_{j}\right)}\right)\leq C_{f}n^{3+6\rho_{f}}2^{-\frac{n}{2}}.
\end{split}
\]
This completes the proof of Lemma \ref{lem:modulus of X}.
\end{proof}
\newpage{}

\bibliographystyle{plain}
\bibliography{mybib}

\end{document}